\documentclass[10pt, reqno]{amsart}
\usepackage{txfonts}    %{article} was 12pt latex e
\usepackage{amssymb}
\usepackage{eucal}
\usepackage{amsmath}
\usepackage{amsthm}
\usepackage{amscd}
\usepackage[dvips]{color}
\usepackage{multicol}
\usepackage[all]{xy}        %xypic macro for latex2.09
\usepackage{graphicx}
\usepackage{color}
\usepackage{colordvi}
\usepackage{xspace}
\usepackage{tikz}

\usepackage{ifpdf}
\ifpdf
    \usepackage[colorlinks,final,hyperindex]{hyperref}
\else
    \usepackage[colorlinks,final,hyperindex,hypertex]{hyperref}
\fi

\newcommand {\emptycomment}[1]{} %to remove paragraphs

\newcommand{\nc}{\newcommand}
\newcommand{\delete}[1]{}
\nc{\mfootnote}[1]{\footnote{#1}} % Use this to show footnotes
\nc{\todo}[1]{\tred{To do:} #1}

\delete{
\nc{\mlabel}[1]{\label{#1}}  % Use this to suppress names
\nc{\mcite}[1]{\cite{#1}}  % Use this to suppress names
\nc{\mref}[1]{\ref{#1}}  % Use this to suppress names
\nc{\meqref}[1]{\eqref{#1}} % Use this to suppress names
\nc{\mbibitem}[1]{\bibitem{#1}} % Use this to show number
}

\nc{\mlabel}[1]{\label{#1}  % Use the next two lines to show names
{\hfill \hspace{1cm}{\bf{{\ }\hfill(#1)}}}}
\nc{\mcite}[1]{\cite{#1}{{\bf{{\ }(#1)}}}}  % Use this lines to show names
\nc{\mref}[1]{\ref{#1}{{\bf{{\ }(#1)}}}}  % Use this lines to show names
\nc{\meqref}[1]{\eqref{#1}{{\bf{{\ }(#1)}}}} % Use this lines to show names
\nc{\mbibitem}[1]{\bibitem[\bf #1]{#1}} % Use this to show name
\newtheorem{thm}{Theorem}[section]
\newtheorem{lem}[thm]{Lemma}
\newtheorem{cor}[thm]{Corollary}
\newtheorem{pro}[thm]{Proposition}
\theoremstyle{definition}
\newtheorem{defi}[thm]{Definition}
\newtheorem{ex}[thm]{Example}
\newtheorem{rmk}[thm]{Remark}

\nc{\tred}[1]{\textcolor{red}{#1}}
\nc{\tblue}[1]{\textcolor{blue}{#1}}
\nc{\tgreen}[1]{\textcolor{green}{#1}}
\nc{\tpurple}[1]{\textcolor{purple}{#1}}
\nc{\btred}[1]{\textcolor{red}{\bf #1}}
\nc{\btblue}[1]{\textcolor{blue}{\bf #1}}
\nc{\btgreen}[1]{\textcolor{green}{\bf #1}}
\nc{\btpurple}[1]{\textcolor{purple}{\bf #1}}

\nc{\ts}[1]{\textcolor{purple}{Tianshui:#1}}
\nc{\cm}[1]{\textcolor{red}{Chengming:#1}}
\nc{\li}[1]{\textcolor{red}{#1}}
\nc{\lir}[1]{\textcolor{blue}{Li:#1}}

\nc{\twovec}[2]{\left(\begin{array}{c} #1 \\ #2\end{array} \right )}
\nc{\threevec}[3]{\left(\begin{array}{c} #1 \\ #2 \\ #3 \end{array}\right )}
\nc{\twomatrix}[4]{\left(\begin{array}{cc} #1 & #2\\ #3 & #4 \end{array} \right)}
\nc{\threematrix}[9]{{\left(\begin{matrix} #1 & #2 & #3\\ #4 & #5 & #6 \\ #7 & #8 & #9 \end{matrix} \right)}}
\nc{\twodet}[4]{\left|\begin{array}{cc} #1 & #2\\ #3 & #4 \end{array} \right|}

\nc{\rk}{\mathrm{r}}

\newcommand{\ad}{\mathrm{ad}}

\nc{\tforall}{\text{ for all }}

\nc{\svec}[2]{{\tiny\left(\begin{matrix}#1\\
#2\end{matrix}\right)\,}}  % column vector
\nc{\ssvec}[2]{{\tiny\left(\begin{matrix}#1\\
#2\end{matrix}\right)\,}} % subscript column vector

\nc{\typeI}{local cocycle $3$-Lie bialgebra\xspace}
\nc{\typeIs}{local cocycle $3$-Lie bialgebras\xspace}
\nc{\typeII}{double construction $3$-Lie bialgebra\xspace}
\nc{\typeIIs}{double construction $3$-Lie bialgebras\xspace}

\nc{\bia}{{$\mathcal{P}$-bimodule ${\bf k}$-algebra}\xspace}
\nc{\bias}{{$\mathcal{P}$-bimodule ${\bf k}$-algebras}\xspace}

\nc{\rmi}{{\mathrm{I}}}
\nc{\rmii}{{\mathrm{II}}}
\nc{\rmiii}{{\mathrm{III}}}
\nc{\pr}{{\mathrm{pr}}}

\nc{\OT}{constant $\theta$-}
\nc{\T}{$\theta$-}
\nc{\IT}{inverse $\theta$-}

\nc{\asi}{ASI\xspace}
\nc{\qadm}{$Q$-admissible\xspace}
\nc{\aybe}{AYBE\xspace}
\nc{\admset}{\{\pm x\}\cup (-x+K^{\times}) \cup K^{\times} x^{-1}}

\nc{\dualrep}{gives a dual representation\xspace}
\nc{\admt}{admissible to\xspace}

\nc{\opa}{\cdot_A}
\nc{\opb}{\cdot_B}

\nc{\post}{positive type\xspace}
\nc{\negt}{negative type\xspace}
\nc{\invt}{inverse type\xspace}

\nc{\pll}{\beta}
\nc{\plc}{\epsilon}

\nc{\ass}{{\mathit{Ass}}}
\nc{\lie}{{\mathit{Lie}}}
\nc{\comm}{{\mathit{Comm}}}
\nc{\dend}{{\mathit{Dend}}}
\nc{\zinb}{{\mathit{Zinb}}}
\nc{\tdend}{{\mathit{TDend}}}
\nc{\prelie}{{\mathit{preLie}}}
\nc{\postlie}{{\mathit{PostLie}}}
\nc{\quado}{{\mathit{Quad}}}
\nc{\octo}{{\mathit{Octo}}}
\nc{\ldend}{{\mathit{ldend}}}
\nc{\lquad}{{\mathit{LQuad}}}

 \nc{\adec}{\check{;}} \nc{\aop}{\alpha}
\nc{\dftimes}{\widetilde{\otimes}} \nc{\dfl}{\succ} \nc{\dfr}{\prec}
\nc{\dfc}{\circ} \nc{\dfb}{\bullet} \nc{\dft}{\star}
\nc{\dfcf}{{\mathbf k}} \nc{\apr}{\ast} \nc{\spr}{\cdot}
\nc{\twopr}{\circ} \nc{\tspr}{\star} \nc{\sempr}{\ast}
\nc{\disp}[1]{\displaystyle{#1}}
\nc{\bin}[2]{ (_{\stackrel{\scs{#1}}{\scs{#2}}})}  %binomial coeff
\nc{\binc}[2]{ \left (\!\! \begin{array}{c} \scs{#1}\\
    \scs{#2} \end{array}\!\! \right )}  %binomial coeff
\nc{\bincc}[2]{  \left ( {\scs{#1} \atop
    \vspace{-.5cm}\scs{#2}} \right )}  %binomial coeff
\nc{\sarray}[2]{\begin{array}{c}#1 \vspace{.1cm}\\ \hline
    \vspace{-.35cm} \\ #2 \end{array}}
\nc{\bs}{\bar{S}} \nc{\dcup}{\stackrel{\bullet}{\cup}}
\nc{\dbigcup}{\stackrel{\bullet}{\bigcup}} \nc{\etree}{\big |}
\nc{\la}{\longrightarrow} \nc{\fe}{\'{e}} \nc{\rar}{\rightarrow}
\nc{\dar}{\downarrow} \nc{\dap}[1]{\downarrow
\rlap{$\scriptstyle{#1}$}} \nc{\uap}[1]{\uparrow
\rlap{$\scriptstyle{#1}$}} \nc{\defeq}{\stackrel{\rm def}{=}}
\nc{\dis}[1]{\displaystyle{#1}} \nc{\dotcup}{\,
\displaystyle{\bigcup^\bullet}\ } \nc{\sdotcup}{\tiny{
\displaystyle{\bigcup^\bullet}\ }} \nc{\hcm}{\ \hat{,}\ }
\nc{\hcirc}{\hat{\circ}} \nc{\hts}{\hat{\shpr}}
\nc{\lts}{\stackrel{\leftarrow}{\shpr}}
\nc{\rts}{\stackrel{\rightarrow}{\shpr}}
\nc{\lleft}{[}
\nc{\lright}{]}
\nc{\uni}[1]{\tilde{#1}}
\nc{\wor}[1]{\check{#1}}
\nc{\free}[1]{\bar{#1}} \nc{\den}[1]{\check{#1}} \nc{\lrpa}{\wr}
\nc{\curlyl}{\left \{ \begin{array}{c} {} \\ {} \end{array}
    \right .  \!\!\!\!\!\!\!}
\nc{\curlyr}{ \!\!\!\!\!\!\!
    \left . \begin{array}{c} {} \\ {} \end{array}
    \right \} }
\nc{\leaf}{\ell}       % number of leafs
\nc{\longmid}{\left | \begin{array}{c} {} \\ {} \end{array}
    \right . \!\!\!\!\!\!\!}
\nc{\ot}{\otimes} \nc{\sot}{{\scriptstyle{\ot}}}
\nc{\otm}{\overline{\ot}}
\nc{\ora}[1]{\stackrel{#1}{\rar}}
\nc{\ola}[1]{\stackrel{#1}{\la}}%${\Bbb Z}$
\nc{\pltree}{\calt^\pl}
\nc{\epltree}{\calt^{\pl,\NC}}
\nc{\rbpltree}{\calt^r}
\nc{\scs}[1]{\scriptstyle{#1}} \nc{\mrm}[1]{{\rm #1}}
\nc{\dirlim}{\displaystyle{\lim_{\longrightarrow}}\,}
\nc{\invlim}{\displaystyle{\lim_{\longleftarrow}}\,}
\nc{\mvp}{\vspace{0.5cm}} \nc{\svp}{\vspace{2cm}}
\nc{\vp}{\vspace{8cm}} \nc{\proofbegin}{\noindent{\bf Proof: }}
\nc{\proofend}{$\blacksquare$ \vspace{0.5cm}}
\nc{\freerbpl}{{F^{\mathrm RBPL}}}
\nc{\sha}{{\mbox{\cyr X}}}  %used to be \cyr
\nc{\ncsha}{{\mbox{\cyr X}^{\mathrm NC}}} \nc{\ncshao}{{\mbox{\cyr
X}^{\mathrm NC,\,0}}}
\nc{\shpr}{\diamond}    %Shuffle product
\nc{\shprm}{\overline{\diamond}}    %Shuffle product
\nc{\shpro}{\diamond^0} %Shuffle product
\nc{\shprr}{\diamond^r}  %product on controlled trees
\nc{\shpra}{\overline{\diamond}^r}
\nc{\shpru}{\check{\diamond}} \nc{\catpr}{\diamond_l}
\nc{\rcatpr}{\diamond_r} \nc{\lapr}{\diamond_a}
\nc{\sqcupm}{\ot}
\nc{\lepr}{\diamond_e} \nc{\vep}{\varepsilon} \nc{\labs}{\mid\!}
\nc{\rabs}{\!\mid} \nc{\hsha}{\hat{\sha}}
\nc{\lsha}{\stackrel{\leftarrow}{\sha}}
\nc{\rsha}{\stackrel{\rightarrow}{\sha}} \nc{\lc}{\lfloor}
\nc{\rc}{\rfloor}
\nc{\tpr}{\sqcup}
\nc{\nctpr}{\vee}
\nc{\plpr}{\star}
\nc{\rbplpr}{\bar{\plpr}}
\nc{\sqmon}[1]{\langle #1\rangle}
\nc{\forest}{\calf}
\nc{\altx}{\Lambda_X} \nc{\vecT}{\vec{T}} \nc{\onetree}{\bullet}
\nc{\Ao}{\check{A}}
\nc{\seta}{\underline{\Ao}}
\nc{\deltaa}{\overline{\delta}}
\nc{\trho}{\tilde{\rho}}

\nc{\rpr}{\circ}
\nc{\dpr}{{\tiny\diamond}}
\nc{\rprpm}{{\rpr}}

\nc{\mmbox}[1]{\mbox{\ #1\ }} \nc{\ann}{\mrm{ann}}
\nc{\Aut}{\mrm{Aut}} \nc{\can}{\mrm{can}}
\nc{\twoalg}{{two-sided algebra}\xspace}
\nc{\colim}{\mrm{colim}}
\nc{\Cont}{\mrm{Cont}} \nc{\rchar}{\mrm{char}}
\nc{\cok}{\mrm{coker}} \nc{\dtf}{{R-{\rm tf}}} \nc{\dtor}{{R-{\rm
tor}}}

\nc{\depth}{{\mrm d}}
\nc{\Div}{{\mrm Div}} \nc{\End}{\mrm{End}} \nc{\Ext}{\mrm{Ext}}
\nc{\Fil}{\mrm{Fil}} \nc{\Frob}{\mrm{Frob}} \nc{\Gal}{\mrm{Gal}}
\nc{\GL}{\mrm{GL}} \nc{\Hom}{\mrm{Hom}} \nc{\hsr}{\mrm{H}}
\nc{\hpol}{\mrm{HP}} \nc{\id}{\mrm{id}} \nc{\im}{\mrm{im}}
\nc{\incl}{\mrm{incl}} \nc{\length}{\mrm{length}}
\nc{\LR}{\mrm{LR}} \nc{\mchar}{\rm char} \nc{\NC}{\mrm{NC}}
\nc{\mpart}{\mrm{part}} \nc{\pl}{\mrm{PL}}
\nc{\ql}{{\QQ_\ell}} \nc{\qp}{{\QQ_p}}
\nc{\rank}{\mrm{rank}} \nc{\rba}{\rm{RBA }} \nc{\rbas}{\rm{RBAs }}
\nc{\rbpl}{\mrm{RBPL}}
\nc{\rbw}{\rm{RBW }} \nc{\rbws}{\rm{RBWs }} \nc{\rcot}{\mrm{cot}}
\nc{\rest}{\rm{controlled}\xspace}
\nc{\rdef}{\mrm{def}} \nc{\rdiv}{{\rm div}} \nc{\rtf}{{\rm tf}}
\nc{\rtor}{{\rm tor}} \nc{\res}{\mrm{res}} \nc{\SL}{\mrm{SL}}
\nc{\Spec}{\mrm{Spec}} \nc{\tor}{\mrm{tor}} \nc{\Tr}{\mrm{Tr}}
\nc{\mtr}{\mrm{sk}}

\nc{\ab}{\mathbf{Ab}} \nc{\Alg}{\mathbf{Alg}}
\nc{\Algo}{\mathbf{Alg}^0} \nc{\Bax}{\mathbf{Bax}}
\nc{\Baxo}{\mathbf{Bax}^0} \nc{\RB}{\mathbf{RB}}
\nc{\RBo}{\mathbf{RB}^0} \nc{\BRB}{\mathbf{RB}}
\nc{\Dend}{\mathbf{DD}} \nc{\bfk}{{\bf k}} \nc{\bfone}{{\bf 1}}
\nc{\base}[1]{{a_{#1}}} \nc{\detail}{\marginpar{\bf More detail}
    \noindent{\bf Need more detail!}
    \svp}
\nc{\Diff}{\mathbf{Diff}} \nc{\gap}{\marginpar{\bf
Incomplete}\noindent{\bf Incomplete!!}
    \svp}
\nc{\FMod}{\mathbf{FMod}} \nc{\mset}{\mathbf{MSet}}
\nc{\rb}{\mathrm{RB}} \nc{\Int}{\mathbf{Int}}
\nc{\Mon}{\mathbf{Mon}}
\nc{\remarks}{\noindent{\bf Remarks: }}
\nc{\OS}{\mathbf{OS}} %free operated semigroup
\nc{\Rep}{\mathbf{Rep}}
\nc{\Rings}{\mathbf{Rings}} \nc{\Sets}{\mathbf{Sets}}
\nc{\DT}{\mathbf{DT}}

\nc{\BA}{{\mathbb A}} \nc{\CC}{{\mathbb C}} \nc{\DD}{{\mathbb D}}
\nc{\EE}{{\mathbb E}} \nc{\FF}{{\mathbb F}} \nc{\GG}{{\mathbb G}}
\nc{\HH}{{\mathbb H}} \nc{\LL}{{\mathbb L}} \nc{\NN}{{\mathbb N}}
\nc{\QQ}{{\mathbb Q}} \nc{\RR}{{\mathbb R}} \nc{\BS}{{\mathbb{S}}} \nc{\TT}{{\mathbb T}}
\nc{\VV}{{\mathbb V}} \nc{\ZZ}{{\mathbb Z}}

\nc{\calao}{{\mathcal A}} \nc{\cala}{{\mathcal A}}
\nc{\calc}{{\mathcal C}} \nc{\cald}{{\mathcal D}}
\nc{\cale}{{\mathcal E}} \nc{\calf}{{\mathcal F}}
\nc{\calfr}{{{\mathcal F}^{\,r}}} \nc{\calfo}{{\mathcal F}^0}
\nc{\calfro}{{\mathcal F}^{\,r,0}} \nc{\oF}{\overline{F}}
\nc{\calg}{{\mathcal G}} \nc{\calh}{{\mathcal H}}
\nc{\cali}{{\mathcal I}} \nc{\calj}{{\mathcal J}}
\nc{\call}{{\mathcal L}} \nc{\calm}{{\mathcal M}}
\nc{\caln}{{\mathcal N}} \nc{\calo}{{\mathcal O}}
\nc{\calp}{{\mathcal P}} \nc{\calq}{{\mathcal Q}} \nc{\calr}{{\mathcal R}}
\nc{\calt}{{\mathcal T}} \nc{\caltr}{{\mathcal T}^{\,r}}
\nc{\calu}{{\mathcal U}} \nc{\calv}{{\mathcal V}}
\nc{\calw}{{\mathcal W}} \nc{\calx}{{\mathcal X}}
\nc{\CA}{\mathcal{A}}

\nc{\fraka}{{\mathfrak a}} \nc{\frakB}{{\mathfrak B}}
\nc{\frakb}{{\mathfrak b}} \nc{\frakd}{{\mathfrak d}}
\nc{\oD}{\overline{D}}
\nc{\frakF}{{\mathfrak F}} \nc{\frakg}{{\mathfrak g}}
\nc{\frakm}{{\mathfrak m}} \nc{\frakM}{{\mathfrak M}}
\nc{\frakMo}{{\mathfrak M}^0} \nc{\frakp}{{\mathfrak p}}
\nc{\frakS}{{\mathfrak S}} \nc{\frakSo}{{\mathfrak S}^0}
\nc{\fraks}{{\mathfrak s}} \nc{\os}{\overline{\fraks}}
\nc{\frakT}{{\mathfrak T}}
\nc{\oT}{\overline{T}}
\nc{\frakX}{{\mathfrak X}} \nc{\frakXo}{{\mathfrak X}^0}
\nc{\frakx}{{\mathbf x}}
\nc{\frakTx}{\frakT}      %All rooted trees, correspond to \ncsha(X)
\nc{\frakTa}{\frakT^a}    % rooted trees for \ncsha(A)
\nc{\frakTxo}{\frakTx^0}   % rooted trees for \ncshao(X)
\nc{\caltao}{\calt^{a,0}}   % rooted trees for \ncshao(A)
\nc{\ox}{\overline{\frakx}} \nc{\fraky}{{\mathfrak y}}
\nc{\frakz}{{\mathfrak z}} \nc{\oX}{\overline{X}}

\font\cyr=wncyr10

\nc{\al}{\alpha}
\nc{\lam}{\lambda}
\nc{\lr}{\longrightarrow}

\usepackage{ulem} % formulars with underline
\allowdisplaybreaks % page break for formulars
\nc{\lyc}[1]{\textcolor{blue}{Lin Yuanchang: #1}}

\begin{document}

\title{Infinite-dimensional Lie bialgebras via affinization of perm bialgebras and pre-Lie bialgebras}

% authors information

\author{Yuanchang Lin}
\address{Chern Institute of Mathematics \& LPMC, Nankai University, Tianjin 300071 China}
\curraddr{School of Mathematics, North University of China, Taiyuan 030051, China}
\email{linyuanchang@mail.nankai.edu.cn}

\author{Peng Zhou}
\address{Chern Institute of Mathematics \& LPMC, Nankai University, Tianjin 300071 China}
\email{zhoupeng\_math@mail.nankai.edu.cn}

\author{Chengming Bai$^\ast$}
\address{Chern Institute of Mathematics \& LPMC, Nankai University, Tianjin 300071, China }
\email{baicm@nankai.edu.cn}
\thanks{$^\ast$Corresponding author}

\subjclass[2020]{16T10, 16T25, 16W99, 17A30, 17A60, 17B38, 17B60, 17B62, 17B65, 17D25}

\keywords{Lie bialgebra, classical Yang-Baxter equation, perm
algebra, perm bialgebra,  pre-Lie algebra, pre-Lie bialgebra,
$\mathcal{O}$-operator, pre-perm algebra}

%\date{\today}

\begin{abstract}
It is known that the operads of perm algebras and pre-Lie algebras
are the Koszul dual each other and hence there is a Lie algebra
structure on the tensor product of a perm algebra and a pre-Lie
algebra. Conversely, we construct a special perm algebra structure
and a special pre-Lie algebra structure on the vector space of
Laurent polynomials such that the tensor product with a pre-Lie
algebra and a perm algebra being a Lie algebra structure
characterizes the pre-Lie algebra and the perm algebra
respectively. This is called the affinization of a pre-Lie algebra
and a perm algebra respectively. Furthermore we extend such
correspondences to the context of bialgebras, that is, there is a
bialgebra structure for a perm algebra or a pre-Lie algebra which
could be characterized by the fact that its affinization by a
quadratic pre-Lie algebra or a quadratic perm algebra respectively
gives an infinite-dimensional Lie bialgebra. In the case of perm
algebras,  the corresponding bialgebra structure is called a
perm bialgebra, which can be independently characterized by a
Manin triple of perm algebras as well as a matched pair of perm
algebras. The notion of the perm Yang-Baxter equation is
introduced, whose symmetric solutions give rise to perm
bialgebras. There is a correspondence between symmetric solutions
of the perm Yang-Baxter equation in perm algebras and certain
skew-symmetric solutions of the classical Yang-Baxter equation in
the infinite-dimensional Lie algebras induced from the perm
algebras. In the case of pre-Lie algebras, the corresponding
bialgebra structure is a pre-Lie bialgebra which is
well-constructed. The similar correspondences for the related
structures are given.
\end{abstract}

\maketitle

\tableofcontents

%%%%%%%%%%%%%%%%%%%%%%%%%%%%%%%%%%%%%%%%%%%%%%%%%%%%%%%%%%%%%%%%%%%%%%%%%%%%%%%%
%%%%%%%%%%%%%%%%%%%%%%%%%%%%%%%%%%%%%%%%%%%%%%%%%%%%%%%%%%%%%%%%%%%%%%%%%%%%%%%%
%%%%%%%%%%%%%%%%%%%%%%%%%%%%%%%%%%%%%%%%%%%%%%%%%%%%%%%%%%%%%%%%%%%%%%%%%%%%%%%%
\section{Introduction}
In this paper, we construct infinite-dimensional Lie algebras on the tensor products of perm algebras and pre-Lie algebras,
which fits into the Koszul duality interpretation.
Furthermore, we extend such constructions in the context of bialgebras,
that is, we construct Lie bialgebras from the affinization of perm bialgebras by quadratic pre-Lie algebras and from the affinization of pre-Lie bialgebras by quadratic perm algebras respectively.

%%%%%%%%%%%%%%%%%%%%%%%%%%%%%%%%%%%%%%%%%%%%%%%%%%%%%%%%%%%%%%%%%%%%%%%%%%%%%%%%
\subsection{Operadic Koszul duality and affinization}
A perm algebra is an associative algebra with the left-commutative identity \cite{chapoton2001on}.
It can be obtained by a commutative associative algebra with an averaging operator \cite{aguiar2000pre}.
A pre-Lie algebra is a Lie-admissible algebra whose left multiplication operators form a Lie algebra.
Classical examples of pre-Lie algebras could be derived from commutative associative algebras with commuting derivations and $W_n$ is the typical one defined on Laurent polynomial algebra $\mathbf{k}[x_1^{\pm},\cdots,x_n^{\pm}]$ with partial differential operators \cite{burde2006left}.

The operads of perm algebras and pre-Lie algebras are the operadic Koszul dual each other \cite{chapoton2001pre}.
Then there is a Lie algebra structure on the tensor product of a perm algebra and a pre-Lie algebra, as \cite{ginzburg1994koszul,loday2012algebraic} indicate that there is a Lie algebra structure on the tensor product of a $\mathcal{P}$-algebra and a $\mathcal{P}^{!}$-algebra, where $\mathcal{P}$ is a binary
quadratic operad and $\mathcal{P}^{!}$ is its operadic Koszul dual.
Moreover, we observe that, fixing the pre-Lie algebra $W_n$ ($n \geq 2$), the perm algebra structure on $P$ could be resolved if there is a Lie algebra structure on the tenser product of $P$ and $W_n$.

Such processes could be viewed as obtaining Lie algebras via
affinization of perm algebras by pre-Lie algebras, and perm
algebras are characterized by such property. Roughly speaking, the
affinization of a given algebraic structure is to define another
algebraic structure  on the vector space of Laurent polynomials
and obtain a Lie algebra structure
on the tensor product, which could resolve the given algebraic
structure in turn. The Lie affinization $L \otimes \mathbf{k}[t, t^{-1}]$ is a classical instance of the affinization, where $L$ is
a Lie algebra and $\mathbf{k}[t, t^{-1}]$ is the usual commutative
associative algebra. The affinization of Novikov algebras by right
Novikov algebras \cite{balinsky1985poisson, hong2023infinite} is
another example. All these examples fit into the interpretation of
Koszul duality.

Note that the roles of perm algebras and pre-Lie algebras are
symmetric, there is also a Lie algebra on the tensor product of a
pre-Lie algebra and a perm algebra. We show that the tensor
product of a pre-Lie algebra and a special perm algebra on the
vector space of Laurent polynomials being a Lie algebra
characterizes the pre-Lie algebra, as what we have anticipated.
That is, we obtain Lie algebras via affinization of pre-Lie
algebras by perm algebras, which gives an affinization
characterization of pre-Lie algebras.

Dually, we consider the coalgebra versions of the affinization of
a pre-Lie and a perm algebra respectively. The notion of a
completed tensor product is first introduced to serve as the
target space of more general coproducts, leading to the notions of
a completed pre-Lie algebra, a completed  perm coalgebra and a
completed Lie coalgebra, together with special examples on the
vector of Laurent polynomials. Then we obtain a completed Lie
coalgebra structure on the tensor product of a perm coalgebra and
a completed pre-Lie coalgebra. Furthermore, we could resolve the
perm coalgebra structure when fixing the special completed pre-Lie
coalgebra. That is, we obtain the affinization of perm coalgebras
by completed pre-Lie coalgebras, which characterizes the perm
coalgebras. Symmetrically, we show that there is a completed Lie
coalgebra structure on the tensor product of a pre-Lie coalgebra
and a completed perm coalgebra, and obtain an affinization
characterization of pre-Lie coalgebras by completed pre-Lie
algebras.

In this paper, we try to lift the affinization of perm
(co)algebras and pre-Lie (co)algebras to the level of bialgebras,
and show that it is feasible for the affinization of perm
bialgebras and pre-Lie bialgebras to obtain Lie bialgebras.

%%%%%%%%%%%%%%%%%%%%%%%%%%%%%%%%%%%%%%%%%%%%%%%%%%%%%%%%%%%%%%%%%%%%%%%%%%%%%%%%
\subsection{Perm bialgebras and their affinization constructions of Lie
bialgebras} Recently, the notion of a perm bialgebra was first
introduced in the thesis of one of the authors \cite{zhou2023perm}
for Master degree and a bialgebra theory for right perm algebras
was obtained independently in \cite{hou2023extending}. A perm
bialgebra is composed of a perm algebra and a perm coalgebra
satisfying certain compatibility conditions. We show that perm
bialgebras are characterized equivalently by Manin triples of perm
algebras, as well as matched pairs of perm algebras. We also
introduce the notion of the perm Yang-Baxter equation (perm-YBE)
in a perm algebra as analogue of the classical Yang-Baxter
equation (CYBE) in a Lie algebra, whose symmetric solutions give
rise to perm bialgebras. Moreover, we introduce the notions of
$\mathcal{O}$-operators and pre-perm algebras to provide symmetric
solutions of the perm-YBE. The relationships between these notions
are illustrated in the following diagram.
\begin{equation}\label{dia:d1}
        \xymatrix@C=0.9cm@R=0.6cm{
            \txt{pre-perm \\ algebras} \ar[r] & \txt{$\mathcal{O}$-operators of \\ perm algebras} \ar[r] & \txt{solutions of \\ the perm-YBE } \ar[r] & \txt{perm \\ bialgebras} \ar@{<->}[r] & \txt{Manin triples of \\ perm algebras}
        }
\end{equation}

One of our goals is to construct a Lie bialgebra by lifting the
affinization of a perm algebra to the affinization of a perm
bialgebra. Recall that a Lie bialgebra, introduced in
\cite{drinfeld1983hamiltonian}, is composed of a Lie algebra and a
Lie coalgebra satisfying a cocycle condition, and could be
characterized by a Manin triple of Lie algebras
\cite{chari1995guide}.
The key for our goal is to find a suitable algebraic structure for
the perm bialgebra affinization. The Manin triple characterization
of a finite-dimensional perm bialgebra enlightens us that a
quadratic pre-Lie algebra should be employed for the perm
bialgebra affinization. Note that there is a natural coalgebra
structure induced by the nondegenerate bilinear form on a
finite-dimensional quadratic pre-Lie algebra.

Unfortunately, the typical pre-Lie algebras $W_n$ could not be
quadratic pre-Lie algebras. As a first step, we have to overcome
this obstacle. Instead, starting from $W_1$, we get a natural
symplectic Lie algebra structure on the direct sum of $W_1$ and
the dual space of $W_1$, which produces a quadratic pre-Lie
algebra on $\mathbf{k}[t, t^{-1}] \oplus \mathbf{k}[s, s^{-1}]$.
After obtaining the special quadratic pre-Lie algebra on
$\mathbf{k}[t, t^{-1}] \oplus \mathbf{k}[s, s^{-1}]$, we show that
it could also be used for the affinization of perm algebras as
$W_n$ does. Moreover, this quadratic pre-Lie algebra also induces
a natural completed pre-Lie coalgebra used for the affinization of
perm coalgebras as previously shown.

Finally, we show that the tensor product of a perm bialgebra and a
quadratic $\mathbb{Z}$-graded pre-Lie algebra can be endowed with
a completed Lie bialgebra structure. Such property characterizes
the perm bialgebra when provided the aforementioned special
quadratic $\mathbb{Z}$-graded pre-Lie algebra on $\mathbf{k}[t,
t^{-1}] \oplus \mathbf{k}[s, s^{-1}]$. In summary, we obtain the
affinization of a perm bialgebra by a quadratic
$\mathbb{Z}$-graded pre-Lie algebra, lifting the affinization of a
perm algebra to the affinization of a perm bialgebra.

In addition, from symmetric solutions of the perm-YBE in a perm
algebra, we obtain constructions of skew-symmetric solutions of
the CYBE in the induced Lie algebra which is obtained from the
tensor product of the perm algebra and a quadratic
$\mathbb{Z}$-graded pre-Lie algebra. Such constructions could be
viewed as the affinization of solutions of the perm-YBE.

In conclusion, we have the following commutative diagram in view of the aforementioned procedure of perm bialgebra affinization.
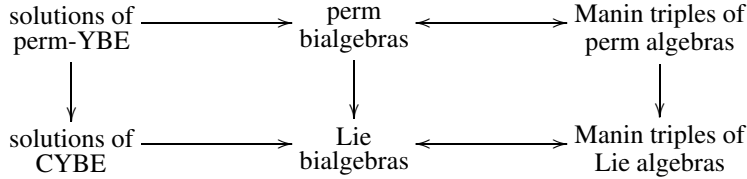
\begin{equation*}
    \xymatrix@C=2cm@R=0.7cm{
        \txt{solutions of \\ perm-YBE} \ar[r] \ar[d] & \txt{perm \\ bialgebras} \ar@{<->}[r] \ar[d] & \txt{Manin triples of \\ perm algebras} \ar[d] \\
        \txt{solutions of \\ CYBE} \ar[r] & \txt{Lie \\ bialgebras} \ar@{<->}[r] & \txt{Manin triples of \\ Lie algebras}
    }
\end{equation*}
Note that relationships illustrated in the below horizontal row
are known in Lie bialgebras \cite{chari1995guide}. Furthermore,
combining with Diagram~(\ref{dia:d1}), there are constructions of
Lie bialgebras from pre-perm algebras.

%%%%%%%%%%%%%%%%%%%%%%%%%%%%%%%%%%%%%%%%%%%%%%%%%%%%%%%%%%%%%%%%%%%%%%%%%%%%%%%%
\subsection{Pre-Lie bialgebras and their affinization constructions of Lie bialgebras}
By the same token, we could lift the affinization of a pre-Lie
algebra to the affinization of a pre-Lie bialgebra. The notion of
a pre-Lie bialgebra was introduced in \cite{bai2008left} and it is
characterized by the para-K\"ahler Lie algebra structure, which
appears in both geometry and mathematical physics
\cite{bai2006further, kupershmidt1994non, kupershmidt1999on,
libermann1954on}. Moreover, there is also an analogue of CYBE in a
Lie algebra, namely, $S$-equation in a pre-Lie algebra, whose
symmetric solutions give rise to pre-Lie bialgebras
\cite{bai2008left}. Parallel to the  affinization of perm
bialgebras, we obtain an infinite-dimensional Lie bialgebra on the
tensor product of a pre-Lie bialgebra and a quadratic perm
algebra, and gain the affinization of pre-Lie bialgebras. Similar
approaches are adopted producing similar results, and the
processes could be summarized into the following commutative
diagram.
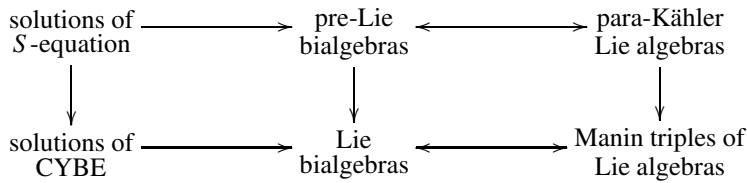
\begin{equation*}
    \xymatrix@C=2cm@R=0.7cm{
        \txt{solutions of \\ $S$-equation} \ar[r] \ar[d] & \txt{pre-Lie \\ bialgebras} \ar@{<->}[r] \ar[d] & \txt{para-K\"ahler \\ Lie algebras} \ar[d] \\
        \txt{solutions of \\ CYBE} \ar[r] & \txt{Lie \\ bialgebras} \ar@{<->}[r] & \txt{Manin triples of \\ Lie algebras}
    }
\end{equation*}

Note that, in both the perm bialgebra affinization and pre-Lie
bialgebra affinization, the algebraic structures employed for the
affinization are ``quadratic'' in the sense that the algebras are
endowed with symmetric or skew-symmetric nondegenerate bilinear
forms satisfying certain conditions such that in the
finite-dimensional case, the algebra structures naturally induced
by the nondegenerate bilinear forms on the dual spaces are
isomorphic to the algebras themselves. This is because the Manin
triple (resp. para-K\"ahler Lie algebra) characterization of perm
(resp. pre-Lie) bialgebras should be compatible with the Manin
triple characterization of Lie bialgebras, i.e., the tensor
product of a Manin triple of perm algebras (resp. a para-K\"ahler
Lie algebra) and the algebraic structure employed for the
affinization should give a Manin triple of Lie algebras naturally,
which requires the algebraic structure employed for the
affinization should have the property that both the algebraic
structures on the underlying vector space and the dual space are
isomorphic. Hence the aforementioned ``quadratic" property is
exactly the needed. Moreover, the  perm bialgebra affinization and
pre-Lie bialgebra affinization would be examples of generalized
phenomena for bialgebra affinization. That is, the affinization of
bialgebras having a Manin triple characterization or a Manin
triple like characterization could be realized by algebras with
such a ``quadratic" property. The affinization of Novikov
bialgebras \cite{hong2023infinite} is another example and more
examples are expected.

Moreover, these examples of bialgebra affinization should also
have an operadic interpretation in terms of a Koszul duality of
properads (or dioperads), that can be applied to the properads of
perm (resp. pre-Lie) bialgebras and quadratic pre-Lie (resp. perm)
algebras. Such a duality has been established for certain
quadratic dioperads and properads
\cite{gan2003koszul,vallette2007koszul}. Lately,
\cite{hong2023infinite} suggests another instance of Koszul
duality for properads. The constructions in our paper give a
strong motivation to develop the duality further and might provide
a very general procedure to construct Lie bialgebras as a guide.

%%%%%%%%%%%%%%%%%%%%%%%%%%%%%%%%%%%%%%%%%%%%%%%%%%%%%%%%%%%%%%%%%%%%%%%%%%%%%%%%
\subsection{Outline of the paper}
This paper is organized as follows.

In Section~\ref{sec:aa}, we recall the notions of perm algebras
and pre-Lie algebras, together with special examples on the vector
space of Laurent polynomials. Then we show in
Proposition~\ref{pro:ppl2l} and Remark~\ref{rmk:plp2l} that there
are Lie algebra structures on the tensor products of perm algebras
and pre-Lie algebras. Furthermore, we obtain an affinization
characterization of a perm algebra and a pre-Lie algebra based on
the special pre-Lie algebra and perm  algebra in
Theorems~\ref{thm:pa2l} and~\ref{thm:pla2l} respectively. Dually,
after introducing the notions of completed perm and pre-Lie
coalgebras, together with special examples on the vector space of
Laurent polynomials, we show in Proposition~\ref{pro:ppl2lc}
(resp. Proposition~\ref{pro:plp2lc}) that there exists a completed
Lie coalgebra structure on the tensor product of a perm (resp.
pre-Lie) coalgebra and a completed pre-Lie (resp. perm) coalgebra,
which could give a characterization of the perm (resp. pre-Lie)
coalgebra by its affinization, see Theorems~\ref{thm:pca2lc}
and~\ref{thm:plca2lc} respectively.

In Section~\ref{sec:ilb}, we show in Theorem~\ref{thm:pb2lb} that
there is a natural completed Lie bialgebra structure on the tensor
product of a perm bialgebra and a quadratic $\mathbb{Z}$-graded
pre-Lie algebra. In the special case when the quadratic
$\mathbb{Z}$-graded pre-Lie algebra takes $\mathbf{k}[t, t^{-1}]
\oplus \mathbf{k}[s, s^{-1}]$ as the underlying vector space, we
obtain an affinization characterization of a perm bialgebra (also
see Theorem~\ref{thm:pb2lb}). We also give the equivalent
characterizations for finite-dimensional perm bialgebras in terms
of matched pairs of perm algebras and Manin triples of perm
algebras (Theorem~\ref{thm:pmmb}). A natural construction of a
Manin triple of Lie algebras from a Manin triple of perm algebras
and a quadratic pre-Lie algebra is presented
(Proposition~\ref{pro:pmt2lmp}). Symmetric solutions of the
perm-YBE in perm algebras are shown to produce a special class of
perm bialgebras (Corollary~\ref{cor:pbc}), and the notions of an
$\mathcal{O}$-operator of a perm algebra and a pre-perm algebra
are also introduced to provide symmetric solutions of the perm-YBE
(Theorem~\ref{thm:o2pybe} and Proposition~\ref{pro:pp2pb}).
Moreover, constructions of skew-symmetric solutions of the CYBE in
the corresponding Lie algebra from the symmetric solutions of the
perm-YBE in a perm algebra are obtained when provided a quadratic
$\mathbb{Z}$-graded pre-Lie algebra
(Proposition~\ref{pro:pybe2cybe}).

In Section~\ref{sec:plb}, we show in Theorem~\ref{thm:plb2lb} that
there is a completed Lie bialgebra structure on the tensor product
of a pre-Lie bialgebra and a quadratic $\mathbb{Z}$-graded perm
algebra. In the special case when the quadratic
$\mathbb{Z}$-graded perm algebra takes $\{f_1 \partial_1 + f_2
\partial_2: f_1, f_2 \in \mathbf{k}[x_1^{\pm}, x_2^{\pm}]\}$ as
the underlying vector space, we obtain a characterization of a
pre-Lie bialgebra by its affinization (also see
Theorem~\ref{thm:plb2lb}). We present a natural construction of a
Manin triple of Lie algebras from a para-K\"ahler Lie algebra and
a quadratic perm algebra (Proposition~\ref{pro:plmt2lmp}).
Constructions of skew-symmetric solutions of the CYBE in the
corresponding Lie algebra from the symmetric solutions of the
$S$-equation in a pre-Lie algebra are also obtained when a
quadratic $\mathbb{Z}$-graded perm algebra is given
(Proposition~\ref{pro:s2cybe}).

In Appendix, the construction of the particular quadratic
$\mathbb{Z}$-graded pre-Lie algebra taking $\mathbf{k}[t, t^{-1}]
\oplus \mathbf{k}[s, s^{-1}]$ as the underlying vector space is
illustrated (Proposition~\ref{pro:plfd2qpl}).

Throughout this paper, all vector spaces and algebras are over a
field $\mathbf{k}$ of characteristic 0. For a vector space $V$,
let
\begin{equation*}
    \sigma: V \otimes V \to V \otimes V, \quad a \otimes b \mapsto b\otimes a,\;\;\;\; a, b \in V,
\end{equation*}
be the flip operator. Let $A$ be a vector space with a binary
operation $\circ$. Define linear maps $L_{\circ}, R_{\circ}:A \to
\End_\mathbf{k}(A)$ respectively by
\begin{equation*}
L_{\circ}(a)b=a\circ b,\;\; R_{\circ}(a)b=b\circ a, \;\;\;a, b\in
A.
\end{equation*}

%%%%%%%%%%%%%%%%%%%%%%%%%%%%%%%%%%%%%%%%%%%%%%%%%%%%%%%%%%%%%%%%%%%%%%%%%%%%%%%%
\section{Affinization of perm (co)algebras and pre-Lie (co)algebras}\label{sec:aa}
There are Lie algebra structures on the tensor products of perm
algebras and pre-Lie algebras, which could give affinization
characterizations of perm algebras and pre-Lie algebras
respectively. Dually, we show that there is a completed Lie
coalgebra structure on the tensor product of a perm (resp.
pre-Lie) coalgebra and a completed pre-Lie (resp. perm) coalgebra,
which could give an affinization characterization of perm (resp. pre-Lie) coalgebras.

%%%%%%%%%%%%%%%%%%%%%%%%%%%%%%%%%%%%%%%%%%%%%%%%%%%%%%%%%%%%%%%%%%%%%%%%%%%%%%%%
\subsection{Affinization of perm algebras and pre-Lie algebras}

We start with notions of perm algebras and pre-Lie algebras.
\begin{defi}\label{defi:perm}
    A \textbf{perm algebra} $(P, \cdot)$ is a vector space $P$ with a binary operation $\cdot: P \otimes P \to P$ such that
    \begin{equation}\label{eq:perm}
        p \cdot (q \cdot r)= (p \cdot q) \cdot r = (q \cdot p) \cdot r , \;\; \forall p, q, r \in P.
    \end{equation}
    A \textbf{$\mathbb{Z}$-graded perm algebra} is a perm algebra $(P, \cdot)$ with a linear decomposition $P = \oplus_{i \in \mathbb{Z}} P_i$ such that each $P_i$ is finite-dimensional and $P_i \cdot P_j \subset P_{i+j}$ for all $i, j \in \mathbb{Z}$.
\end{defi}

The following example is inspired by the free perm algebra on a 2-dimensional vector space (\cite{gnedbaye2017operads}).
\begin{ex}\label{ex:zgp}
    Let $P = \{f_1 \partial_1 + f_2 \partial_2: f_1, f_2 \in \mathbf{k}[x_1^{\pm}, x_2^{\pm}]\}$ and define a binary operation $\cdot: P \otimes P \to P$ by
    \begin{equation*}
        (x_1^{i_1} x_2^{i_2} \partial_s) \cdot (x_1^{j_1} x_2^{j_2} \partial_t) := \delta_{s, 1} x_1^{i_1+j_1+1} x_2^{i_2+j_2} \partial_t + \delta_{s, 2} x_1^{i_1+j_1} x_2^{i_2+j_2+1} \partial_t, \;\; \forall i_1, i_2, j_1, j_2 \in \mathbb{Z}, s, t \in \{1, 2\}.
    \end{equation*}
    Then $(P, \cdot)$ is a $\mathbb{Z}$-graded perm algebra with the linear decomposition $P = \oplus_{i \in \mathbb{Z}} P_i$
    where
    \begin{equation*}
        P_i = \{\sum_{k=1}^2 f_k \partial_k: f_k \text{ is a homogeneous polynomial with } \deg(f_k) = i - 1, k=1,2\}
    \end{equation*}
    for all $i \in \mathbb{Z}$.
\end{ex}

\begin{defi}
    A \textbf{(left) pre-Lie algebra} $(A, \diamond)$ is a vector space $A$ with a binary operation $\diamond: A \otimes A \to A$ such that
    \begin{equation}\label{eq:prelie}
        (a \diamond b) \diamond c - a \diamond (b \diamond c) = (b \diamond a) \diamond c - b \diamond (a \diamond c), \;\; \forall a, b, c \in A.
    \end{equation}
    A \textbf{$\mathbb{Z}$-graded pre-Lie algebra} is a pre-Lie algebra $(A$, $\diamond)$ with a linear decomposition $A = \oplus_{i \in \mathbb{Z}} A_i$ such that each $A_i$ is finite-dimensional and $A_i \diamond A_j \subset A_{i+j}$ for all $i, j \in \mathbb{Z}$.
\end{defi}

Classical examples of pre-Lie algebras are induced by derivations on commutative associative algebras.
\begin{ex}\label{ex:cd2pl}
    {\rm (\cite{burde2006left})}
    Let $\mathfrak{D} = \{\partial_1, \cdots, \partial_n\}$ be a set of commuting derivations on a commutative associative algebra $(U, \cdot)$, and $U\mathfrak{D}=\{\sum_{i=1}^n u_i \partial_i: u_i \in U, \partial_i \in \mathfrak{D}\}$.
    Define a binary operation $\diamond: U\mathfrak{D} \otimes U\mathfrak{D} \to U\mathfrak{D}$ by
    \begin{equation*}
        u\partial_i \diamond v \partial_j = u \cdot (\partial_i(v)) \partial_j, \;\; \forall u, v \in U, \partial_i, \partial_j \in \mathfrak{D}.
    \end{equation*}
    Then $(U\mathfrak{D}, \diamond)$ is a pre-Lie algebra.
    Special cases are $\mathfrak{D} = \{\frac{\partial}{\partial x_1}, \cdots, \frac{\partial}{\partial x_n}\}$ on the Laurent polynomial algebra $\mathbf{k}[x_1^{\pm}, \cdots$, $x_n^{\pm}]$,
    denoted by $(W_n, \diamond)$.
    Evidently, $(W_n, \diamond)$ is a $\mathbb{Z}$-graded pre-Lie algebra with the linear decomposition $W_n = \oplus_{k \in \mathbb{Z}} (W_{n})_k$
    where
    \begin{equation*}
        (W_{n})_k = \{\sum_{i=1}^n u_i \partial_i: u_i \text{ is a homogeneous polynomial with } \deg(u_i) = k+1, 1\leq i\leq
        n\}.
    \end{equation*}
    In particular, when $n=1$, we obtain a pre-Lie algebra $(W_1, \diamond)$ satisfying an additional identity:
    \begin{equation*}
        (a \diamond b) \diamond c = (a \diamond c) \diamond b, \;\; \forall a, b, c \in W_1.
    \end{equation*}
    A pre-Lie algebra satisfying such an identity is called a \textbf{(left) Novikov algebra}.
\end{ex}

The following example is the very vital pre-Lie algebra in our
paper. It is constructed from $(W_1, \diamond)$, and for the sake of coherence,
we put the details into Appendix.

\begin{ex}\label{ex:zgpl}
    Let $A = \mathbf{k}[t, t^{-1}] \oplus \mathbf{k}[s, s^{-1}]$ and define a binary operation $\diamond: A \otimes A \to A$ by
    \begin{equation*}
        t^i \diamond t^j = j t^{i+j-1}, \; t^i \diamond s^j = (2i+j-1)s^{i+j-1}, \; s^j \diamond t^i = i s^{i+j-1},\; s^i \diamond s^j = 0, \;\; \forall i, j \in \mathbb{Z}.
    \end{equation*}
    Then $(A, \diamond)$ is a $\mathbb{Z}$-graded pre-Lie algebra with the linear decomposition $A = \oplus_{i \in \mathbb{Z}}A_i$ where $A_i = \mathbf{k}t^{i+1} \oplus \mathbf{k} s^{i+1}$ for all $i \in \mathbb{Z}$.
\end{ex}

Note that the operads of perm algebras and pre-Lie algebras are
the operadic Koszul dual each other (\cite{chapoton2001pre}) and
hence by \cite[Proposition~7.6.5]{loday2012algebraic}, we have the
following conclusion.
\begin{pro}\label{pro:ppl2l}
    Let $(P, \cdot)$ be a perm algebra and $(A, \diamond)$ be a pre-Lie algebra.
    Define a binary operation $[-, -]: (P \otimes A) \otimes (P \otimes A) \to P \otimes A$ by
    \begin{equation}\label{eq:ppl2l}
        [p_1 \otimes a_1, p_2 \otimes a_2] := p_1 \cdot p_2 \otimes a_1 \diamond a_2 - p_2 \cdot p_1 \otimes a_2 \diamond a_1, \;\; \forall p_1, p_2 \in P, a_1, a_2 \in A.
    \end{equation}
    Then $(P \otimes A, [-, -])$ is a Lie algebra, called \textbf{the induced Lie algebra from $(P, \cdot)$ and $(A, \diamond)$}.
\end{pro}

\begin{rmk}\label{rmk:plp2l}
    Due to the symmetric roles of perm algebras and pre-Lie algebras,
    there is also a Lie algebra structure defined on $A \otimes P$ by
    \begin{equation}\label{eq:plp2l}
        [a_1 \otimes p_1, a_2 \otimes p_2] := a_1 \diamond a_2 \otimes p_1 \cdot p_2 - a_2 \diamond a_1 \otimes p_2 \cdot p_1, \;\; \forall a_1, a_2 \in A, p_1, p_2 \in P,
    \end{equation}
    called \textbf{the induced Lie algebra from $(A, \diamond)$ and $(P, \cdot)$}.
\end{rmk}

\begin{thm}\label{thm:pa2l}
    Let $P$ be a vector space and $\cdot: P \otimes P \to P$ be a binary operation.
    Suppose that $(A, \diamond)$ is the $\mathbb{Z}$-graded pre-Lie algebra $(W_n, \diamond)$ with $n \geq 2$ given in Example~\ref{ex:cd2pl} or the $\mathbb{Z}$-graded pre-Lie algebra given in Example~\ref{ex:zgpl}.
    Then Eq.~\eqref{eq:ppl2l} defines a Lie algebra structure on $P \otimes A$ if and only if $(P, \cdot)$ is a perm algebra.
\end{thm}
\begin{proof}
    Due to Proposition~\ref{pro:ppl2l}, we only need to prove the ``only if'' part.
    Suppose that $(P \otimes A, [-, -])$ is a Lie algebra where $(A, \diamond)$ is the $\mathbb{Z}$-graded pre-Lie algebra given in
    Example~\ref{ex:zgpl}.
    Denote $p \otimes t^i$ by $p t^i$ and $p \otimes s^i$ by  $p s^i$, for all $p\in P$ and $i \in \mathbb{Z}$.
    Let
    $p, q, r \in P$. Then we have    \begin{eqnarray*}
        0 &=& [[qt^i, p t^0], r t^0] + [[pt^0, r t^0], q t^i] + [[r t^0, q t^i], p t^0] \;\;=\;\; i(i-1)(r \cdot (p \cdot q) - p \cdot (r \cdot q)) t^{i-2}, \\
        0 &=& [[p t^0, q t^i], r s^1] + [[q t^i, r s^1], p t^0] + [[r s^1, p t^0], q t^i] \\
        &=& i(i-1)( 2((p \cdot q)\cdot r - p \cdot (q \cdot r)) + (p \cdot (r \cdot q) - r \cdot (p \cdot q))  )   s^{i-1}.
    \end{eqnarray*}
    Hence,
    \begin{equation*}
        2((p \cdot q)\cdot r - p \cdot (q \cdot r)) + (p \cdot (r \cdot q) - r \cdot (p \cdot q)) = r \cdot (p \cdot q) - p \cdot (r \cdot q) = 0,
    \end{equation*}
    which shows that $(P, \cdot)$ is a perm algebra.

    For the case of $W_n$ ($n \geq 2$), denote $p\otimes
    u_i\partial_i$ by $p  u_i\partial_i$ for all $p\in P,
    u_i\partial_i\in W_n$ and $i \in \{1, \cdots, n\}$.
    Let $p,q,r\in P$.
Then the following equations
    \begin{eqnarray*}
            &&0 = [[p \partial_1, q \partial_1], r x_1^2 \partial_1] + [[q \partial_1, r x_1^2 \partial_1], p \partial_1] + [[r x_1^2 \partial_1, p \partial_1], q \partial_1] = 2(q \cdot (p \cdot r)-p \cdot (q \cdot r)) \partial_1, \\
            &&0 = [[p x_2 \partial_1, q x_1 \partial_1], r \partial_2] + [[q x_1 \partial_1, r \partial_2], p x_2 \partial_1] + [[r \partial_2, p x_2 \partial_1], q x_1 \partial_1] = ((r \cdot p) \cdot q - r \cdot (p \cdot q)) \partial_1,
        \end{eqnarray*}
    imply that $(P, \cdot)$ is a perm algebra.
    The proof is completed.
\end{proof}

\begin{rmk}
    Note that $n \neq 1$ in Theorem~\ref{thm:pa2l}.
    In fact, $(A = W_1, \diamond)$ is a Novikov algebra and Eq.~\eqref{eq:ppl2l} defines a Lie algebra structure on $P \otimes A$ if and only if $(P, \cdot)$ is a right Novikov algebra as \cite{balinsky1985poisson} shows.
\end{rmk}

\begin{rmk}
Note that both the $\mathbb{Z}$-graded pre-Lie algebra given in
Example~\ref{ex:zgpl}
    and the $\mathbb{Z}$-graded pre-Lie algebra $(W_n, \diamond)$ with $n \geq 2$ given in Example~\ref{ex:cd2pl} could be used for the affinization of perm
    algebras and hence in this sense, the affinization is usually not
    unique. However, as one would find out later,
    only the former could be used for the affinization of perm bialgebras in our approach.
\end{rmk}

\begin{thm}\label{thm:pla2l}
    Let $A$ be a vector space and $\diamond: A \otimes A \to A$ be a binary operation.
    Suppose that $(P, \cdot)$ is the $\mathbb{Z}$-graded perm algebra given in Example~\ref{ex:zgp}. Then Eq.~\eqref{eq:plp2l} defines a Lie algebra structure on $A \otimes P$ if and only if $(A, \diamond)$ is a pre-Lie algebra.
\end{thm}
\begin{proof}
    The ``if'' part follows from Remark~\ref{rmk:plp2l}.
    Conversely, suppose that $(A \otimes P, [-, -])$ is a Lie
    algebra. Denote $a \otimes x_1^i x_2^j \partial_s$ by $a x_1^i x_2^j \partial_s$ for
    all $a \in A$, $i, j \in \mathbb{Z}$ and $s=1,2$.
    Let  $a, b, c \in A$. Then we have
    \begin{equation*}
        [[a \partial_1, b \partial_1], c\partial_2] + [[b\partial_1, c\partial_2], a\partial_1] + [[c\partial_2, a \partial_1], b \partial_1] =
        0.
    \end{equation*}
    Comparing the coefficients of $x_1^2 \partial_2$ in the expansion, we have
    \begin{equation*}
        (a \diamond b - b \diamond a) \diamond c - a \diamond (b \diamond c)+ b \diamond (a \diamond c) = 0,
    \end{equation*}
    which completes the proof.
\end{proof}

%%%%%%%%%%%%%%%%%%%%%%%%%%%%%%%%%%%%%%%%%%%%%%%%%%%%%%%%%%%%%%%%%%%%%%%%%%%%%%%%
\subsection{Affinization of perm coalgebras and pre-Lie coalgebras}
Next we give the coalgebra versions of
Proposition~\ref{pro:ppl2l}, Theorems~\ref{thm:pa2l} and
\ref{thm:pla2l}. Before proceeding further, we introduce the
notion of a completed tensor product to serve as the target space
of more general coproducts. See \cite{hong2023infinite,
takeuchi1985topological} for more details.

Let $C = \oplus_{i \in \mathbb{Z}} C_i$ and $D = \oplus_{j \in
\mathbb{Z}}D_j$ be $\mathbb{Z}$-graded vector spaces with all
$C_i$ and $D_i$ finite-dimensional. We call the \textbf{completed
tensor product} of $C$ and $D$ to be the vector space
\begin{equation*}
    C \hat{\otimes} D := \prod_{i, j \in \mathbb{Z}} C_i \otimes D_j.
\end{equation*}
Thus a general term of $C\hat{\otimes}D$ is a possibly infinite
sum
\begin{equation*}\label{eq:ssum}
 \sum_{i,j,\alpha} c_{i\alpha}\ot d_{j\alpha},
\end{equation*}
where $i,j\in \ZZ$ and $\alpha$ is in a finite index set (which
might depend on $i$ and $j$).

With these notations, for linear maps $f: C \to C^\prime$ and $g:
D \to D^\prime$, define
\begin{equation*}
    f \hat{\otimes} g: C \hat{\otimes} D \to C^\prime \hat{\otimes} D^\prime, \sum_{i,j,\alpha} c_{i\alpha} \otimes d_{j\alpha} \mapsto \sum_{i,j,\alpha} f(c_{i\alpha}) \otimes g(d_{j\alpha}).
\end{equation*}
Also the flip operator $\sigma$ has its completion
\begin{equation*}
    \hat{\sigma}: C \hat{\otimes} C \to C \hat{\otimes} C, \sum_{i,j,\alpha}c_{i\alpha} \otimes d_{j\alpha} \mapsto \sum_{i,j,\alpha}d_{j\alpha} \otimes c_{i\alpha}.
\end{equation*}
Finally, let $C$ be a vector space and $D = \oplus_{j \in
\mathbb{Z}}D_j$ be a $\mathbb{Z}$-graded vector space with all
$D_i$ finite-dimensional. For all $\sum_{i_1, \cdots, i_l} c_{i_1}
\otimes \cdots \otimes c_{i_l} \in C \otimes \cdots \otimes C$ and
$\sum_{j_1, \cdots, j_l, \alpha} d_{j_1\alpha} \otimes \cdots
\otimes d_{j_l\alpha} \in D \hat{\otimes} \cdots \hat{\otimes} D$,
set
\begin{equation*}
    \sum_{i_1, \cdots, i_l} c_{i_1} \otimes \cdots \otimes c_{i_l} \bullet \sum_{j_1, \cdots, j_l, \alpha} d_{j_1\alpha} \otimes \cdots \otimes d_{j_l\alpha} :=  \sum_{j_1, \cdots, j_l, \alpha}  \sum_{i_1, \cdots, i_l}  (c_{i_1} \otimes d_{j_1\alpha}) \otimes \cdots \otimes (c_{i_l} \otimes d_{j_l\alpha}) \in (C \otimes D) \hat{\otimes} \cdots \hat{\otimes} (C \otimes D).
\end{equation*}

\begin{defi}\label{def:cpc}
    A \textbf{completed perm coalgebra} is a pair $(P, \Delta_P)$ where $P = \oplus_{i \in \mathbb{Z}} P_i$ is a $\mathbb{Z}$-graded vector space with all $P_i$ finite-dimensional and $\Delta_P: P \to P \hat{\otimes} P$ is a linear map such that
    \begin{equation}\label{eq:cpc}
        (\Delta_P \hat{\otimes} \id)\Delta_P = (\id \hat{\otimes} \Delta_P)\Delta_P = (\hat{\sigma} \hat{\otimes} \id) (\Delta_P \hat{\otimes} \id)\Delta_P.
    \end{equation}
    When $P = P_0$, it is simply called a \textbf{perm coalgebra}.
\end{defi}
The notion of a perm coalgebra could be viewed as the duality of the notion of a perm algebra.
That is, if $P$ is a finite-dimensional vector space, $\Delta_P: P \to P \otimes P$ is a linear map and $\Delta_P^*: P^* \otimes P^* \to P$ is the linear dual of $\Delta_P$,
then $(P, \Delta_P)$ is a perm coalgebra if and only if $(P^*, \Delta_P^*)$ is a perm algebra.

\begin{ex}\label{ex:cpc}
    Let $P = \{f_1 \partial_1 + f_2 \partial_2: f_1, f_2 \in \mathbf{k}[x_1^{\pm}, x_2^{\pm}]\}$.
    Define a linear map $\Delta_P: P \to P \hat{\otimes} P$ by
    \begin{eqnarray*}
        \Delta_P(x_1^{m}x_2^{n}\partial_1) &=& \sum_{i_1, i_2 \in \mathbb{Z}} (x_1^{i_1}x_2^{i_2} \partial_1 \otimes x_1^{m-i_1}x_2^{n-i_2+1}\partial_1 - x_1^{i_1}x_2^{i_2} \partial_2 \otimes x_1^{m-i_1+1}x_2^{n-i_2}\partial_1), \\
        \Delta_P(x_1^{m}x_2^{n}\partial_2) &=& \sum_{i_1, i_2 \in \mathbb{Z}} (x_1^{i_1}x_2^{i_2} \partial_1 \otimes x_1^{m-i_1}x_2^{n-i_2+1}\partial_2 - x_1^{i_1}x_2^{i_2} \partial_2 \otimes x_1^{m-i_1+1}x_2^{n-i_2}\partial_2), \;\; \forall m, n \in \mathbb{Z}.
    \end{eqnarray*}
    Then $(P, \Delta_P)$ is a completed perm coalgebra.
\end{ex}

\begin{defi}
    A \textbf{completed pre-Lie coalgebra} is a pair $(A, \Delta_A)$ where $A = \oplus_{i \in \mathbb{Z}} A_i$ is a $\mathbb{Z}$-graded vector space with all $A_i$ finite-dimensional and $\Delta_A: A \to A \hat{\otimes} A$ is a linear map such that
    \begin{equation}\label{eq:cplc}
        (\Delta_A \hat{\otimes} \id)\Delta_A(a) - (\hat{\sigma} \hat{\otimes} \id)(\Delta_A \hat{\otimes} \id)\Delta_A(a) = (\id \hat{\otimes} \Delta_A)\Delta_A(a) - (\hat{\sigma} \hat{\otimes} \id)(\id \hat{\otimes} \Delta_A)\Delta_A(a), \;\; \forall a \in A.
    \end{equation}
    When $A = A_0$, it is simply called a \textbf{pre-Lie coalgebra}.
\end{defi}

\begin{ex}\label{ex:cplc}
    Let $A = \mathbf{k}[t, t^{-1}]\oplus \mathbf{k}[s, s^{-1}]$.
    Define a linear map $\Delta_A: A \to A \hat{\otimes} A$ by
    \begin{equation*}
        \Delta_A(t^i) = \sum_{j \in \mathbb{Z}} ((i+j-1)t^{-j} \otimes s^{i+j-1} + (i-j)s^{-j} \otimes t^{i+j-1} ), \;
        \Delta_A(s^i) = \sum_{j \in \mathbb{Z}} (i+j-1)s^{-j} \otimes s^{i+j-1},
        \;\; \forall i \in \mathbb{Z}.
    \end{equation*}
    Then $(A, \Delta_A)$ is a completed pre-Lie coalgebra.
\end{ex}

\begin{defi}
    A \textbf{completed Lie coalgebra} is a pair $(L, \delta)$ where $L = \oplus_{i \in \mathbb{Z}} L_i$ is a $\mathbb{Z}$-graded vector space with all $L_i$ finite-dimensional and $\delta: L \to L \hat{\otimes} L$ is a linear map such that
    \begin{eqnarray}
        \delta(a) + \hat{\sigma}\delta(a) &=& 0, \label{eq:clca} \\
        (\id \hat{\otimes} \delta)\delta(a) - (\hat{\sigma} \hat{\otimes} \id)(\id \hat{\otimes} \delta)\delta(a) - (\delta \hat{\otimes} \id)\delta(a) &=& 0, \;\; \forall a \in L. \label{eq:clcj}
    \end{eqnarray}
\end{defi}

With the help of the preceding notions, we give the dual version of Proposition~\ref{pro:ppl2l}.
\begin{pro}\label{pro:ppl2lc}
    Let $(P, \Delta_P)$ be a perm coalgebra and $(A, \Delta_A)$ be a completed pre-Lie coalgebra.
    Then $(L:=P \otimes A, \delta)$ is a completed Lie coalgebra, where $\delta: L \to L \hat{\otimes} L$ is defined by
    \begin{equation}\label{eq:pc2lc}
        \delta(p \otimes a) := (\id_{L \hat{\otimes} L} - \hat{\sigma})(\Delta_P(p) \bullet \Delta_A(a)), \;\; \forall p \in P, a \in A.
    \end{equation}
\end{pro}
\begin{proof}
    Obviously, $\delta = -\hat{\sigma}\delta$.
    For all $p \in P$ and $a \in A$, we have
    \begin{eqnarray*}
        &&(\id \hat{\otimes} \delta)\delta(p \otimes a) - (\hat{\sigma} \hat{\otimes} \id)(\id \hat{\otimes} \delta) \delta(p \otimes a) - (\delta \hat{\otimes} \id)\delta(p \otimes a) \\
        &&=  (\id \otimes \Delta_P) \Delta_P(p) \bullet  (\id \hat{\otimes} \Delta_A) \Delta_A(a) - (\id \otimes \sigma)(\id \otimes \Delta_P)\Delta_P(p) \bullet (\id \hat{\otimes} \hat{\sigma})(\id \hat{\otimes} \Delta_A)\Delta_A(a) \\
        &&\quad - (\id \otimes \Delta_P)\sigma \Delta_P(p) \bullet (\id \hat{\otimes} \Delta_A)\hat{\sigma} \Delta_A(a) + (\id \otimes \sigma)(\id \otimes \Delta_P)\sigma \Delta_P(p) \bullet (\id \hat{\otimes} \hat{\sigma})(\id \hat{\otimes} \Delta_A)\hat{\sigma} \Delta_A(a) \\
        %%%%%%%%%%%%%%%%%%%%%%%%%%%%%%%%%%%%%%
        &&\quad - (\sigma \otimes \id)(\id \otimes \Delta_P) \Delta_P(p) \bullet (\hat{\sigma} \hat{\otimes} \id)(\id \hat{\otimes} \Delta_A) \Delta_A(a) \\
        &&\quad + (\sigma \otimes \id)(\id \otimes \sigma)(\id \otimes \Delta_P)\Delta_P(p) \bullet (\hat{\sigma} \hat{\otimes} \id)(\id \hat{\otimes} \hat{\sigma})(\id \hat{\otimes} \Delta_A)\Delta_A(a) \\
        &&\quad + (\sigma \otimes \id)(\id \otimes \Delta_P)\sigma \Delta_P(p) \bullet (\hat{\sigma} \hat{\otimes} \id)(\id \hat{\otimes} \Delta_A)\hat{\sigma} \Delta_A(a) \\
        &&\quad - (\sigma \otimes \id)(\id \otimes \sigma)(\id \otimes \Delta_P)\sigma \Delta_P(p) \bullet (\hat{\sigma} \hat{\otimes} \id)(\id \hat{\otimes} \hat{\sigma})(\id \hat{\otimes} \Delta_A)\hat{\sigma} \Delta_A(a) \\
        %%%%%%%%%%%%%%%%%%%%%%%%%%%%%%%%%%%%%%
        &&\quad - (\Delta_P \otimes \id)\Delta_P(p) \bullet (\Delta_A \hat{\otimes} \id )\Delta_A(a) +  (\sigma \otimes \id)(\Delta_P \otimes \id)\Delta_P(p) \bullet (\hat{\sigma} \otimes \id)(\Delta_A \hat{\otimes} \id )\Delta_A(a) \\
        &&\quad + (\Delta_P \otimes \id)\sigma\Delta_P(p) \bullet (\Delta_A \hat{\otimes} \id)\hat{\sigma}\Delta_A(a)  - (\sigma \otimes \id)(\Delta_P \otimes \id)\sigma\Delta_P(p) \bullet (\hat{\sigma} \hat{\otimes} \id)(\Delta_A \hat{\otimes} \id)\hat{\sigma}\Delta_A(a).
    \end{eqnarray*}
    Note that
    \begin{eqnarray*}
        &&(\id \otimes \sigma)(\id \otimes \Delta_P)\Delta_P = (\sigma \otimes \id)(\id \otimes \Delta_P)\sigma \Delta_P = (\sigma \otimes \id)(\id \otimes \sigma)(\id \otimes \Delta_P)\sigma \Delta_P = (\Delta_P \otimes \id) \sigma\Delta_P, \\
        &&(\id \otimes \Delta_P)\sigma \Delta_P = (\id \otimes \sigma)(\id \otimes \Delta_P)\sigma\Delta_P = (\sigma \otimes \id)(\id \otimes \sigma)(\id \otimes \Delta_P)\Delta_P = (\sigma \otimes \id)(\Delta_P \otimes \id )\sigma \Delta_P, \\
        &&(\hat{\sigma} \hat{\otimes} \id)(\id \hat{\otimes} \Delta_A)\hat{\sigma}\Delta_A = (\id \hat{\otimes} \hat{\sigma})(\Delta_A \hat{\otimes} \id)\Delta_A, \\
        &&(\hat{\sigma} \hat{\otimes} \id)(\id \hat{\otimes} \hat{\sigma})(\id \hat{\otimes} \Delta_A)\hat{\sigma} \Delta_A = (\id \hat{\otimes} \hat{\sigma})(\hat{\sigma} \hat{\otimes} \id)(\Delta_A \hat{\otimes} \id)\Delta_A, \\
        &&(\Delta_A \hat{\otimes} \id)\hat{\sigma}\Delta_A = (\id \hat{\otimes} \hat{\sigma})(\hat{\sigma} \hat{\otimes} \id)(\id \hat{\otimes} \Delta_A) \Delta_A, \\
        &&(\id \hat{\otimes} \Delta_A)\hat{\sigma} \Delta_A = (\hat{\sigma} \hat{\otimes} \id)(\id \hat{\otimes} \hat{\sigma})(\Delta_A \hat{\otimes} \id) \Delta_A, \\
        &&(\id \hat{\otimes} \hat{\sigma})(\id \hat{\otimes} \Delta_A)\hat{\sigma} \Delta_A = (\hat{\sigma} \hat{\otimes} \id)(\id \hat{\otimes} \hat{\sigma})(\hat{\sigma} \hat{\otimes} \id)(\Delta_A \hat{\otimes} \id) \Delta_A, \\
        &&(\hat{\sigma} \hat{\otimes} \id)(\Delta_A \hat{\otimes} \id)\hat{\sigma} \Delta_A = (\hat{\sigma} \hat{\otimes} \id)(\id \hat{\otimes} \hat{\sigma})(\hat{\sigma} \hat{\otimes} \id)(\id \hat{\otimes} \Delta_A)\Delta_A,
    \end{eqnarray*}
    we have
    \begin{eqnarray*}
        &&(\id \hat{\otimes} \delta)\delta(p \otimes a) - (\sigma \hat{\otimes} \id)(\id \hat{\otimes} \delta) \delta(p \otimes a) - (\delta \hat{\otimes} \id)\delta(p \otimes a) \\
        &&= (\Delta_P \otimes \id) \Delta_P(p) \bullet  \\
        &&\quad\quad\quad ( (\id \hat{\otimes} \Delta_A) \Delta_A(a) - (\hat{\sigma} \hat{\otimes} \id)(\id \hat{\otimes} \Delta_A) \Delta_A(a) - (\Delta_A \hat{\otimes} \id )\Delta_A(a) + (\hat{\sigma} \hat{\otimes} \id)(\Delta_A \hat{\otimes} \id )\Delta_A(a) ) \\
        %%%%%%%%%%%%%%%%%%%%%%%%%%%%%%%%%%%%%%
        &&\quad - ( (\Delta_P \otimes \id)\sigma\Delta_P(p) ) \bullet  \\
        &&\quad\quad\quad (\id \hat{\otimes} \hat{\sigma})( (\id \hat{\otimes} \Delta_A)\Delta_A(a) - (\Delta_A \hat{\otimes} \id) \Delta_A(a) + (\hat{\sigma} \hat{\otimes} \id)(\Delta_A \hat{\otimes} \id) \Delta_A(a) - (\hat{\sigma} \hat{\otimes} \id)(\id \hat{\otimes} \Delta_A) \Delta_A(a) ) \\
        %%%%%%%%%%%%%%%%%%%%%%%%%%%%%%%%%%%%%%
        &&\quad - ( (\id \otimes \Delta_P)\sigma \Delta_P(p) ) \bullet  \\
        &&\quad\quad\quad  (\hat{\sigma} \hat{\otimes} \id) (\id \hat{\otimes} \hat{\sigma})( (\Delta_A \hat{\otimes} \id) \Delta_A(a) - (\hat{\sigma} \hat{\otimes} \id)(\Delta_A \hat{\otimes} \id) \Delta_A(a) - (\id \hat{\otimes} \Delta_A)\Delta_A(a) + (\hat{\sigma} \hat{\otimes} \id)(\id \hat{\otimes} \Delta_A)\Delta_A(a) ) \\
        &&= 0.
    \end{eqnarray*}
    Therefore, $(L, \delta)$ is a completed Lie coalgebra.
\end{proof}

Similarly to Proposition~\ref{pro:ppl2l}, the roles of perm
algebras and pre-Lie algebras in Proposition~\ref{pro:ppl2lc} are
somehow symmetric, and we have the following result as what we
have anticipated. The proofs are almost identical, with the major
change being the replacement of a perm coalgebra and a completed
pre-Lie coalgebra by a pre-Lie coalgebra and a completed perm
coalgebra respectively.

\begin{pro}\label{pro:plp2lc}
    Let $(A, \Delta_A)$ be a pre-Lie coalgebra and $(P, \Delta_P)$ be a completed perm coalgebra.
    Then $(L:=A \otimes P, \delta)$ is a completed Lie coalgebra, where $\delta: L \to L \hat{\otimes} L$ is defined by
    \begin{equation}\label{eq:plc2lc}
        \delta(a \otimes p) := (\id_{L \hat{\otimes} L} - \hat{\sigma})(\Delta_A(a) \bullet \Delta_P(p)), \;\; \forall a \in A, p \in P.
    \end{equation}
\end{pro}

The following two theorems are the dual versions of Theorem~\ref{thm:pa2l} and \ref{thm:pla2l} respectively.
\begin{thm}\label{thm:pca2lc}
    Let $P$ be a vector space and $\Delta_P: P \to P \otimes P$ be a linear map.
    Suppose that $(A, \Delta_A)$ is the completed pre-Lie coalgebra given in Example~\ref{ex:cplc}, then Eq.~\eqref{eq:pc2lc} defines a completed Lie coalgebra structure on $P \otimes A$ if and only if $(P, \Delta_P)$ is a perm coalgebra.
\end{thm}
\begin{proof}
    The ``if" part follows from Proposition~\ref{pro:ppl2lc}.
    Conversely, suppose that $(P \otimes A, \delta)$ is a completed Lie coalgebra.
    Then for all $p \in P$ and $i \in \mathbb{Z}$, using the same notations in the proof of Theorem~\ref{thm:pa2l} and the Sweedler notation, we have
    \begin{eqnarray}
        \delta(p t^i) &=& \sum_{j}\sum_{(p)} \big( (i+j-1)p_{(1)} t^{-j} \otimes p_{(2)} s^{i+j-1} + (i-j)p_{(1)} s^{-j} \otimes p_{(2)} t^{i+j-1} \nonumber \\
        &&\quad - (i+j-1)p_{(2)} s^{i+j-1} \otimes p_{(1)} t^{-j} - (i-j) p_{(2)} t^{i+j-1} \otimes p_{(1)} s^{-j} \big), \label{eq:tic}\\
        \delta(p s^i) &=& \sum_{j}\sum_{(p)} (i+j-1)\big(p_{(1)} s^{-j} \otimes p_{(2)} s^{i+j-1} - p_{(2)} s^{i+j-1} \otimes p_{(1)} s^{-j}\big). \label{eq:sic}
    \end{eqnarray}
    Expanding the left hand side of Eq.~\eqref{eq:clcj} into formal series, we have
    \begin{eqnarray*}
        &&(\delta \hat{\otimes} \id)\delta(p s^i) + (\hat{\sigma} \hat{\otimes} \id)(\id \hat{\otimes} \delta) \delta(p s^i) -(\id \hat{\otimes} \delta)\delta(p s^i)  \\
        &&=\sum_{j,k}\sum_{(p)} \big( \\
        &&\quad (i+j-1)(j-k+1)(-p_{(11)}s^{-k} \otimes p_{(12)}s^{-j+k-1} \otimes p_{(2)} s^{i+j-1} + p_{{(12)}}s^{-j+k-1} \otimes p_{(11)}s^{-k} \otimes p_{(2)} s^{i+j-1} ) \\
        &&\quad + (i+j-1)(i+j+k-2)(
            -p_{(21)} s^{-k} \otimes p_{(22)} s^{i+j+k-2} \otimes p_{(1)} s^{-j} + p_{(22)} s^{i+j+k-2} \otimes p_{(21)} s^{-k} \otimes p_{(1)} s^{-j}) \\
        &&\quad + (i+j-1)(j-k+1)(p_{(11)}s^{-k} \otimes p_{(2)} s^{i+j-1} \otimes p_{(12)}s^{-j+k-1} - p_{(12)}s^{-j+k-1} \otimes p_{(2)} s^{i+j-1} \otimes p_{(11)}s^{-k} ) \\
        &&\quad + (i+j-1)(i+j+k-2)(
            -p_{(22)} s^{i+j+k-2} \otimes p_{(1)} s^{-j} \otimes p_{(21)} s^{-k} + p_{(21)} s^{-k} \otimes p_{(1)} s^{-j} \otimes p_{(22)} s^{i+j+k-2}) \\
        &&\quad + (i+j-1)(j-k+1)(-p_{(2)} s^{i+j-1} \otimes p_{(11)}s^{-k} \otimes p_{(12)}s^{-j+k-1} + p_{(2)} s^{i+j-1} \otimes p_{(12)}s^{-j+k-1} \otimes p_{(11)}s^{-k}) \\
        &&\quad + (i+j-1)(i+j+k-2)(
            -p_{(1)} s^{-j} \otimes p_{(21)} s^{-k} \otimes p_{(22)} s^{i+j+k-2} + p_{(1)} s^{-j} \otimes p_{(22)} s^{i+j+k-2} \otimes p_{(21)} s^{-k})
        \big) \\
        &&=0.
    \end{eqnarray*}
    Taking $i=0$ and comparing the coefficients of $s^0 \otimes s^0 \otimes s^{-2}$ above, we have
    \begin{equation}
        \sum_{(p)} p_{(21)} \otimes p_{(1)} \otimes p_{(22)} = \sum_{(p)} p_{(1)} \otimes p_{(21)} \otimes p_{(22)}.
        \label{eq:pf7}
    \end{equation}
    Similarly, taking $i=1$ and comparing the coefficients of $s^{-2} \otimes s^{0} \otimes t^{1}$ in the expansion of
    \begin{equation*}
        (\delta \hat{\otimes} \id)\delta(p t^i) + (\hat{\sigma} \hat{\otimes} \id)(\id \hat{\otimes} \delta) \delta(p t^i) -(\id \hat{\otimes} \delta)\delta(p t^i) = 0,
    \end{equation*}
    we have
    \begin{equation}
        \sum_{(p)} ( p_{(22)} \otimes p_{(21)} \otimes p_{(1)} - p_{(21)} \otimes p_{(1)} \otimes p_{(22)} + p_{(22)} \otimes p_{(1)} \otimes p_{(21)} + p_{(1)} \otimes p_{(21)} \otimes p_{(22)} ) = 2 \sum_{(p)} p_{(2)} \otimes p_{(11)} \otimes p_{(12)}.
        \label{eq:pf8}
    \end{equation}
    Combining Eqs.~\eqref{eq:pf7} and \eqref{eq:pf8}, we show that $(P, \Delta_P)$ is a perm coalgebra.
    This completes the proof.
\end{proof}

\begin{rmk}
    Although $(W_n, \diamond)$ with $n \geq 2$ given in Example~\ref{ex:cd2pl} would not be used for affinization of perm bialgebras in our approach, we still give the corresponding dual version of
the part involving it in
    Theorem~\ref{thm:pa2l} as follows.
    Let $P$ be a vector space and $\Delta_P: P \to P \otimes P$ be a linear map.
    Let $A$ be the underlying vector space of $(W_n, \diamond)$ and define a linear map  $\Delta_A: A \to A \hat{\otimes} A$ by
    \begin{align*}
        \Delta_A(x_1^{i_i}\cdots x_n^{i_n} \partial_s) := \sum_{j_1, \cdots, j_n \in \mathbb{Z}} \sum_{t=1}^n j_t x_1^{i_1-j_1+\delta_{1,t}} \cdots x_n^{i_n-j_n+\delta_{n,t}} \partial_t \otimes x_1^{j_1} \cdots x_n^{j_n} \partial_s,
        \;\; \forall i_1, \cdots, i_n \in \mathbb{Z}, s = 1, \cdots, n.
    \end{align*}
    Then $(A, \Delta_A)$ is a completed pre-Lie coalgebra.
    Moreover, Eq.~\eqref{eq:pc2lc} defines a completed Lie coalgebra structure on $P \otimes A$ if and only if $(P, \Delta_P)$ is a perm coalgebra.
\end{rmk}

\begin{thm}\label{thm:plca2lc}
    Let $A$ be a vector space and $\Delta_A: A \to A \otimes A$ be a linear map.
    Suppose that $(P, \Delta_P)$ is the completed perm algebra given in Example~\ref{ex:cpc},
    then Eq.~\eqref{eq:plc2lc} defines a completed Lie coalgebra structure on $A \otimes P$ if and only if $(A, \Delta_A)$ is a pre-Lie coalgebra.
\end{thm}
\begin{proof}
    The ``if'' part follows from Proposition~\ref{pro:plp2lc}.
    Conversely, suppose that $(A \otimes P, \delta)$ is a completed Lie coalgebra.
    For all $a \in A$ and  $m, n \in \mathbb{Z}$, using the same notations in the proof of
    Theorem~\ref{thm:pla2l}, we have
    \begin{eqnarray}
        \delta(a x_1^{m}x_2^n \partial_1) &=& \sum_{i_1, i_2}\sum_{(a)}( a_{(1)} x_1^{i_1}x_2^{i_2} \partial_1 \otimes a_{(2)} x_1^{m-i_1}x_2^{n-i_2+1}\partial_1 - a_{(1)} x_1^{i_1}x_2^{i_2} \partial_2 \otimes a_{(2)} x_1^{m-i_1+1}x_2^{n-i_2}\partial_1  \nonumber \\
        &&\quad - a_{(2)} x_1^{m-i_1}x_2^{n-i_2+1}\partial_1 \otimes a_{(1)} x_1^{i_1}x_2^{i_2} \partial_1 + a_{(2)} x_1^{m-i_1+1}x_2^{n-i_2}\partial_1 \otimes a_{(1)} x_1^{i_1}x_2^{i_2} \partial_2 ), \label{eq:mpc1}\\
        \delta(a x_1^{m}x_2^n \partial_2) &=& \sum_{i_1, i_2}\sum_{(a)} (
        a_{(1)} x_1^{i_1}x_2^{i_2} \partial_1 \otimes a_{(2)} x_1^{m-i_1}x_2^{n-i_2+1}\partial_2 - a_{(1)} x_1^{i_1}x_2^{i_2} \partial_2 \otimes a_{(2)} x_1^{m-i_1+1}x_2^{n-i_2}\partial_2 \nonumber \\
        &&\quad - a_{(2)} x_1^{m-i_1}x_2^{n-i_2+1}\partial_2 \otimes a_{(1)} x_1^{i_1}x_2^{i_2} \partial_1 + a_{(2)} x_1^{m-i_1+1}x_2^{n-i_2}\partial_2 \otimes a_{(1)} x_1^{i_1}x_2^{i_2} \partial_2). \label{eq:mpc2}
    \end{eqnarray}
    Comparing the coefficients of $x_1^{j_1}x_2^{j_2} \partial_2 \otimes x_1^{i_1-j_1+1}x_2^{i_2-j_2}\partial_2 \otimes x_1^{m-i_1+1}x_2^{n-i_2}\partial_1$ in the expansion of
    \begin{equation*}
        0 = (\delta \hat{\otimes} \id)\delta(a x_1^{m}x_2^n \partial_1) + (\hat{\sigma} \hat{\otimes} \id)(\id \hat{\otimes} \delta) \delta(a x_1^{m}x_2^n \partial_1) -(\id \hat{\otimes} \delta)\delta(a x_1^{m}x_2^n \partial_1),
    \end{equation*}
    we obtain
    \begin{equation*}
        (\Delta_A \otimes \id)\Delta_A(a) - (\sigma \otimes \id)(\Delta_A \otimes \id)\Delta_A(a) + (\sigma \otimes \id)(\id \otimes \Delta_A)\Delta_A(a) - (\id \otimes \Delta_A)\Delta_A(a) = 0,
    \end{equation*}
    which shows that $(A, \Delta_A)$ is a pre-Lie coalgebra.
\end{proof}

%%%%%%%%%%%%%%%%%%%%%%%%%%%%%%%%%%%%%%%%%%%%%%%%%%%%%%%%%%%%%%%%%%%%%%%%%%%%%%%%
\section{Infinite-dimensional Lie bialgebras from perm bialgebras}\label{sec:ilb}
We construct infinite-dimensional Lie bialgebras from pairs of a
perm bialgebra and a quadratic $\mathbb{Z}$-graded pre-Lie algebra, and show that a
perm bialgebra could be characterized by such a construction
provided a quadratic $\mathbb{Z}$-graded pre-Lie algebra on
$\mathbf{k}[t, t^{-1}] \oplus \mathbf{k}[s, s^{-1}]$. We also show
that a finite-dimensional perm bialgebra is characterized by a
Manin triple of perm algebras, as well as a matched pair of perm
algebras. There is a construction of a Manin triple of Lie
algebras from a Manin triple of perm algebras and a quadratic
pre-Lie algebra, interpreting the corresponding construction of
the Lie bialgebra from a perm bialgebra and a quadratic pre-Lie
algebra. The perm Yang-Baxter equation (perm-YBE) in a perm
algebra is introduced, whose symmetric solutions give rise to perm
bialgebras. The notions of an $\mathcal{O}$-operator of a perm
algebra and a pre-perm algebra are introduced to construct
symmetric solutions of the perm-YBE. Moreover, a symmetric
solution of the perm-YBE in a perm algebra gives a skews-symmetric
solution of the CYBE in the corresponding Lie algebra.

%%%%%%%%%%%%%%%%%%%%%%%%%%%%%%%%%%%%%%%%%%%%%%%%%%%%%%%%%%%%%%%%%%%%%%%%%%%%%%%%
\subsection{Lie bialgebras from perm bialgebras and quadratic pre-Lie algebras}
We start with the notion of a perm bialgebra.
\begin{defi}
    A \textbf{perm bialgebra} is a triple $(P, \cdot, \Delta_P)$ where $(P, \cdot)$ is a perm algebra and $(P, \Delta_P)$ is a perm coalgebra such that the following equations hold:
    \begin{eqnarray}
        \Delta_P(p \cdot q) &=& ((L_\cdot-R_\cdot)(p) \otimes \id)\Delta_P(q) + (\id \otimes R_\cdot(q))\Delta_P(p), \label{eq:pb1}\\
        \sigma (R_\cdot(q) \otimes \id)\Delta_P(p) &=& (R_\cdot(p) \otimes \id)\Delta_P(q), \label{eq:pb2}\\
        \Delta_P(p \cdot q) &=& (\id \otimes L_\cdot(p))\Delta_P(q) + ((L_\cdot-R_\cdot)(q) \otimes \id)(\Delta_P(p)-\sigma \Delta_P(p)), \label{eq:pb3}
    \end{eqnarray}
    for all $p, q \in P$.
\end{defi}

\begin{ex}\label{ex:1p}
    Let $P$ be a vector space with a basis $\{e\}$.
    Define a binary operation $\cdot: P \otimes P \to P$ and a linear map $\Delta_P: P \to P \otimes P$ by
    \begin{equation*}
        e \cdot e = e \;\;{\rm and }\;\; \Delta_P(e) = e \otimes e
    \end{equation*}
    respectively.
    Then $(P, \cdot, \Delta_P)$ is a perm bialgebra.
\end{ex}

%%%%%%%%%%%%%%%%%%%%%%%%%%%%%%%%%%%%%%%%%%%%%%%%%%%%%%%%%%%%%%%%%%%%%%%%%%%%%%%%

We now turn to the construction of a Lie bialgebra from a perm
bialgebra, and the quadratic $\mathbb{Z}$-graded pre-Lie algebra
is the last ingredient employed. We follow the notations given in
\cite{hong2023infinite}.
\begin{defi}
    A bilinear form $\omega$ on a $\mathbb{Z}$-graded vector space $V = \oplus_{i \in \mathbb{Z}} V_i$ is called \textbf{graded} if there exists some $m \in \mathbb{Z}$ such that
    \begin{equation*}
        \omega(V_i, V_j) = 0, \;{\rm whenever }\; i + j \neq m.
    \end{equation*}
\end{defi}
Let $\omega$ be a nondegenerate graded bilinear form on a
$\mathbb{Z}$-graded vector space $V = \oplus_{i \in \mathbb{Z}}
V_i$. For all $l \geq 2$,  the bilinear form
$\widetilde{\omega}_l: \overbrace{(V \hat{\otimes} \cdots
\hat{\otimes} V)}^{l-\text{fold}} \otimes \overbrace{(V \otimes
\cdots \otimes V)}^{l-\text{fold}} \to \mathbf{k}$  defined by
\begin{equation*}
    \widetilde{\omega}_l(\sum_{i_1, \cdots, i_l, \alpha} v_{1i_1\alpha} \otimes \cdots \otimes v_{li_l\alpha}, \sum_{j_1, \cdots, j_l}u_{j_1} \otimes \cdots \otimes u_{j_l}) := \sum_{i_1,\cdots, i_k, \alpha} \sum_{j_1, \cdots, j_l} \prod_{m=1}^l \omega(v_{mi_m \alpha}, u_{j_l}),
\end{equation*}
is \textbf{left nondegenerate} in the sense that
$\widetilde{\omega}_l(\sum_{i_1, \cdots, i_l, \alpha}
v_{1i_1\alpha} \otimes \cdots \otimes v_{li_l\alpha}, u_1 \otimes
\cdots \otimes u_l) = 0$ for all homogeneous elements $u_1,
\cdots, u_l \in V$ implies $\sum_{i_1, \cdots, i_l, \alpha}
v_{1i_1\alpha} \otimes \cdots \otimes v_{li_l\alpha} = 0$. For
brevity, we will suppress the index $l$ without ambiguity.

\begin{rmk}
    Let $\omega$ be a nondegenerate graded bilinear form on a $\mathbb{Z}$-graded vector space $V = \oplus_{i \in \mathbb{Z}} V_i$ with all $V_i$ finite-dimensional and $\{e_\lambda\}_{\lambda \in \Lambda}$ be a basis of $V$ consisting of homogeneous elements.
    Then we can always find its graded dual basis $\{f_\lambda\}_{\lambda \in \Lambda}$ with respect to $\omega$ consisting of homogeneous elements,
    that is, $\omega(f_\lambda, e_\eta) = \delta_{\lambda, \eta}$ for all $\lambda, \eta \in \Lambda$.
\end{rmk}

\begin{defi}
    A graded bilinear form $\omega$ on a $\mathbb{Z}$-grade pre-Lie algebra $(A = \oplus_{i \in \mathbb{Z}} A_i, \diamond)$ is called \textbf{invariant} if it satisfies
    \begin{equation*}
        \omega(a_1 \diamond a_2, a_3) = - \omega(a_2, a_1 \diamond a_3 - a_3 \diamond a_1), \;\; \forall a_1, a_2, a_3 \in A.
    \end{equation*}
    A \textbf{quadratic $\mathbb{Z}$-graded pre-Lie algebra $(A, \diamond, \omega)$} is a $\mathbb{Z}$-graded pre-Lie algebra $(A = \oplus_{i \in \mathbb{Z}} A_i, \diamond)$ with a skew-symmetric nondegenerate invariant graded bilinear form $\omega$.
    In particular, when $A = A_0$, it is simply called a \textbf{quadratic pre-Lie algebra}.
\end{defi}

\begin{ex}\label{ex:qgpl}
    Let $(A = \oplus_{i \in \mathbb{Z}} A_i, \diamond)$ be the $\mathbb{Z}$-graded pre-Lie algebra given in Example~\ref{ex:zgpl}.
    Define a bilinear form $\omega$ on $(A, \diamond)$ by
    \begin{equation*}
        \omega(s^i, t^j) = -\omega(t^j, s^i) = \delta_{i+j,0}, \; \omega(t^i, t^j) = 0, \;\; \omega(s^i, s^j) = 0, \;\; \forall i, j \in \mathbb{Z}.
    \end{equation*}
    Then $(A, \diamond, \omega)$ is a quadratic $\mathbb{Z}$-graded pre-Lie algebra.
    Moreover, $\{s^{-i}, -t^{i}\}_{i \in \mathbb{Z}}$ is the dual basis of $\{t^i, s^{-i}\}_{i \in \mathbb{Z}}$, consisting of homogeneous elements.
\end{ex}

\begin{lem}\label{lem:qgpl2cplc}
    Let $(A = \oplus_{i \in \mathbb{Z}} A_i, \diamond, \omega)$ be a quadratic $\mathbb{Z}$-graded pre-Lie algebra.
    Define a linear map $\Delta_A: A \to A \hat{\otimes} A$ by
    \begin{equation}\label{eq:qgpl2cplc}
        \widetilde{\omega}(\Delta_A(a), b \otimes c) = -\omega(a, b \diamond c), \;\; \forall a, b, c \in A.
    \end{equation}
    Then $(A, \Delta_A)$ is a completed pre-Lie coalgebra.
\end{lem}
\begin{proof}
    For all homogeneous element $a_1, a_2, a_3 \in A$, we have
    \begin{eqnarray*}
        &&\widetilde{\omega}((\Delta_A \hat{\otimes} \id)\Delta_A(a) - (\hat{\sigma} \hat{\otimes} \id)(\Delta_A \hat{\otimes} \id)\Delta_A(a) - (\id \hat{\otimes} \Delta_A)\Delta_A(a) + (\hat{\sigma} \hat{\otimes} \id)(\id \hat{\otimes} \Delta_A)\Delta_A(a), a_1 \otimes a_2 \otimes a_3) \\
        &&=\omega(a, (a_1 \diamond a_2) \diamond a_3 - (a_2 \diamond a_1) \diamond a_3 - a_1 \diamond (a_2 \diamond a_3) + a_2 \diamond (a_1 \diamond a_3)) = 0.
    \end{eqnarray*}
    Since $\widetilde{\omega}$ is left nondegenerate, Eq.~\eqref{eq:cplc} holds.
    The proof is completed.
\end{proof}

\begin{rmk}\label{rmk:qcp}
    The completed pre-Lie coalgebra obtained by Lemma~\ref{lem:qgpl2cplc} from the quadratic $\mathbb{Z}$-graded pre-Lie algebra $(A = \oplus_{i \in \mathbb{Z}} A_i, \diamond, \omega)$ in Example~\ref{ex:qgpl} is exactly the one given in Example~\ref{ex:cplc}.
\end{rmk}

The next notion and theorem deal with the completed Lie bialgebras, lifting Proposition~\ref{pro:ppl2l} and Theorem~\ref{thm:pa2l} to the level of bialgebras.
\begin{defi}
    A \textbf{completed Lie bialgebra} is a triple $(L, [-, -], \delta)$ such that
    $(L, [-, -])$ is a Lie algebra,
    $(L, \delta)$ is a completed Lie coalgebra,
    and the following compatibility condition holds:
    \begin{equation}\label{eq:lb}
        \delta([a, b]) = (\ad(a) \hat{\otimes} \id + \id \hat{\otimes} \ad(a))\delta(b) - (\ad(b) \hat{\otimes} \id + \id \hat{\otimes} \ad(b))\delta(a), \;\; \forall a, b \in L,
    \end{equation}
    where $\ad(a)(b) := [a, b]$ for all $a, b \in L$.
\end{defi}

\begin{thm}\label{thm:pb2lb}
    Let $(P, \cdot, \Delta_P)$ be a perm bialgebra, $(A = \oplus_{i \in \mathbb{Z}} A_i, \diamond, \omega)$ be a quadratic $\mathbb{Z}$-graded pre-Lie algebra
  and $(L := P \otimes A, [-, -])$ be the induced Lie algebra from $(P, \cdot)$ and $(A, \diamond)$.
   Let $\Delta_A: A \to A \hat{\otimes} A$ be the linear map defined by Eq.~\eqref{eq:qgpl2cplc}
    and $\delta: L \to L \hat{\otimes} L$ be the linear map defined by Eq.~\eqref{eq:pc2lc}.
    Then $(L, [-, -], \delta)$ is a completed Lie bialgebra.
    Further, if $(A, \diamond, \omega)$ is the quadratic $\mathbb{Z}$-graded pre-Lie algebra given in Example~\ref{ex:qgpl}, then $(L, [-, -], \delta)$ is a completed Lie bialgebra if and only if $(P, \cdot, \Delta_P)$ is a perm bialgebra.
\end{thm}
\begin{proof}
    By Lemma~\ref{lem:qgpl2cplc}, $(A, \Delta_A)$ is a completed pre-Lie coalgebra.
    Then by Proposition~\ref{pro:ppl2lc}, $(L, \delta)$ is a completed Lie coalgebra.
    To show that $(L, [-, -], \delta)$ is a completed Lie bialgebra, it is sufficient to prove Eq.~\eqref{eq:lb} holds.
    Let $p \otimes a, q \otimes b \in P \otimes A$ with $a, b \in A$ homogeneous elements.
    Suppose that
    \begin{equation*}
        \Delta_A(a) = \sum_{i,j,\alpha}a_{1i\alpha} \otimes a_{2j\alpha}, \;\; \Delta_A(b) = \sum_{i,j,\alpha}b_{1i\alpha} \otimes b_{2j\alpha}.
    \end{equation*}
    For all homogeneous elements $c, d \in A$, we have
    \begin{eqnarray*}
        &&\widetilde{\omega}(\sum_{i,j,\alpha}b_{1i\alpha} \otimes a \diamond b_{2j\alpha} - \sum_{i,j,\alpha} b_{1i\alpha} \otimes b_{2j\alpha} \diamond a + \sum_{i,j,\alpha} b \diamond a_{1i\alpha} \otimes a_{2j\alpha}, c \otimes d) \\
        &&= -\omega(b, c \diamond (d \diamond a) - c \diamond (a \diamond d)) + \omega(b, c \diamond (d \diamond a)) - \omega(b, c \diamond (a \diamond d)) = 0.
    \end{eqnarray*}
    Since $\widetilde{\omega}$ is left nondegenerate, we obtain
    \begin{equation*}
        \sum_{i,j,\alpha}b_{1i\alpha} \otimes a \diamond b_{2j\alpha} = \sum_{i,j,\alpha} b_{1i\alpha} \otimes b_{2j\alpha} \diamond a - \sum_{i,j,\alpha} b \diamond a_{1i\alpha} \otimes a_{2j\alpha}.
    \end{equation*}
    Similarly, we have
    \begin{eqnarray*}
        \sum_{i,j,\alpha} a \diamond b_{1i\alpha} \otimes b_{2j\alpha} &=& \sum_{i,j,\alpha} b \diamond a_{1i\alpha} \otimes a_{2j\alpha} + \Delta_A(a \diamond b) - \Delta_A(b \diamond a), \\
        \sum_{i,j,\alpha} b_{1i\alpha} \otimes a \diamond b_{2j\alpha} &=& \sum_{i,j,\alpha} b_{1i\alpha} \otimes b_{2j\alpha} \diamond a - \sum_{i,j,\alpha} b \diamond a_{1i\alpha} \otimes a_{2j\alpha}, \\
        \sum_{i,j,\alpha} a_{1i\alpha} \diamond b \otimes a_{2j\alpha} &=& \sum_{i,j,\alpha} b \diamond a_{1i\alpha} \otimes a_{2j\alpha} \\
        && + \hat{\sigma}(\Delta_A(a \diamond b) - \Delta_A(b \diamond a) + \sum_{i,j,\alpha} b \diamond a_{1i\alpha} \otimes a_{2j\alpha} - \sum_{i,j,\alpha} b_{1i\alpha} \diamond a \otimes b_{2j\alpha}), \\
        \sum_{i,j,\alpha} a_{1i\alpha} \otimes b \diamond a_{2j\alpha} &=&  \hat{\sigma}(\sum_{i,j,\alpha} b \diamond a_{1i\alpha} \otimes a_{2j\alpha} - \sum_{i,j,\alpha} b_{1i\alpha} \otimes b_{2j\alpha} \diamond a), \\
        \sum_{i,j,\alpha} a_{1i\alpha} \otimes a_{2j\alpha} \diamond b &=& \sum_{i,j,\alpha} b \diamond a_{1i\alpha} \otimes a_{2j\alpha} + \Delta_A(a \diamond b) - \Delta_A(b \diamond a) \\
        && + \hat{\sigma}(\sum_{i,j,\alpha} b \diamond a_{1i\alpha} \otimes a_{2j\alpha} - \sum_{i,j,\alpha} b_{1i\alpha} \otimes b_{2j\alpha} \diamond a).
    \end{eqnarray*}
    Using these equations and Eqs.~\eqref{eq:pb1}-\eqref{eq:pb3} together, we have
    \begin{eqnarray*}
        &&(\ad(p \otimes a) \hat{\otimes} \id + \id \hat{\otimes} \ad(p \otimes a))\delta(q \otimes b) - (\ad(q \otimes b) \hat{\otimes} \id + \id \hat{\otimes} \ad(q \otimes b))\delta(p \otimes a) \\
        &&=(\id_{L \hat{\otimes}L} - \hat{\sigma}) \biggl(
            (\sum_{(q)} p \cdot q_{(1)} \otimes q_{(2)}) \bullet (\sum_{i,j,\alpha}a \diamond b_{1i\alpha} \otimes b_{2j\alpha}) - (\sum_{(q)}q_{(1)} \cdot p \otimes q_{(2)}) \bullet (\sum_{i,j,\alpha} b_{1i\alpha} \diamond a \otimes b_{2j\alpha}) \\
        &&\quad + (\sum_{(q)} q_{(1)} \otimes p \cdot  q_{(2)}) \bullet (\sum_{i,j,\alpha} b_{1i\alpha} \otimes a \diamond b_{2j\alpha}) - (\sum_{(q)}q_{(1)} \otimes q_{(2)} \cdot p ) \bullet (\sum_{i,j,\alpha} b_{1i\alpha} \otimes b_{2j\alpha} \diamond a ) \\
        &&\quad - (\sum_{(p)} q \cdot p_{(1)} \otimes p_{(2)}) \bullet (\sum_{i,j,\alpha} b \diamond a_{1i\alpha} \otimes a_{2j\alpha}) + (\sum_{(p)} p_{(1)} \cdot q \otimes p_{(2)}) \bullet (\sum_{i,j,\alpha} a_{1i\alpha} \diamond b \otimes a_{2j\alpha}) \\
        &&\quad - (\sum_{(p)}p_{(1)} \otimes q \cdot p_{(2)}) \bullet (\sum_{i,j,\alpha}a_{1i\alpha} \otimes b \diamond a_{2j\alpha}) + (\sum_{(p)}p_{(1)} \otimes p_{(2)} \cdot q ) \bullet (\sum_{i,j,\alpha} a_{1i\alpha} \otimes a_{2j\alpha} \diamond b )
        \biggr) \\
        &&=(\id_{L \hat{\otimes}L} - \hat{\sigma})\biggl( \\
        &&\quad ( \sum_{(p)}(p_{(1)} \cdot q \otimes p_{(2)} - q \cdot p_{(1)} \otimes p_{(2)} - p_{(2)} \otimes p_{(1)} \cdot q + q \cdot p_{(2)} \otimes p_{(1)} + p_{(1)} \otimes p_{(2)} \cdot q ) \\
        &&\quad\quad -  \sum_{(p)} p_{(2)} \cdot q \otimes p_{(1)} +  \sum_{(q)}(p \cdot q_{(1)} \otimes q_{(2)} - q_{(1)} \otimes p \cdot q_{(2)}) ) \bullet (\sum_{i,j,\alpha} b \diamond a_{1i\alpha} \otimes a_{2j\alpha}) \\
        &&\quad + (-\sum_{(q)}q_{(1)} \cdot p \otimes q_{(2)} + \sum_{(p)}p_{(2)} \otimes p_{(1)} \cdot q) \bullet (\sum_{i,j,\alpha}b_{1i\alpha} \diamond a \otimes b_{2j\alpha}) \\
        &&\quad + (\sum_{(q)}(q_{(1)} \otimes p \cdot q_{(2)} - q_{(1)} \otimes q_{(2)} \cdot p) - \sum_{(p)}(q \cdot p_{(2)} \otimes p_{(1)} - p_{(2)} \cdot q \otimes p_{(1)})) \bullet (\sum_{i,j,\alpha} b_{1i\alpha} \otimes b_{2j\alpha} \diamond a) \\
        &&\quad + (\sum_{(q)}p \cdot q_{(1)} \otimes q_{(2)} - \sum_{(p)}p_{(2)} \otimes p_{(1)} \cdot q + \sum_{(p)}p_{(1)} \otimes p_{(2)} \cdot q) \bullet \Delta_A(a \diamond b) \\
        &&\quad - (\sum_{(q)}p \cdot q_{(1)} \otimes q_{(2)} - \sum_{(p)}p_{(2)} \otimes p_{(1)} \cdot q + \sum_{(p)}p_{(1)} \otimes p_{(2)} \cdot q) \bullet \Delta_A(b \diamond a)
        \biggr) \\
        &&= (\id_{L \hat{\otimes}L} - \hat{\sigma})( (\Delta_P(p \cdot q) - \Delta_P(q \cdot p)) \bullet (\sum_{i,j,\alpha} b_{1i\alpha} \otimes b_{2j\alpha} \diamond a) + (\Delta_P(q \cdot p) - \Delta_P(p \cdot q)) \bullet \Delta_A(b \diamond a) ) \\
        &&\quad + \delta([p \otimes a, q \otimes b]) \\
        &&= (\id_{L \hat{\otimes}L} - \hat{\sigma})( (\Delta_P(p \cdot q) - \Delta_P(q \cdot p)) \bullet (\sum_{i,j,\alpha}(b_{1i\alpha} \otimes b_{2j\alpha} \diamond a) - \Delta_A(b \diamond a) ) ) + \delta([p \otimes a, q \otimes b]).
    \end{eqnarray*}
    Note that
    \begin{eqnarray*}
        \sigma(\Delta_P(p \cdot q) - \Delta_P(q \cdot p)) &=& \Delta_P(p \cdot q) - \Delta_P(q \cdot p),\\
        \hat{\sigma}(\sum_{i,j,\alpha}b_{1i\alpha} \otimes b_{2j\alpha} \diamond a - \Delta_A(b \diamond a) )&=& \sum_{i,j,\alpha}b_{1i\alpha} \otimes b_{2j\alpha} \diamond a - \Delta_A(b \diamond a),
    \end{eqnarray*}
    we show that
    \begin{equation*}
        (\ad(p \otimes a) \hat{\otimes} \id + \id \hat{\otimes} \ad(p \otimes a))\delta(q \otimes b) - (\ad(q \otimes b) \hat{\otimes} \id + \id \hat{\otimes} \ad(q \otimes b))\delta(p \otimes a) = \delta([p \otimes a, q \otimes b]).
    \end{equation*}
    Conversely, suppose that $(L, [-, -], \delta)$ is a completed Lie bialgebra and $(A, \diamond, \omega)$ is the quadratic $\mathbb{Z}$-graded pre-Lie algebra given in Example~\ref{ex:qgpl}.
    Then $(P, \cdot)$ is a perm algebra by Theorem~\ref{thm:pa2l} and $(P, \Delta_P)$ is a perm coalgebra by Theorem~\ref{thm:pca2lc} and Remark~\ref{rmk:qcp}.
    Thus it is sufficient to show that Eqs.~\eqref{eq:pb1}-\eqref{eq:pb3} hold.
    Note that in this case, $\delta$ is defined by Eqs.~\eqref{eq:tic} and \eqref{eq:sic}.
    Taking $i=j=0$ and comparing the coefficients of $s^{0} \otimes s^{-2}$ in the expansion of
    \begin{equation}
        \delta([p t^i, q s^j]) = (\ad(p t^i) \hat{\otimes} \id + \id \hat{\otimes} \ad(p t^i))\delta(q s^j) - (\ad(q s^j) \hat{\otimes} \id + \id \hat{\otimes} \ad(q s^j))\delta(p t^i), \label{eq:pf50}
    \end{equation}
    we have
    \begin{eqnarray}
        &&\Delta_P(p \cdot q) = (\id \otimes L_\cdot(p)) \Delta_P(q) + ( (L_\cdot-R_\cdot)(q) \otimes \id) \Delta_P(p) - ((L_\cdot-R_\cdot)(q) \otimes \id) \sigma \Delta_P(p). \label{eq:pf9}
    \end{eqnarray}
    Taking $i=-1,j=1$ and comparing the coefficients of $s^0 \otimes s^{-2}$ in the expansion of Eq.~\eqref{eq:pf50}, we have
    \begin{equation}
        2 \Delta_P(p \cdot q) = \Delta_P(q \cdot p) + ((L_\cdot-R_\cdot)(a) \otimes \id) \Delta_P(q) + ((L_\cdot-R_\cdot)(p) \otimes \id) \sigma\Delta_P(q) +  (2 \id \otimes R_\cdot(q) - \id \otimes L_\cdot(q)) \Delta_P(p).  \label{eq:pf14}
    \end{equation}
    Similarly, taking $i=0,j=2$ and comparing the coefficients of $s^{1} \otimes t^{-1}$ in the expansion of
    \begin{equation*}
        0 = \delta([p t^i, q t^j]) - (\ad(p t^i) \hat{\otimes} \id + \id \hat{\otimes} \ad(p t^i))\delta(q t^j) + (\ad(q t^j) \hat{\otimes} \id + \id \hat{\otimes} \ad(q t^j))\delta(p t^i),
    \end{equation*}
    we have
    \begin{eqnarray}
        2\Delta_P(p \cdot q) &=& \sigma\Delta_P(p \cdot q) + 2 (L_\cdot(p) \otimes \id) \Delta_P(q)  - (L_\cdot(p) \otimes \id) \sigma \Delta_P(q) + (\id \otimes L_\cdot(q))\Delta_P(p) \nonumber\\
        &&+  (\id \otimes R_\cdot(q))\Delta_P(p)  - (\id \otimes L_\cdot(q))\sigma\Delta_P(p) - (\id \otimes R_\cdot(q))\sigma\Delta_P(p). \label{eq:pf15}
    \end{eqnarray}
    Replacing $\Delta_P(q \cdot p)$ in Eq.~\eqref{eq:pf14} by Eq.~\eqref{eq:pf9}, we obtain
    \begin{equation}
        \Delta_P(p \cdot q) = ((L_\cdot-R_\cdot)(p) \otimes \id) \Delta_P(q) + (\id \otimes R_\cdot(q)) \Delta_P(p). \label{eq:pf10}
    \end{equation}
    Replacing $\Delta_P(p \cdot q)$ in the right hand side of Eq.~\eqref{eq:pf15} by Eq.~\eqref{eq:pf9}, we obtain
    \begin{equation}
        \Delta_P(p \cdot q) = (L_\cdot(p) \otimes \id)\Delta_P(p) - (\id \otimes R_\cdot(q)) \sigma \Delta_P(p)+ (\id \otimes R_\cdot(q))\Delta_P(p). \label{eq:pf11}
    \end{equation}
    Eqs.~\eqref{eq:pf10} and \eqref{eq:pf11} show that
    \begin{equation}
        (R_\cdot(p) \otimes \id)\Delta_P(q) = (\id \otimes R_\cdot(q))\sigma \Delta_P(p). \label{eq:pf12}
    \end{equation}
    Now Eqs.~\eqref{eq:pf9}, \eqref{eq:pf10} and \eqref{eq:pf12} show that $(P, \cdot, \Delta_P)$ is a perm bialgebra.
\end{proof}

\begin{ex}
    Let $(P, \cdot, \Delta_P)$ be the perm bialgebra given in Example~\ref{ex:1p} and $(A, \diamond, \omega)$ be the quadratic $\mathbb{Z}$-graded pre-Lie algebra given in Example~\ref{ex:qgpl}.
    Then by Theorem~\ref{thm:pb2lb}, there is a completed Lie bialgebra $(L:=P \otimes A,  [-,-], \delta)$ defined by
    \begin{eqnarray*}
        [et^i, e t^j] = (j-i) e t^{i+j-1}, \; [e t^i, e s^j] = -[e s^j, e t^i] = (i+j-1)e s^{i+j-1}, \;  [e s^i, e s^j] = 0, \\
        \delta(e t^i) = \sum_{k \in \mathbb{Z}} i( e s^{-k} \otimes e t^{i+k-1} - e t^{i+k-1} \otimes e s^{-k}), \; \delta(e s^i) = \sum_{k \in \mathbb{Z}} (i+2k-1) e s^{-k} \otimes e s^{i+k-1}, \;\; \forall i, j \in \mathbb{Z}.
    \end{eqnarray*}
\end{ex}

%%%%%%%%%%%%%%%%%%%%%%%%%%%%%%%%%%%%%%%%%%%%%%%%%%%%%%%%%%%%%%%%%%%%%%%%%%%%%%%%
\subsection{Characterizations of perm bialgebras}
We now turn to the equivalent characterizations of a perm bialgebra
in terms of a Manin triple of perm algebras and a matched pair of
perm algebras.
In the rest of this section, we assume that all perm algebras and
their representations are finite-dimensional, although some
results still hold in the infinite-dimensional case.
\begin{defi}\label{defi:prep}
    A \textbf{representation} of a perm algebra $(P, \cdot)$ is a triple $(V, l, r)$, where $V$ is a vector space and $l, r: P \to \End_\mathbf{k}(V)$ are linear maps satisfying
    \begin{eqnarray}
        && l(p_1 \cdot p_2) = l(p_1) l(p_2) = l(p_2) l(p_1), \label{eq:prep1}\\
        && r(p_1 \cdot p_2) = r(p_2) r(p_1) = r(p_2) l(p_1) = l(p_1) r(p_2), \;\; \forall p_1, p_2 \in P. \label{eq:prep2}
    \end{eqnarray}
\end{defi}
Obviously $(P, L_\cdot, R_\cdot)$ is a representation of $(P, \cdot)$, called the \textbf{adjoint representation} of $(P, \cdot)$.

\begin{pro}
    Let $(P, \cdot)$ be a perm algebra.
    Let $V$ be a vector space and $l, r: P \to \End_\mathbf{k}(V)$ be linear maps.
    Define a binary operation, still denoted by $\cdot$, on $P \oplus V$ by
    \begin{equation*}
        (p_1 + v_1) \cdot (p_2 + v_2) := p_1 \cdot p_2 + (l(p_1) v_2 + r(p_2) v_1), \;\; \forall p_1, p_2 \in P, \; v_1, v_2 \in V.
    \end{equation*}
    Then $(V, l, r)$ is a representation of $(P, \cdot)$ if and only if $(P \oplus V, \cdot)$ is a perm algebra.
    In such a case, we call $(P \oplus V, \cdot)$ the \textbf{semi-direct product perm algebra by $(P, \cdot)$ and $(V, l, r)$} and denote it by $(P \ltimes_{l, r} V, \cdot)$.
\end{pro}
\begin{proof}
    It can be proved directly by Definitions~\ref{defi:perm} and \ref{defi:prep} or as a direct consequence of
    Theorem~\ref{thm:mc} in the special case that $P_2$ is a perm
    algebra with trivial products.
\end{proof}

Denote the standard pairing between the dual space $V^*$ and $V$ by
\begin{equation*}
    \langle \ ,\ \rangle: V^* \otimes V \to \mathbf{k}, \;\; \langle v^*, v \rangle := v^*(v), \;\; \forall v \in V, v^* \in V^*.
\end{equation*}
For a linear map $\varphi: V \to W$, the transpose map $\varphi^*: W^* \to V^*$ is defined by
\begin{equation*}
    \langle \varphi^*(w^*), v \rangle := \langle w^*, \varphi(v)\rangle, \;\; \forall v \in V, w^* \in W^*.
\end{equation*}
Evidently, $(\varphi \psi)^* = \psi^* \varphi^*$, where $\psi: U \to V$ and $\varphi: V \to W$ are linear maps.

Let $(P, \cdot)$ be a perm algebra and $V$ be a vector space.
For a linear map $\rho: P \to \End_\mathbf{k}(V)$, the linear map $\rho^*: P \to \End_\mathbf{k}(V^*)$ is defined by
\begin{equation*}
    \langle \rho^*(p)v^*, v\rangle := \langle v^*, \rho(p) v\rangle, \;\; \forall p \in P, v \in V, v^* \in V^*.
\end{equation*}
Clearly, $\rho^*(p) = \rho(p)^*$ for all $p \in P$.

\begin{pro}
    Let $(P, \cdot)$ be a perm algebra and $(V, l, r)$ be a representation of $(P, \cdot)$.
    Then $(V^*, l^*, l^* - r^*)$ is a representation of $(P, \cdot)$.
    In particular, $(P^*, L_\cdot^*, L_\cdot^* - R_\cdot^*)$ is a representation of $(P, \cdot)$.
\end{pro}
\begin{proof}
    It follows from a straightforward verification.
\end{proof}

\begin{thm}\label{thm:mc}
    Let $(P_1, \cdot)$ and $(P_2, \circ)$ be perm algebras.
    Suppose that $(P_2, l_{P_1}, r_{P_1})$ and $(P_1, l_{P_2}, r_{P_2})$ are representations of $(P_1, \cdot)$ and $(P_2, \circ)$ respectively satisfying
    \begin{eqnarray}
        l_{P_1}(p_1) ( p_2 \circ p_2^\prime) &=& (l_{P_1}(p_1) p_2) \circ p_2^\prime + l_{P_1}(r_{P_2}(p_2)p_1)p_2^\prime, \label{eq:pmp1}\\
        r_{P_1}(p_1)(p_2 \circ p_2^\prime) &=& p_2 \circ (r_{P_1}(p_1)p_2^\prime) + r_{P_1}(l_{P_2}(p_2^\prime)p_1)p_2, \label{eq:pmp2}\\
        l_{P_2}(p_2) ( p_1 \cdot p_1^\prime) &=& (l_{P_2}(p_2) p_1) \cdot p_1^\prime + l_{P_2}(r_{P_1}(p_1)p_2)p_1^\prime, \label{eq:pmp3}\\
        r_{P_2}(p_2)(p_1 \cdot p_1^\prime) &=& p_1 \cdot (r_{P_2}(p_2)p_1^\prime) + r_{P_2}(l_{P_1}(p_1^\prime)p_2)p_1, \label{eq:pmp4}\\
        (r_{P_1}(p_1)p_2) \circ p_2^\prime + l_{P_1}(l_{P_2}(p_2)p_1) p_2^\prime &=& p_2 \circ (l_{P_1}(p_1)p_2^\prime) + r_{P_1}(r_{P_2}(p_2^\prime)p_1) p_2, \label{eq:pmp5}\\
        (r_{P_2}(p_2)p_1) \cdot p_1^\prime + l_{P_2}(l_{P_1}(p_1)p_2) p_1^\prime &=& p_1 \cdot (l_{P_2}(p_2)p_1^\prime) + r_{P_2}(r_{P_1}(p_1^\prime)p_2) p_1, \label{eq:pmp6}\\
        (l_{P_1}(p_1) p_2) \circ p_2^\prime + l_{P_1}(r_{P_2}(p_2)p_1)p_2^\prime &=& (r_{P_1}(p_1)p_2) \circ p_2^\prime + l_{P_1}(l_{P_2}(p_2)p_1)p_2^\prime, \label{eq:pmp7}\\
        (l_{P_2}(p_2) p_1) \cdot p_1^\prime + l_{P_2}(r_{P_1}(p_1)p_2)p_1^\prime &=& (r_{P_2}(p_2)p_1) \cdot p_1^\prime + l_{P_2}(l_{P_1}(p_1)p_2)p_1^\prime, \label{eq:pmp8}\\
        r_{P_1}(p_1)(p_2 \circ p_2^\prime) &=& r_{P_1}(p_1)(p_2^\prime \circ p_2), \label{eq:pmp9}\\
        r_{P_2}(p_2)(p_1 \cdot p_1^\prime) &=& r_{P_2}(p_2)(p_1^\prime \cdot p_1), \label{eq:pmp10}
    \end{eqnarray}
    for all $p_1, p_1^\prime \in P_1, p_2, p_2^\prime \in P_2$.
    Then there is a perm algebra structure $\star$ on the direct sum $P_1 \oplus P_2$ of the underlying vector spaces of $(P_1, \cdot)$ and $(P_2, \circ)$ given by
    \begin{equation}\label{eq:mp2m}
        (p_1 + p_2) \star (p_1^\prime + p_2^\prime) := (p_1 \cdot p_1^\prime + l_{P_2}(p_2) p_1^\prime + r_{P_2}(p_2^\prime)p_1) + (p_2 \circ p_2^\prime + l_{P_1}(p_1) p_2^\prime + r_{P_1}(p_1^\prime)
        p_2),
    \end{equation}
    for all $p_1, p_1^\prime \in P_1, p_2, p_2^\prime \in P_2$.
    The resulting perm algebra $(P_1 \oplus P_2, \star)$ is denoted by $(P_1 \bowtie_{l_{P_1}, r_{P_1}}^{l_{P_2}, r_{P_2}} P_2, \star)$ or simply $(P_1 \bowtie P_2, \star)$, and the sextuple $((P_1, \cdot), (P_2, \circ), l_{P_1}, r_{P_1}, l_{P_2}, r_{P_2})$ satisfying the above conditions is called a \textbf{matched pair of perm algebras}.
    Conversely, every perm algebra with a decomposition into the direct sum of the underlying vector spaces of two perm subalgebras can be obtained in this way.
\end{thm}
\begin{proof}
    It is straightforward.
\end{proof}

\begin{thm}\label{thm:mp2pb}
    Let $(P, \cdot)$ be a perm algebra and $\Delta_P: P \to P \otimes P$ be a linear map.
    Suppose that the linear dual of $\Delta_P$ gives a perm algebra structure $\circ$ on $P^*$.
    Then $((P, \cdot), (P^*, \circ), L_\cdot^* , L_\cdot^* - R_\cdot^*, L_\circ^*, L_\circ^* - R_\circ^*)$ is a matched pair of perm algebras if and only if $(P, \cdot, \Delta_P)$ is a perm bialgebra.
\end{thm}
\begin{proof}
    Let $m^*: P^* \to P^* \otimes P^*$ be the linear dual of $\cdot$ and write $m^*(q^*) = \sum_{(q^*)} q_{(1)}^* \otimes q_{(2)}^*$ for all $q^* \in P^*$.
  Let $p, q \in P$ and $p^*, q^* \in P^*$. Then we have
    \begin{eqnarray*}
        &&\langle q^*, (L_\circ^*(p^*) p) \cdot q + L_\circ^*((L_\cdot^* - R_\cdot^*)(p)p^*)q - L_\circ^*(p^*) ( p \cdot q)\rangle \\
        &&=\sum_{(q^*)}\langle q_{(1)}^*, L_\circ^*(p^*) p\rangle \langle q_{(2)}^*, q\rangle + \langle ((L_\cdot^* - R_\cdot^*)(p)p^*) \circ q^* ,q\rangle - \langle p^* \otimes q^*, \Delta_P(p \cdot q)\rangle \\
        &&= \sum_{(q^*)}\langle p^* \circ q_{(1)}^*, p\rangle \langle q_{(2)}^*, q\rangle + \langle p^* \otimes q^* ,((L_\cdot - R_\cdot)(p) \otimes \id)\Delta_P(q) - \Delta_P(p \cdot q)\rangle \\
        &&= \sum_{(p)}\sum_{(q^*)}\langle p^*, p_{(1)}\rangle \langle q_{(1)}^*, p_{(2)}\rangle \langle q_{(2)}^*, q\rangle + \langle p^* \otimes q^* ,((L_\cdot - R_\cdot)(p) \otimes \id)\Delta_P(q) - \Delta_P(p \cdot q)\rangle \\
        &&= \sum_{(p)}\langle p^*, p_{(1)}\rangle \langle q^*, p_{(2)} \cdot q\rangle + \langle p^* \circ q^* ,((L_\cdot - R_\cdot)(p) \otimes \id)\Delta_P(q) - \Delta_P(p \cdot q)\rangle \\
        &&= \langle p^* \otimes q^* , ((L_\cdot - R_\cdot)(p) \otimes \id)\Delta_P(q) + (\id \otimes R_\cdot(q))\Delta_P(p) - \Delta_P(p \cdot q)\rangle.
    \end{eqnarray*}
    Hence Eq.~\eqref{eq:pmp3} holds , where $r_{P_1} = L_\cdot^* - R_\cdot^*, l_{P_2}=L_\circ^*$, if and only if Eq.~\eqref{eq:pb1} holds.
    Similar argument shows that Eqs.~\eqref{eq:pmp1}-\eqref{eq:pmp10}
    (here Eq.~\eqref{eq:pmp3} is still kept for completeness) hold, where $l_{P_1}=L_\cdot^*, r_{P_1} = L_\cdot^* - R_\cdot^*, l_{P_2}=L_\circ^*, r_{P_2} = L_\circ^* - R_\circ^*$, if and only if the following equations hold respectively:
    \begin{eqnarray}
        \Delta_P(p \cdot q) &=& (L_\cdot(p) \otimes \id)\Delta_P(q) + (\id \otimes R_\cdot(q))\Delta_P(p) - (\id \otimes R_\cdot(q))\sigma\Delta(p), \label{pf:pmp1} \\
        \Delta_P(p \cdot q) - \Delta_P(q \cdot p) &=& (\id \otimes (L_\cdot - R_\cdot)(p) )\Delta_P(p) - ((L_\cdot-R_\cdot)(q) \otimes \id)\sigma\Delta_P(p), \label{pf:pmp2} \\
        \Delta_P(p \cdot q) &=& ((L_\cdot-R_\cdot)(p) \otimes \id)\Delta_P(q) + (\id \otimes R_\cdot(q))\Delta_P(p), \label{pf:pmp3} \\
        \Delta_P(p \cdot q) - \sigma\Delta_P(p \cdot q) &=& (\id \otimes L_\cdot(p))(\Delta_P(q) - \sigma\Delta_P(q)) + (L_\cdot(q) \otimes \id)(\Delta_P(p) - \sigma\Delta_P(p)), \label{pf:pmp4} \\
        \Delta_P(p \cdot q) - \Delta_P(q \cdot p) &=& \sigma\Delta_P(p \cdot q) - \sigma\Delta_P(q \cdot p), \label{pf:pmp5} \\
        \Delta_P(p \cdot q) - \sigma\Delta_P(p \cdot q) &=& \Delta_P(q \cdot p) - \sigma\Delta_P(q \cdot p), \label{pf:pmp6} \\
        ((L_\cdot-R_\cdot)(p) \otimes \id)\Delta_P(q) &=& (\id \otimes L_\cdot(p))\Delta_P(q) -(\id \otimes R_\cdot(q))\Delta_P(p) \nonumber \\
        && + ((L_\cdot-R_\cdot)(q) \otimes \id)(\Delta_P(p)-\sigma \Delta_P(p)) ,  \label{pf:pmp7} \\
        (L_\cdot(p) \otimes \id)\Delta_P(q) &=& (\id \otimes L_\cdot(p))\Delta_P(q) -(\id \otimes R_\cdot(q))(\Delta_P(p) - \sigma\Delta_P(p)) \nonumber \\
        &&+ ((L_\cdot-R_\cdot)(q) \otimes \id)(\Delta_P(p)-\sigma \Delta_P(p)),  \label{pf:pmp8} \\
        (\id \otimes R_\cdot(q))\sigma\Delta_P(p) &=& (R_\cdot(p) \otimes \id)\Delta_P(q),  \label{pf:pmp9} \\
        (\id \otimes R_\cdot(q))\sigma\Delta_P(p) &=& (R_\cdot(p) \otimes \id)\Delta_P(q).  \label{pf:pmp10}
    \end{eqnarray}
    Note that
    \begin{eqnarray*}
        &&{\rm Eq.~\eqref{pf:pmp7}~and~Eq.~\eqref{pf:pmp10}} \Longrightarrow {\rm Eq.~\eqref{pf:pmp8}}, \;\; {\rm Eq.~\eqref{pf:pmp10}} \Longleftrightarrow {\rm Eq.~\eqref{pf:pmp9}}, \\
        &&{\rm Eq.~\eqref{pf:pmp3}~and~Eq.~\eqref{pf:pmp10}} \Longrightarrow {\rm Eq.~\eqref{pf:pmp1}}, \;\;
        {\rm Eq.~\eqref{pf:pmp3}~and~Eq.~\eqref{pf:pmp7}} \Longrightarrow {\rm Eq.~\eqref{pf:pmp2}}, \;\;\\
        &&{\rm Eq.~\eqref{pf:pmp3},~ Eq.~\eqref{pf:pmp7}~and~Eq.~\eqref{pf:pmp10}} \Longrightarrow {\rm Eq.~\eqref{pf:pmp4}} \Longrightarrow {\rm Eq.~\eqref{pf:pmp5}} \Longleftrightarrow {\rm
        Eq.~\eqref{pf:pmp6}}.
    \end{eqnarray*}
   Hence $((P, \cdot), (P^*, \circ), L_\cdot^* , L_\cdot^* - R_\cdot^*, L_\circ^*, L_\circ^* - R_\circ^*)$ is a matched pair of perm algebras if and only if Eqs.~\eqref{pf:pmp3}, \eqref{pf:pmp7} and \eqref{pf:pmp10}
   hold, that is,  $(P, \cdot, \Delta_P)$ is a perm bialgebra.
\end{proof}

%%%%%%%%%%%%%%%%%%%%%%%%%%%%%%%%%%%%%%%%%%%%%%%%%%%%%%%%%%%%%%%%%%%%%%%%%%%%%%%%
\begin{defi}\label{defi:qp}
    Let $(P, \cdot)$ be a perm algebra.
    A bilinear form $\kappa$ on $(P, \cdot)$ is called \textbf{invariant} if
    \begin{equation}\label{eq:qpi}
        \kappa(p_1 \cdot p_2, p_3) = \kappa(p_1, p_2 \cdot p_3 - p_3 \cdot p_2), \;\; \forall p_1, p_2, p_3 \in P.
    \end{equation}
    A \textbf{quadratic perm algebra} $(P, \cdot, \kappa)$ is a perm algebra $(P, \cdot)$ with a skew-symmetric nondegenerate invariant bilinear form $\kappa$.
\end{defi}

\begin{defi}
    Let $(P, \cdot)$ and $(P^*, \circ)$ be perm algebras.
    Suppose that there is a perm algebra structure $\star$ on the direct sum $P \oplus P^*$ of the underlying vector spaces $(P, \cdot)$ and $(P^*, \circ)$, which contains both $(P, \cdot)$ and $(P^*, \circ)$ as perm subalgebras.
    Define a bilinear form on $P \oplus P^*$ by
    \begin{equation}
        \kappa_d(p_1 + p_1^*, p_2+ p_2^*):= \langle p_1^*, p_2 \rangle - \langle p_2^*, p_1 \rangle, \;\; \forall p_1, p_2 \in P, p_1^*, p_2^* \in P^*. \label{eq:mtpa}
    \end{equation}
    If $\kappa_d$ is invariant on $(P \oplus P^*, \star)$, then $(P \oplus P^*, \star, \kappa_d)$ is a quadratic perm algebra and the triple $((P \oplus P^*, \star, \kappa_d), P, P^*)$ is called a \textbf{(standard) Manin triple of perm algebras associated to $(P, \cdot)$ and $(P^*, \circ)$}.
\end{defi}

\begin{thm}\label{thm:mt2mp}
    Let $(P, \cdot)$ and $(P^*, \circ)$ be perm algebras.
    Then there is a Manin triple of perm algebras associated to $(P, \cdot)$ and $(P^*, \circ)$ if and only if $((P, \cdot), (P^*, \circ), L_\cdot^* , L_\cdot^* - R_\cdot^*, L_\circ^*, L_\circ^* - R_\circ^*)$ is a matched pair of perm algebras.
\end{thm}
\begin{proof}
    The proof is similar to that of \cite[Theorem~2.2.1]{bai2010double} for associative algebras.
\end{proof}

Combining Theorems~\ref{thm:mp2pb} and \ref{thm:mt2mp}, we give the following characterization of perm bialgebras.
\begin{thm}\label{thm:pmmb}
    Let $(P, \cdot)$ be a perm algebra and $\Delta_P: P \to P \otimes P$ be a linear map.
    Suppose that the linear dual of $\Delta_P$ gives a perm algebra structure $\circ$ on $P^*$.
    Then the following conditions are equivalent.
    \begin{enumerate}
        \item $(P, \cdot, \Delta_P)$ is a perm bialgebra;
        \item $((P, \cdot), (P^*, \circ), L_\cdot^* , L_\cdot^* - R_\cdot^*, L_\circ^*, L_\circ^* - R_\circ^*)$ is a matched pair of perm algebras;
        \item There is a Manin triple of perm algebras associated to $(P, \cdot)$ and $(P^*, \circ)$.
    \end{enumerate}
\end{thm}

%%%%%%%%%%%%%%%%%%%%%%%%%%%%%%%%%%%%%%%%%%%%%%%%%%%%%%%%%%%%%%%%%%%

Note that a Lie bialgebra is also characterized by a Manin triple
of Lie algebras. Hence we give the correspondence between perm
bialgebras and Lie bialgebras in terms of Manin triples as
follows.

Recall that a bilinear form $\mathfrak{B}$ on a Lie algebra $(L, [-, -])$ is called \textbf{invariant} if
\begin{equation*}
    \mathfrak{B}([a, b], c) = \mathfrak{B}(a, [b, c]), \;\; \forall a, b, c \in L.
\end{equation*}

\begin{defi}
    {\rm (\cite{chari1995guide})}
    Let $L, L_1$ and $L_2$ be Lie algebras.
    If there is a symmetric nondegenerate invariant bilinear form $\mathfrak{B}$ on $L$ such that
    \begin{enumerate}
        \item $L_1$ and $L_2$ are Lie subalgebras of $L$ and $L = L_1 \oplus L_2$ as a direct sum of vector spaces;
        \item $L_1$ and $L_2$ are isotropic with respect to $\mathfrak{B}$, that is, $\mathfrak{B}(L_i, L_i)$ = 0 for $i=1, 2$,
    \end{enumerate}
   then the triple $(L, L_1, L_2)$ is called a \textbf{Manin triple of Lie algebras} associated with $\mathfrak{B}$.
    Two Manin triples of Lie algebras $(L, L_1, L_2)$ and $(L^\prime, L_1^\prime, L_2^\prime)$ associated with $\mathfrak{B}$ and $\mathfrak{B}^\prime$ respectively are called \textbf{isomorphic} if there exists an isomorphism $\varphi: L \to L^\prime$ of Lie algebras such that $\varphi(L_i) = L_i^\prime$ for $i=1, 2$ and $\mathfrak{B}^\prime(\varphi(a), \varphi(b)) = \mathfrak{B}(a, b)$ for all $a, b \in L$.
\end{defi}

\begin{pro}\label{pro:lb2lmt}
    {\rm (\cite{chari1995guide})}
    Consider a finite-dimensional Lie algebra $(L, [-, -])$ and a Lie coalgebra $(L, \delta)$, such that the linear dual $\delta^*: L^* \otimes L^* \to L^*$ defines a Lie algebra $(L^*,\{-, -\})$.
    The triple $(L, [-, -], \delta)$ is a Lie bialgebra if and only if $(L \oplus L^*, L, L^*)$ is a Manin triple of Lie algebras associated with the bilinear form defined by
    \begin{equation}\label{eq:smp}
        \mathfrak{B}_d(a + a^*, b + b^*) := \langle a^*, b\rangle + \langle b^*, a \rangle, \;\; \forall a, b \in L, a^*, b^* \in L^*.
    \end{equation}
\end{pro}

\begin{pro}\label{pro:qp2snib}
    Let $(P, \cdot, \kappa)$ be a quadratic perm algebra,
    $(A, \diamond, \omega)$ be a quadratic pre-Lie algebra
    and $(L := P \otimes A, [-, -])$ be the induced Lie algebra from $(P, \cdot)$ and $(A, \diamond)$.
    Define a bilinear form $\mathfrak{B}$ on $L$ by
    \begin{equation*}
        \mathfrak{B}(p_1 \otimes a_1, p_2 \otimes a_2) = \kappa(p_1, p_2) \omega(a_1, a_2), \;\; \forall p_1, p_2 \in P, a_1, a_2 \in A.
    \end{equation*}
    Then $\mathfrak{B}$ is a symmetric nondegenerate invariant bilinear form on the Lie algebra $(L, [-, -])$.
\end{pro}
\begin{proof}
    For all $p_1, p_2, p_3 \in P, a_1, a_2, a_3 \in A$, we have
    \begin{eqnarray*}
        &&\mathfrak{B}([p_1 \otimes a_1, p_2 \otimes a_2], p_3 \otimes a_3) =\kappa(p_1 \cdot p_2, p_3) \omega(a_1 \diamond a_2, a_3) - \kappa(p_2 \cdot p_1, p_3) \omega(a_2 \diamond a_1, a_3)\\
        &&= \kappa(p_1, p_2 \cdot p_3 - p_3 \cdot p_2) \omega(a_1, a_3 \diamond a_2) + \kappa(p_1, p_2 \cdot p_3) \omega(a_1, a_2 \diamond a_3 - a_3 \diamond a_2) \\
        &&= -\kappa(p_1, p_3 \cdot p_2) \omega(a_1, a_3 \diamond a_2) + \kappa(p_1, p_2 \cdot p_3)\omega(a_1, a_2 \diamond a_3) =\mathfrak{B}(p_1 \otimes a_1, [p_2 \otimes a_2, p_3 \otimes a_3]).
    \end{eqnarray*}
    That is, $\mathfrak{B}$ is invariant.
    It is straightforward to verify that $\mathfrak{B}$ is symmetric and nondegenerate.
\end{proof}

\begin{pro}\label{pro:pmt2lmp}
    Let $(A, \diamond, \omega)$ be a quadratic pre-Lie algebra.
    Let $(P, \cdot)$ and $(P^*, \circ)$ be perm algebras,
    and $(P \otimes A, [-, -]_{P \otimes A})$ and $(P^* \otimes A, [-, -]_{P^* \otimes A})$ be the induced Lie algebras from $(P, \cdot)$ and $(A,\diamond)$ as well as  $(P^*, \circ)$ and $(A,\diamond)$ respectively.
    If $((P \oplus P^*, \star, \kappa_d), P, P^*)$
is a Manin triple of perm algebras, then $((P \oplus P^*)\otimes
A, P \otimes A, P^* \otimes A)$ is a Manin triple of Lie algebras
associated with the bilinear form
    \begin{equation}\label{eq:pmt2lmp}
        \mathfrak{B}(p_1 \otimes a_1 + p_1^* \otimes a_2 , p_2 \otimes a_3 + p_2^* \otimes a_4) := \langle p_1^*, p_2 \rangle \omega( a_2, a_3 ) - \langle p_2^*, p_1 \rangle  \omega(a_1, a_4),
    \end{equation}
    for all $p_1, p_2 \in P, p_1^*, p_2^* \in P^*$ and $a_1, a_2, a_3, a_4 \in A$.
\end{pro}
\begin{proof}
    It follows from Proposition~\ref{pro:qp2snib}.
\end{proof}

With the same assumptions in Proposition~\ref{pro:pmt2lmp}, due to
the nondegeneracy of the bilinear form $\omega$ in the quadratic
pre-Lie algebra $(A, \diamond, \omega)$, and hence by the linear
isomorphism $\varphi: A \to A^*$ defined by $\langle \varphi(a),
b\rangle = \omega(a, b)$ for all $a, b \in A$, we obtain a pre-Lie
algebra $(A^*, \diamond^\prime)$ on $A^*$ defined by
\begin{equation*}
    a^* \diamond^\prime b^* := \varphi(\varphi^{-1}(a^*) \diamond \varphi^{-1}(b^*)), \;\; \forall a^*, b^* \in A^*.
\end{equation*}
Note that $(A,\diamond)$ and $(A^*, \diamond^\prime)$ are isomorphic as pre-Lie algebras.
Then we obtain the induced Lie algebra
$(L^*=P^* \otimes A^*, [-, -]_{P^* \otimes A^*})$ from $(P^*,
\circ)$ and $(A^*, \diamond^\prime)$. Further, we can define a
linear map $f: (P \oplus P^*) \otimes A \to (P \otimes A) \oplus
(P^* \otimes A^*)$ by
\begin{equation*}
    p_1 \otimes a_1 + p_2^* \otimes a_2 \mapsto p_1 \otimes a_1 + p_2^* \otimes \varphi(a_2), \;\; \forall p_1 \in P, p_2^* \in P^*, a_1, a_2 \in A.
\end{equation*}
Evidently, $f$ is a linear isomorphism.
Thus we can transport the Lie algebra structure on $(P \oplus P^*) \otimes A$ to $(P \otimes A) \oplus (P^* \otimes A^*)$ by
\begin{equation*}
    [p_1 \otimes a_1 + p_1^* \otimes a_1^*, p_2 \otimes a_2 + p_2^* \otimes a_2^*]^\prime := f[p_1 \otimes a_1 + p_1^* \otimes \varphi^{-1}(a_1^*), p_2 \otimes a_2 + p_2^* \otimes
    \varphi^{-1}(a_2^*)],
\end{equation*}
for all $a_1, a_2 \in A, a_1^*, a_2^* \in A^*, p_1, p_2 \in P,
p_1^*, p_2^* \in P^*$. Moreover, the restrictions of the Lie
bracket $[-,-]^\prime$ on $P\otimes A$ and $P^* \otimes A^*$
coincide with $[-, -]_{P \otimes A}$ and $[-, -]_{P^* \otimes
A^*}$ respectively. Hence, we arrive at the following conclusion
immediately.
\begin{lem}\label{lem:mtpeq}
    With the notations above,
    $(L \oplus L^*, L = P \otimes A, L^* = P^* \otimes A^*)$ is a Manin triple of Lie algebras associated with the bilinear form defined by Eq.~\eqref{eq:smp}.
    Moreover, this Manin triple is isomorphic to the Manin triple $((P \oplus P^*) \otimes A, P \otimes A, P^* \otimes A)$ associated with the bilinear form defined by Eq.~\eqref{eq:pmt2lmp}.
\end{lem}

\begin{rmk}\label{rmk:qzplip}
    The above discussions explicitly explain the appearance of the quadratic property of $(A, \diamond)$ in
    Theorem~\ref{thm:pb2lb}. That is, a quadratic pre-Lie algebra $(A, \diamond, \omega)$
    naturally induces a pre-Lie algebra $(A^*,\diamond')$ from $\omega$ which is isomorphic to
    $(A,\diamond)$ itself,
   guaranteeing that $((P \oplus P^*) \otimes A, P \otimes A, P^* \otimes A)$ is isomorphic to $(L \oplus L^*, L = P \otimes A, L^* = P^* \otimes A^*)$ as Manin triples of Lie algebras.
\end{rmk}

Summarizing the above study, we have the following conclusion.

\begin{pro}\label{pro:pblbmm}
Let $(P, \cdot, \Delta_P)$ be a perm bialgebra and $(A, \diamond, \omega)$ be a quadratic pre-Lie algebra.
Then the Lie bialgebra $(L=P\otimes A, [-, -], \delta)$ obtained
in Theorem~\ref{thm:pb2lb} coincides with the one obtained from
the Manin triple $(L \oplus L^*, L = P \otimes A, L^* = P^*
\otimes A^*)$ of Lie algebras given in Lemma~\ref{lem:mtpeq}.
    That is, we have the following commutative diagram.
    \begin{equation*}
        \xymatrix@C=3cm@R=0.75cm{
            \txt{$(P, \cdot, \Delta_P)$ \\ a perm bialgerbra} \ar@{<->}[r]^-{Thm.~\ref{thm:pmmb}} \ar[d]^-{Thm.~\ref{thm:pb2lb}} & \txt{$((P \oplus P^*, \star, \kappa_d), P, P^*)$ \\ a Manin triple of perm algebras}
            \ar[d]^-{Prop.~\ref{pro:pmt2lmp},~Lem.~\ref{lem:mtpeq}}\\
            \txt{$(L = P \otimes A, [-, -], \delta)$ \\ a Lie bialgebra} \ar@{<->}[r]^-{Prop.~\ref{pro:lb2lmt}} & \txt{$(L \oplus L^*, L = P \otimes A, L^* = P^* \otimes A^*)$ \\ a Manin triple of Lie algebras}
        }
    \end{equation*}
\end{pro}

%%%%%%%%%%%%%%%%%%%%%%%%%%%%%%%%%%%%%%%%%%%%%%%%%%%%%%%%%%%%%%%%%%%%%%%%%%%%%%%%
\subsection{The perm Yang-Baxter equation, \texorpdfstring{$\mathcal{O}$}{O}-operators and pre-perm algebras}
Consider a special class of perm bialgebras $(P, \cdot, \Delta_P)$, where $\Delta_P$ is defined by
\begin{equation}\label{eq:pcob}
    \Delta_P(p) := (\id \otimes R_\cdot(p) - (L_\cdot - R_\cdot)(p) \otimes \id)(r), \;\; \forall p \in
    P,
\end{equation}
for some $r \in P \otimes P$.

Let $r = \sum_{\alpha \in I} p_\alpha \otimes q_\alpha \in P
\otimes P$ and $r^\prime = \sum_{\beta \in J} p_\beta^\prime
\otimes q_\beta^\prime \in P \otimes P$ where $I, J$ are finite
index sets. Set
\begin{eqnarray*}
    &&r_{12} \cdot r_{23}^\prime := \sum_{\alpha, \beta} p_\alpha \otimes q_\alpha \cdot p_\beta^\prime \otimes q_\beta^\prime, \;\;\; r_{13} \cdot r_{23}^\prime := \sum_{\alpha, \beta} p_\alpha \otimes p_\beta^\prime \otimes q_\alpha \cdot q_\beta^\prime, \\
    &&r_{13} \cdot r_{12}^\prime := \sum_{\alpha, \beta} p_\alpha \cdot p_\beta^\prime \otimes q_\beta^\prime \otimes q_\alpha, \;\;\; r_{12} \cdot r_{13}^\prime := \sum_{\alpha, \beta} p_\alpha \cdot p_\beta^\prime \otimes q_\alpha \otimes q_\beta^\prime, \\
    &&r_{23} \cdot r_{13}^\prime := \sum_{\alpha, \beta} p_\beta^\prime \otimes p_\alpha \otimes q_\alpha \cdot q_\beta^\prime, \;\;\; r_{23} \cdot r_{12}^\prime := \sum_{\alpha, \beta} p_\beta^\prime \otimes p_\alpha \cdot q_\beta^\prime \otimes q_\alpha.
\end{eqnarray*}

\begin{lem}\label{lem:pbco}
    Let $(P, \cdot)$ be a perm algebra and $r \in P \otimes P$.
    Define a linear map $\Delta_P: P \to P \otimes P$ by Eq.~\eqref{eq:pcob}.
    Then the following conclusions hold.
    \begin{enumerate}
        \item \label{it:pb1} Eq.~\eqref{eq:pb1} holds automatically.
        \item \label{it:pb2}  Eq.~\eqref{eq:pb2} holds if and only if the following equation holds:
        \begin{equation}
            (R_\cdot(p) \otimes R_\cdot(q))(r - \sigma(r)) = 0, \;\; \forall p, q \in P. \label{eq:pb2co}
        \end{equation}
        \item \label{it:pb3}  Eq.~\eqref{eq:pb3} holds if and only if the following equation holds:
        \begin{equation}
            ((R_\cdot(q \cdot p)-R_\cdot(p \cdot q)) \otimes \id - (L_\cdot - R_\cdot )(q) \otimes (L_\cdot - R_\cdot )(p) )(r - \sigma(r)) = 0, \;\; \forall p, q \in A. \label{eq:pb3co}
        \end{equation}
        \item \label{it:pco} $(P, \Delta_P)$ is a perm coalgebra if and only if the following equations hold:
        \begin{eqnarray}
            &&(\id \otimes \id \otimes R_\cdot(p) - ((L_\cdot-R_\cdot)(p) \otimes \id \otimes \id)(\sigma \otimes \id))(r_{12} \cdot r_{23} - r_{13} \cdot r_{23} + r_{12} \cdot r_{13} - r_{13} \cdot r_{12} ) \nonumber \\
            &&= ((L_\cdot-R_\cdot)(p) \otimes \id \otimes \id)(\sigma \otimes \id) ( (\sigma(r)-r)_{12} \cdot r_{23} + (\sigma(r)-r)_{12} \cdot r_{13} - r_{13} \cdot (\sigma(r)-r)_{12}), \label{eq:pcass} \\
            &&(\id \otimes \id \otimes R_\cdot(p))(\id \otimes \id \otimes \id - \sigma \otimes \id )(r_{12} \cdot r_{23} - r_{13} \cdot r_{23} + r_{12}\cdot r_{13} - r_{13} \cdot r_{12}  ) \nonumber \\
            &&= (\id \otimes (L_\cdot - R_\cdot)(p) \otimes \id)( (r - \sigma(r))_{12} \cdot r_{23}) + ((L_\cdot - R_\cdot)(a) \otimes \id \otimes \id)((r-\sigma(r))_{12} \cdot r_{13}), \label{eq:pclc}
        \end{eqnarray}
        for all $p, q \in P$.
    \end{enumerate}
\end{lem}
\begin{proof}
    Let $p, q \in P$.

    (\ref{it:pb1}).
    By Eq.~\eqref{eq:perm}, we have $R_\cdot(q) R_\cdot(p)  = R_\cdot(p \cdot q)$ and $(L_\cdot-R_\cdot)(p)(L_\cdot - R_\cdot)(q) = (L_\cdot-R_\cdot)(p \cdot q)$.
    Therefore,
    \begin{eqnarray*}
        &&((L_\cdot-R_\cdot)(p) \otimes \id)\Delta_P(q) + (\id \otimes R_\cdot(q))\Delta_P(p) - \Delta_P(p \cdot q)\\
        &&=( (L_\cdot-R_\cdot)(p) \otimes R_\cdot(q) - (L_\cdot-R_\cdot)(p)(L_\cdot - R_\cdot)(q) \otimes \id + \id \otimes R_\cdot(q) R_\cdot(p) - (L_\cdot - R_\cdot)(p) \otimes R_\cdot(q) )(r) \\
        &&\quad - (\id \otimes R_\cdot(p \cdot q) - (L_\cdot-R_\cdot)(p \cdot q)\otimes \id )(r) = 0 .
    \end{eqnarray*}
    Hence, Eq.~\eqref{eq:pb1} holds automatically.
    Similarly, we show that Items (\ref{it:pb2}) and (\ref{it:pb3}) hold.

    (\ref{it:pco}).
    Writing $r = \sum_\alpha p_\alpha \otimes q_\alpha \in P \otimes P$, we have
    \begin{eqnarray*}
        &&(\Delta \otimes \id)\Delta(p) = \sum_\alpha(\Delta \otimes \id)(p_\alpha \otimes q_\alpha \cdot p - (p \cdot p_\alpha - p_\alpha \cdot p) \otimes q_\alpha) \\
        &&= \sum_{\alpha, \beta}( p_\alpha \otimes q_\alpha \cdot p_\beta \otimes q_\beta \cdot p - (p_\beta \cdot p_\alpha - p_\alpha \cdot p_\beta) \otimes q_\alpha \otimes q_\beta \cdot p \\
        &&\quad - p_\alpha \otimes (p \cdot q_\alpha \cdot p_\beta - q_\alpha \cdot p_\beta \cdot p) \otimes q_\beta -  (p \cdot p_\alpha \cdot p_\beta - p_\alpha \cdot p_\beta \cdot p) \otimes q_\alpha \otimes q_\beta )
    \end{eqnarray*}
    and similarly
    \begin{eqnarray*}
        (\id \otimes \Delta)\Delta(p) &=& \sum_{\alpha, \beta} ( p_\beta \otimes p_\alpha \otimes q_\beta \cdot q_\alpha \cdot p - p_\alpha \otimes (p \cdot q_\alpha \cdot p_\beta - q_\alpha \cdot p_\beta \cdot p) \otimes q_\beta \\
        &&\quad - (p \cdot p_\beta - p_\beta \cdot p) \otimes p_\alpha \otimes q_\alpha \cdot q_\beta + (p \cdot p_\beta - p_\beta \cdot p) \otimes (q_\beta \cdot p_\alpha - p_\alpha \cdot q_\beta) \otimes q_\alpha ).
    \end{eqnarray*}
    Therefore,
    \begin{eqnarray*}
        (\Delta \otimes \id)\Delta(p) - (\id \otimes \Delta)\Delta(p) &=& (\id \otimes \id \otimes R_\cdot(p))(r_{12} \cdot r_{23} - r_{13} \cdot r_{23} + r_{12} \cdot r_{13} - r_{13} \cdot r_{12}  ) \\
        && - ((L_\cdot - R_\cdot)(p) \otimes \id \otimes \id)(r_{12} \cdot r_{13} - r_{23} \cdot r_{13}+r_{12} \cdot r_{23}-r_{23} \cdot r_{12}).
    \end{eqnarray*}
    It is straightforward to show that
    \begin{eqnarray}\label{eq:eee}
        &&(r_{12} \cdot r_{13} - r_{23} \cdot r_{13} + r_{12} \cdot r_{23} - r_{23} \cdot r_{12}) - (\sigma \otimes \id)(r_{12} \cdot r_{23} - r_{13} \cdot r_{23} + r_{12} \cdot r_{13} - r_{13} \cdot r_{12} ) \nonumber\\
        &&= (\sigma \otimes \id) ( (\sigma(r)-r)_{12} \cdot r_{23} + (\sigma(r)-r)_{12} \cdot r_{13} - r_{13} \cdot (\sigma(r)-r)_{12}).
    \end{eqnarray}
    Then $(\Delta \otimes \id)\Delta = (\id \otimes \Delta)\Delta$ if
    and only if Eq.~\eqref{eq:pcass} holds.
    On the other hand, note that
    \begin{equation*}
        (\id \otimes \id \otimes R_\cdot(p)) (r_{13} \cdot r_{23}) = (\id \otimes \id \otimes R_\cdot(p))(\sigma \otimes \id) (r_{13} \cdot r_{23}),
    \end{equation*}
    we have
    \begin{eqnarray*}
        &&(\Delta \otimes \id)\Delta(p) - (\sigma \otimes \id)(\Delta \otimes \id)\Delta(p) \\
        &&=\sum_{\alpha, \beta}( p_\alpha \otimes q_\alpha \cdot p_\beta \otimes q_\beta \cdot p - (p_\beta \cdot p_\alpha - p_\alpha \cdot p_\beta) \otimes q_\alpha \otimes q_\beta \cdot p \\
        &&\quad - p_\alpha \otimes (p \cdot q_\alpha \cdot p_\beta - q_\alpha \cdot p_\beta \cdot p) \otimes q_\beta -  (p \cdot p_\alpha \cdot p_\beta - p_\alpha \cdot p_\beta \cdot p) \otimes q_\alpha \otimes q_\beta \\
        &&\quad -q_\alpha \cdot p_\beta \otimes p_\alpha \otimes  q_\beta \cdot p + q_\alpha \otimes (p_\beta \cdot p_\alpha - p_\alpha \cdot p_\beta) \otimes q_\beta \cdot p \\
        &&\quad + (p \cdot q_\alpha \cdot p_\beta - q_\alpha \cdot p_\beta \cdot p) \otimes p_\alpha \otimes q_\beta +  q_\alpha \otimes (p \cdot p_\alpha \cdot p_\beta - p_\alpha \cdot p_\beta \cdot p) \otimes q_\beta ) \\
        &&= (\id \otimes \id \otimes R_\cdot(p)) (
        r_{12} \cdot r_{23} - r_{13} \cdot r_{23} + r_{12} \cdot r_{13} - r_{13} \cdot r_{12}  ) \\
        &&\quad + (\id \otimes \id \otimes R_\cdot(p))(\sigma \otimes \id)(-r_{12} \cdot r_{23} + r_{13} \cdot r_{23} - r_{12} \cdot r_{13} + r_{13} \cdot r_{12}) \\
        &&\quad - (\id \otimes (L_\cdot - R_\cdot)(p) \otimes \id)((r - \sigma(r))_{12} \cdot r_{23}) - ((L_\cdot - R_\cdot)(p) \otimes \id \otimes \id)((r-\sigma(r))_{12} \cdot r_{13}),
    \end{eqnarray*}
    that is, $(\Delta \otimes \id)\Delta = (\sigma \otimes \id)(\Delta \otimes \id)\Delta$ if and only if Eq.~\eqref{eq:pclc} holds.
    The proof is completed.
\end{proof}

There are the following two conclusions immediately.

\begin{cor}
    Let $(P, \cdot)$ be a perm algebra and $r \in P \otimes P$.
    Define a linear map $\Delta_P: P \to P \otimes P$ by Eq.~\eqref{eq:pcob}.
    Then $(P, \cdot, \Delta_P)$ is a perm bialgebra if and only if Eqs.~\eqref{eq:pb2co}-\eqref{eq:pclc} hold.
\end{cor}

\begin{cor}\label{cor:cpbr}
    Let $(P, \cdot)$ be a perm algebra and $r \in P \otimes P$ be symmetric.
    Define a linear map $\Delta_P: P \to P \otimes P$ by Eq.~\eqref{eq:pcob}.
    Then $(P, \cdot, \Delta)$ is a perm bialgebra if and only if the following equations hold:
    \begin{eqnarray*}
        (\id \otimes \id \otimes R_\cdot(p) - ((L_\cdot-R_\cdot)(p) \otimes \id \otimes \id)(\sigma \otimes \id))(r_{12} \cdot r_{23} - r_{13} \cdot r_{23} + r_{12} \cdot r_{13} - r_{13} \cdot r_{12} ) &=& 0, \\
        (\id \otimes \id \otimes R_\cdot(p))(\id \otimes \id \otimes \id - \sigma \otimes \id )(r_{12} \cdot r_{23} - r_{13} \cdot r_{23} + r_{12} \cdot r_{13} - r_{13} \cdot r_{12} ) &=& 0,
    \end{eqnarray*}
    for all $p \in P$.
\end{cor}

Corollary~\ref{cor:cpbr} motivates to give the following notion,
which is an analogue of the classical Yang-Baxter equation (CYBE)
in a Lie algebra (\cite{chari1995guide}).
\begin{defi}
    Let $(P, \cdot)$ be a perm algebra and $r \in P \otimes P$.
    The equation
    \begin{equation*}
        r_{12} \cdot r_{23} - r_{13} \cdot r_{23} + r_{12} \cdot r_{13} - r_{13} \cdot r_{12}  = 0
    \end{equation*}
    is called the \textbf{perm Yang-Baxter equation} in $(P, \cdot)$ or simply the \textbf{perm-YBE} in $(P, \cdot)$.
\end{defi}

By Corollary~\ref{cor:cpbr}, we have the following conclusion.
\begin{cor}\label{cor:pbc}
    Let $(P, \cdot)$ be a perm algebra and $r \in P \otimes P$.
    If $r$ is a symmetric solution of the perm-YBE in $(P, \cdot)$,
    then $(P, \cdot, \Delta_P)$ is a perm bialgebra,
    where $\Delta_P$ is defined by Eq.~\eqref{eq:pcob}.
\end{cor}

For a vector space $P$, through the isomorphism $P \otimes P \cong \Hom_\mathbf{k}(P^*, P)$,
any $r = \sum\limits_{\alpha \in I} p_\alpha \otimes q_\alpha \in P \otimes P$ can be identified as a map $r^\sharp$ from $P^*$ to $P$:
\begin{equation*}
    r^\sharp: P^* \to P, \;\; p^* \mapsto \sum\limits_{\alpha \in I} \langle p^*, p_\alpha \rangle q_\alpha, \;\;
    \forall p^* \in P^*.
\end{equation*}

\begin{thm}\label{thm:pybeo}
    Let $(P, \cdot)$ be a perm algebra and $r \in P \otimes P$ be symmetric.
    Then $r$ is a solution of the perm-YBE in $(P, \cdot)$ if and only if $r^\sharp$ satisfies the following equation:
    \begin{equation*}
        r^\sharp(p^*) \cdot r^\sharp(q^*) = r^\sharp( L_\cdot^*(r^\sharp(p^*)) q^* + (L_\cdot^* - R_\cdot^*)(r^\sharp(q^*)) p^* ), \;\; \forall p^*, q^* \in P^*.
    \end{equation*}
\end{thm}
\begin{proof}
The proof is similar to that of \cite[Theorem~2.4.7]{bai2010double}.
\end{proof}

The notion of an  $\mathcal{O}$-operator of a Lie algebra was
introduced in \cite{kupershmidt1999what} as a natural
generalization of the CYBE in a Lie algebra. Similarly, we give the following notion as
a generalization of the perm-YBE in a perm algebra.

\begin{defi}
    Let $(P, \cdot)$ be a perm algebra and $(V, l, r)$ be a representation of $(P, \cdot)$.
    A linear map $T: V \to P$ is called an $\mathcal{O}$-operator of $(P, \cdot)$ associated to $(V, l, r)$ if $T$ satisfies
    \begin{equation*}
        T(u) \cdot T(v) = T( l(T(u))v + r(T(v)) u ), \;\; \forall u, v \in V.
    \end{equation*}
\end{defi}

Thus Theorem~\ref{thm:pybeo} is rewritten in terms of
$\mathcal{O}$-operators.
\begin{cor}
    Let $(P, \cdot)$ be a perm algebra and $r \in P \otimes P$ be symmetric.
    Then $r$ is a solution of the perm-YBE in $(P, \cdot)$ if and only if
    $r^\sharp$ is an $\mathcal{O}$-operator of $(P, \cdot)$ associated to $(P^*, L_\cdot^*, L_\cdot^* - R_\cdot^*)$.
\end{cor}

\begin{thm}\label{thm:o2pybe}
    Let $(P, \cdot)$ be a perm algebra and $(V, l, r)$ be a representation of $(P, \cdot)$.
    Let $T: V \to P$ be a linear map which is identified as an element in $(P \ltimes_{l^*, l^* - r^*} V^*) \otimes(P \ltimes_{l^*, l^* - r^*} V^*)$
    (through $\Hom_\mathbf{k}(V, P) \cong P \otimes V^* \subset (P \ltimes_{l^*, l^* - r^*} V^*) \otimes (P \ltimes_{l^*, l^* - r^*} V^*)$).
    Then $r = T + \sigma(T)$ is a symmetric solution of the perm-YBE in $(P \ltimes_{l^*, l^* - r^*}V^*, \cdot)$ if and only if $T$ is an $\mathcal{O}$-operator of $(P, \cdot)$ associated to $(V, l, r)$.
\end{thm}
\begin{proof}
    The proof follows the same argument as the one in \cite{bai2007unified} for Lie algebras.
\end{proof}

\begin{defi}
    A \textbf{pre-perm algebra} is a triple $(P, \lhd, \rhd)$, where $P$ is a vector space, and $\lhd$ and $\rhd$ are binary operations such that
    \begin{eqnarray*}
        &&p_1 \lhd (p_2 \lhd p_3 + p_2 \rhd p_3) = (p_1 \lhd p_2) \lhd p_3 = (p_2 \rhd p_1) \lhd p_3 = p_2 \rhd (p_1 \lhd p_3), \\
        &&(p_1 \lhd p_2 + p_1 \rhd p_2) \rhd p_3 = p_1 \rhd (p_2 \rhd p_3) = p_2 \rhd (p_1 \rhd p_3), \;\; \forall p_1, p_2, p_3 \in P.
    \end{eqnarray*}
\end{defi}

By a direct verification, we have the following conclusion.
\begin{pro}\label{pro:pp2p}
\begin{enumerate}
\item \label{it:a1} Let $(A, \lhd, \rhd)$ be a pre-perm algebra.
    Then the binary operation
    \begin{equation*}
        p \cdot q := p \lhd q + p \rhd q, \;\; \forall p, q \in P,
    \end{equation*}
    defines a perm algebra $(P, \cdot)$, called the \textbf{sub-adjacent perm algebra} of $(P, \lhd, \rhd)$.
    Moreover $(P, L_\rhd, R_\lhd)$ is a representation of $(P, \cdot)$ and
    the identity map $\id$ is an $\mathcal{O}$-operator of $(P, \cdot)$ associated to $(P, L_\rhd, R_\lhd)$.
\item Let $T: V \to P$ be an $\mathcal{O}$-operator of a perm
algebra $(P, \cdot)$ associated to $(V, l, r)$.
    Then there exists a pre-perm algebra structure $(V, \lhd, \rhd)$ on $V$,
    where $\lhd$ and $\rhd$ are respectively defined by
    \begin{equation*}
        u \rhd v := l(T(u)) v \;\; {\rm and } \;\;
        u \lhd v := r(T(v)) u, \;\;
        \forall u, v \in V.
    \end{equation*}
\end{enumerate}
\end{pro}

\begin{pro}\label{pro:pp2pb}
    Let $(P, \lhd, \rhd)$ be a pre-perm algebra and $(P, \cdot)$ be the sub-adjacent perm algebra of $(P, \lhd, \rhd)$.
    Let $\{e_1, e_2, \cdots, e_n\}$ be a basis of $P$ and $\{e_1^*, e_2^*, \cdots, e_n^*\}$ be the dual basis.
    Then $$r = \sum_{i=1}^n (e_i \otimes e_i^* + e_i^* \otimes e_i)$$ is a symmetric solution of the perm-YBE in the perm algebra $(P \ltimes_{L_\rhd^*, L_\rhd^*-R_\lhd^*} P^*, \cdot)$.
    Therefore there is a perm bialgebra
    $(P \ltimes_{L_\rhd^*, L_\rhd^*-R_\lhd^*} P^*, \cdot, \Delta)$,
    where the linear map $\Delta$ is defined by Eq.~\eqref{eq:pcob} through the above $r$.
\end{pro}
\begin{proof}
    By Proposition~\ref{pro:pp2p} (\ref{it:a1}), the identity map $\id$ is an $\mathcal{O}$-operator of $(P, \cdot)$ associated to $(P, L_\rhd, R_\lhd)$.
    Note that $\id = \sum_{i=1}^n e_i \otimes e_i^*$.
    Hence, the conclusion follows from Theorem~\ref{thm:o2pybe} and Corollary~\ref{cor:pbc}.
\end{proof}

%%%%%%%%%%%%%%%%%%%%%%%%%%%%%%%%%%%%%%%%%%%%%%%%%%%%%%%%%%%%%%%%%%%%%%%%%%%%%%%%
\subsection{Infinite-dimensional Lie bialgebras from the perm Yang-Baxter equation}
We now turn to the constructions of skew-symmetric solutions of
the CYBE from symmetric solutions of the perm-YBE.

Let $(L = \oplus_{i \in \mathbb{Z}} L_i, [-, -])$ be a
$\mathbb{Z}$-graded Lie algebra and $r = \sum_{i,j,\alpha}
a_{i\alpha} \otimes b_{j\alpha} \in L \hat{\otimes} L$. Set
\begin{eqnarray*}
    \hphantom{}&&[r_{12}, r_{13}] := \sum_{i,j,k,l,\alpha,\beta} [a_{i\alpha}, a_{k\beta}] \otimes a_{j\alpha} \otimes b_{l\beta}, \;\;\; [r_{12}, r_{23}] := \sum_{i,j,k,l,\alpha,\beta} a_{i\alpha} \otimes [b_{j\alpha}, a_{k\beta}] \otimes b_{l\beta}, \\
    \hphantom{}&&[r_{13}, r_{23}] := \sum_{i,j,k,l,\alpha,\beta} a_{i\alpha} \otimes a_{k\beta} \otimes [b_{j\alpha}, b_{l\beta}]
\end{eqnarray*}
provided the sums make sense.
If $r \in L \hat{\otimes} L$ is skew-symmetric and satisfies the \textbf{classical Yang-Baxter equation (CYBE)}
\begin{equation*}
    [r_{12}, r_{13}] + [r_{12}, r_{23}] + [r_{13},r_{23}] = 0
\end{equation*}
as an element in $L \hat{\otimes} L \hat{\otimes} L$, then $r$ is called a \textbf{completed solution of the CYBE} in $(L, [-, -])$.

\begin{pro}\label{pro:lbc}
    {\rm (\cite{hong2023infinite})}
    Let $r \in L \hat{\otimes} L$ be a skew-symmetric completed solution of the CYBE in $(L, [-, -])$.
    Then the triple $(L, [-, -], \delta)$ is a completed Lie bialgebra,
    where $\delta: L \to L \hat{\otimes} L$ is defined by
    \begin{equation}\label{eq:lbc}
        \delta(a) := (\ad(a) \hat{\otimes} \id + \id \hat{\otimes} \ad(a))(r), \;\; \forall a \in L.
    \end{equation}
\end{pro}

\begin{pro}\label{pro:pybe2cybe}
    Let $(P, \cdot)$ be a perm algebra, $(A = \oplus_{i \in \mathbb{Z}} A_i, \diamond, \omega)$ be a quadratic $\mathbb{Z}$-graded pre-Lie algebra, and $(L := P \otimes A, [-, -])$ be the induced Lie algebra from $(P, \cdot)$ and $(A,\diamond)$.
    Let $\{e_\lambda\}_{\lambda \in \Lambda}$ be a basis of $A$ consisting of homogeneous elements and $\{f_\lambda\}_{\lambda \in \Lambda}$ be the dual basis with respect to $\omega$ consisting of homogeneous elements.
    If $r = \sum_{\alpha} p_{\alpha} \otimes q_{\alpha} \in P \otimes P$ is a symmetric solution of the perm-YBE in $(P, \cdot)$,
    then the tensor element
    \begin{equation}\label{eq:pybe2cybe}
        \tilde{r} := \sum_{\lambda \in \Lambda} \sum_{\alpha} (p_\alpha \otimes e_\lambda) \otimes (q_\alpha \otimes f_\lambda) \in L \hat{\otimes} L
    \end{equation}
    is a skew-symmetric completed solution of the CYBE in $(L, [-, -])$.
    Further, if $(A, \diamond, \omega)$ is the quadratic $\mathbb{Z}$-graded pre-Lie algebra given in Example~\ref{ex:qgpl}, then
    \begin{equation*}
        \tilde{r} := \sum_{i \in \mathbb{Z}} \sum_{\alpha} (p_\alpha t^i \otimes q_\alpha s^{-i} - p_\alpha s^{-i} \otimes q_\alpha t^{i})  \in L \hat{\otimes} L
    \end{equation*}
    is a skew-symmetric completed solution of the CYBE in $(L, [-, -])$ if and only if $r$ is a symmetric solution of the perm-YBE in $(P, \cdot)$.
\end{pro}
\begin{proof}
    For all $\mu, \nu \in \Lambda$, we have
    \begin{equation*}
        \widetilde{\omega}(\sum_{\lambda \in \Lambda} e_\lambda \otimes f_\lambda, e_\mu \otimes e_\nu)  = \omega(e_\nu, e_\mu) =  \widetilde{\omega}(-\sum_{\lambda \in \Lambda} f_\lambda \otimes e_\lambda, e_\mu \otimes e_\nu).
    \end{equation*}
    Then the left nondegeneracy of $\widetilde{\omega}$ yields that $\sum_{\lambda \in \Lambda} e_\lambda \otimes f_\lambda = - \sum_{\lambda \in \Lambda} f_\lambda \otimes e_\lambda$.
    Therefore, $\tilde{r}$ is skew-symmetric.
    Furthermore,
    \begin{eqnarray*}
        &&[\tilde{r}_{12}, \tilde{r}_{13}] + [\tilde{r}_{12}, \tilde{r}_{23}] + [\tilde{r}_{13}, \tilde{r}_{23}] \\
        &&=\sum_{\lambda, \eta \in \Lambda}\sum_{\alpha, \beta} \big( (p_\alpha \cdot p_\beta \otimes e_\lambda \diamond e_\eta - p_\beta \cdot p_\alpha \otimes e_\eta \diamond e_\lambda) \otimes (q_\alpha \otimes f_\lambda) \otimes (q_\beta \otimes f_\eta) \\
        &&\quad + (p_\alpha \otimes e_\lambda) \otimes (q_\alpha \cdot p_\beta \otimes f_\lambda \diamond e_\eta - p_\beta \cdot q_\alpha \otimes e_\eta \diamond f_\lambda) \otimes (q_\beta \otimes f_\eta) \\
        &&\quad (p_\alpha \otimes e_\lambda) \otimes (p_\beta \otimes e_\eta) \otimes (q_\alpha \cdot q_\beta \otimes f_\lambda \diamond f_\eta - q_\beta \cdot q_\alpha \otimes f_\eta \diamond f_\lambda) \big)
    \end{eqnarray*}
    For all $s,u,v \in \Lambda$, we have
    \begin{equation*}
        \widetilde{\omega}(\sum_{\lambda, \eta \in \Lambda} e_\lambda \diamond e_\eta \otimes f_\lambda \otimes f_\eta, e_s \otimes e_u \otimes e_v) = \omega(e_u \diamond e_v, e_s)  = \widetilde{\omega}(\sum_{\lambda, \eta \in \Lambda} e_\lambda \otimes f_\lambda \diamond e_\eta \otimes f_\eta, e_s \otimes e_u \otimes e_v).
    \end{equation*}
    Since $\widetilde{\omega}$ is left nondegenerate, we obtain
    \begin{equation*}
        \sum_{\lambda, \eta \in \Lambda} e_\lambda \diamond e_\eta \otimes f_\lambda \otimes f_\eta = \sum_{\lambda, \eta \in \Lambda} e_\lambda \otimes f_\lambda \diamond e_\eta \otimes f_\eta.
    \end{equation*}
    Similarly, we show that
    \begin{eqnarray*}
        \sum_{\lambda, \eta \in \Lambda} e_\lambda \otimes e_\eta \diamond f_\lambda \otimes f_\eta &=& \sum_{\lambda, \eta \in \Lambda} e_\lambda \diamond e_\eta \otimes f_\lambda \otimes f_\eta - \sum_{\lambda, \eta \in \Lambda} e_\eta \diamond e_\lambda \otimes f_\lambda \otimes f_\eta, \\
        \sum_{\lambda, \eta \in \Lambda} e_\lambda \otimes e_\eta \otimes f_\lambda \diamond f_\eta &=& - \sum_{\lambda, \eta \in \Lambda} e_\eta \diamond e_\lambda \otimes f_\lambda \otimes f_\eta, \\
        \sum_{\lambda, \eta \in \Lambda} e_\lambda \otimes e_\eta \otimes f_\eta \diamond f_\lambda &=& \sum_{\lambda, \eta \in \Lambda} e_\lambda \diamond e_\eta \otimes f_\lambda \otimes f_\eta - \sum_{\lambda, \eta \in \Lambda} e_\eta \diamond e_\lambda \otimes f_\lambda \otimes f_\eta.
    \end{eqnarray*}
    Therefore,
    \begin{eqnarray*}
        [\tilde{r}_{12}, \tilde{r}_{13}] + [\tilde{r}_{12}, \tilde{r}_{23}] + [\tilde{r}_{13}, \tilde{r}_{23}] &=& (r_{12} \cdot r_{13} - r_{23} \cdot r_{13} + r_{12} \cdot r_{23} - r_{23} \cdot r_{12} ) \bullet (\sum_{\lambda, \eta \in \Lambda}e_\lambda \diamond e_\eta \otimes f_\lambda \otimes f_\eta) \\
        && + (r_{23} \cdot r_{12} - r_{13} \cdot r_{12} -r_{13} \cdot r_{23}+r_{23} \cdot r_{13})\bullet (\sum_{\lambda, \eta \in \Lambda}e_\eta \diamond e_\lambda \otimes f_\lambda \otimes f_\eta).
    \end{eqnarray*}
    By Eq.~(\ref{eq:eee}), we have
    \begin{eqnarray*}
        &&(-r_{13} \cdot r_{12} + r_{23} \cdot r_{12}-r_{13} \cdot r_{23}+r_{23} \cdot r_{13}) \\
        &&=-(r_{12} \cdot r_{13} -r_{23} \cdot r_{13} + r_{12} \cdot r_{23}-r_{23} \cdot r_{12})+(r_{12} \cdot r_{23} - r_{13} \cdot r_{23} + r_{12} \cdot r_{13} - r_{13} \cdot r_{12}  ) = 0.
    \end{eqnarray*}
    Hence,
    \begin{equation*}
        [\tilde{r}_{12}, \tilde{r}_{13}] + [\tilde{r}_{12}, \tilde{r}_{23}] + [\tilde{r}_{13}, \tilde{r}_{23}] = 0.
    \end{equation*}
    Next consider the case when $(A, \diamond, \omega)$ is the quadratic $\mathbb{Z}$-graded pre-Lie algebra given in Example~\ref{ex:qgpl}.
    Note that $\{s^{-i}, -t^{i}\}_{i \in \mathbb{Z}}$ is the dual basis of $\{t^i, s^{-i}\}_{i \in \mathbb{Z}}$,
    the ``if'' part follows immediately.
    Conversely, suppose that $\tilde{r}$ is a skew-symmetric completed solution of the CYBE in $(L, [-, -])$.
    Note that
    \begin{equation*}
        0 = \hat{\sigma}(\tilde{r}) + \tilde{r} = \sum_{i \in \mathbb{Z}}\sum_{\alpha}(q_\alpha s^{-i} \otimes p_\alpha t^i - q_\alpha t^i \otimes p_\alpha s^{-i} +p_\alpha t^i \otimes q_\alpha s^{-i} - p_\alpha s^{-i} \otimes q_\alpha t^i),
    \end{equation*}
    we show that $r$ is symmetric by comparing the coefficients of $s^{-i} \otimes t^i$.
    Moreover,
    \begin{eqnarray*}
        0 &=& [\tilde{r}_{12}, \tilde{r}_{13}] + [\tilde{r}_{12}, \tilde{r}_{23}] + [\tilde{r}_{13}, \tilde{r}_{23}] \\
        &=& (r_{12} \cdot r_{13} -r_{23} \cdot r_{13} + r_{12} \cdot r_{23}-r_{23} \cdot r_{12}) \\
        &&\quad \bullet \sum_{i,j}(j t^{i+j-1} \otimes s^{-i} \otimes s^{-j} - (2i-j-1) s^{i-j-1} \otimes s^{-i} \otimes t^{j} - i s^{i-j-1} \otimes t^j \otimes s^{-i}) \\
        && + (-r_{13} \cdot r_{12} + r_{23} \cdot r_{12}-r_{13} \cdot r_{23}+r_{23} \cdot r_{13}) \\
        &&\quad \bullet \sum_{i,j}(i t^{i+j-1} \otimes s^{-i} \otimes s^{-j} - i s^{i-j-1} \otimes s^{-i} \otimes t^{j} - (2i-j-1) s^{i-j-1} \otimes t^j \otimes s^{-i})
    \end{eqnarray*}
    Comparing the coefficients of $t^{0} \otimes s^{0} \otimes s^{-1}$, we get
    \begin{equation*}
        r_{12} \cdot r_{13} + r_{12} \cdot r_{23}-r_{23} \cdot r_{12}-r_{23} \cdot r_{13} = 0.
    \end{equation*}
    Then, by Eq.~(\ref{eq:eee}), we show that $r$ is a symmetric solution of the perm-YBE in $(P, \cdot)$.
\end{proof}

\begin{cor}\label{cor:pybe2lb}
    With the same assumptions in Proposition~\ref{pro:pybe2cybe},
    let $\Delta_P: P \to P \otimes P$ be a linear map defined by Eq.~\eqref{eq:pcob} through $r$,
    $\Delta_A: A \to A \hat{\otimes} A$ be the linear map defined by Eq.~\eqref{eq:qgpl2cplc}
    and $\delta: L \to L \hat{\otimes} L$ be the linear map defined by Eq.~\eqref{eq:pc2lc}.
    Then $(P, \cdot, \Delta_P)$ is a perm bialgebra by Corollary~\ref{cor:pbc} and hence $(L, [\cdot, \cdot], \delta)$ is a completed Lie bialgebra by  Theorem~\ref{thm:pb2lb}.
    It coincides with the completed Lie bialgebra with $\delta$ defined by Eq.~\eqref{eq:lbc} through $\tilde{r}$ by Proposition~\ref{pro:lbc}, where $\tilde{r}$ is defined by Eq.~\eqref{eq:pybe2cybe}.
    Hence, we have the following commutative diagram.
    \begin{equation*}
        \xymatrix@C=3cm@R=0.75cm{
            \txt{symmetric solutions of the perm-YBE} \ar[r]^-{Cor.~\ref{cor:pbc}} \ar[d]^-{Prop.~\ref{pro:pybe2cybe}} & \txt{perm bialgebras} \ar[d]^-{Thm.~\ref{thm:pb2lb}}\\
            \txt{skew-symmetric solutions of the CYBE} \ar[r]^-{Prop.~\ref{pro:lbc}} & \txt{Lie bialgebras}
        }
    \end{equation*}
\end{cor}
\begin{proof}
    By Corollary~\ref{cor:pbc}, $(P, \cdot, \Delta_P)$ is a perm bialgebra.
    Then by Theorem~\ref{thm:pb2lb}, there is a completed Lie bialgebra $(L, [-, -],\delta)$ on $L$ where $\delta$ is defined by Eq.~\eqref{eq:pc2lc}.
    For all $\mu, \nu \in \Lambda$, we have
    \begin{eqnarray*}
        &&\widetilde{\omega}(\sum_{\lambda \in \Lambda} a \diamond e_\lambda \otimes f_\lambda, e_\mu \otimes e_\nu) = \omega(a \diamond e_\nu, e_\mu)= \omega(a, e_\mu \diamond e_\nu) = -\widetilde{\omega}(\Delta_A(a), e_\mu \otimes e_\nu), \\
        &&\widetilde{\omega}(\sum_{\lambda \in \Lambda} e_\lambda \diamond a \otimes f_\lambda, e_\mu \otimes e_\nu) = \omega(e_\nu \diamond a, e_\mu)= \omega(a, e_\mu \diamond e_\nu - e_\nu \diamond e_\mu) = -\widetilde{\omega}(\Delta_A(a) - \hat{\sigma}\Delta_A(a), e_\mu \otimes e_\nu).
    \end{eqnarray*}
    Since $\widetilde{\omega}$ is left nondegenerate, we obtain
    \begin{eqnarray*}
        &&\sum_{\lambda \in \Lambda} a \diamond e_\lambda \otimes f_\lambda = -\sum_{\lambda \in \Lambda} a \diamond f_\lambda \otimes e_\lambda = -\Delta_A(a), \\
        &&\sum_{\lambda \in \Lambda} e_\lambda \diamond a \otimes f_\lambda = - \sum_{\lambda \in \Lambda} f_\lambda \diamond a \otimes e_\lambda = -\Delta_A(a) + \hat{\sigma}(\Delta_A(a)).
    \end{eqnarray*}
    Therefore,
    \begin{eqnarray*}
        &&(\ad(p \otimes a) \hat{\otimes} \id + \id \hat{\otimes} \ad(p \otimes a))(\sum_{\lambda \in \Lambda} \sum_{\alpha} (p_\alpha \otimes e_\lambda)\otimes (q_\alpha \otimes f_\lambda)) \\
        &&=
        (\sum_\alpha p \cdot p_\alpha \otimes q_\alpha) \bullet (\sum_{\lambda \in \Lambda} a \diamond e_\lambda \otimes f_\lambda) - (\sum_\alpha p_\alpha \cdot p \otimes q_\alpha) \bullet (\sum_{\lambda \in \Lambda} e_\lambda \diamond a \otimes f_\lambda) \\
        &&\quad +  (\sum_\alpha p_\alpha \otimes p \cdot q_\alpha) \bullet (\sum_{\lambda \in \Lambda} e_\lambda \otimes a \diamond f_\lambda) - (\sum_\alpha p_\alpha \otimes q_\alpha \cdot p) \bullet (\sum_{\lambda \in \Lambda} e_\lambda \otimes f_\lambda \diamond a)  \\
        &&=
        ( \sum_\alpha -p \cdot p_\alpha \otimes q_\alpha) \bullet \Delta_A(a) +  (\sum_\alpha p_\alpha \cdot p \otimes q_\alpha) \bullet (\Delta_A(a) - \hat{\sigma}(\Delta_A(a))) \\
        &&\quad +  (\sum_\alpha p_\alpha \otimes p \cdot q_\alpha) \bullet \hat{\sigma}(\Delta_A(a)) -  (\sum_\alpha p_\alpha \otimes q_\alpha \cdot p) \bullet (\hat{\sigma
        }\Delta_A(a) - \Delta_A(a) )  \\
        &&= (\sum_\alpha
        (-p \cdot p_\alpha \otimes q_\alpha + p_\alpha \cdot p \otimes q_\alpha + p_\alpha \otimes q_\alpha \cdot p) ) \bullet \Delta_A(a) \\
        &&\quad + (\sum_\alpha \sigma(-q_\alpha \otimes p_\alpha \cdot p + p \cdot q_\alpha \otimes p_\alpha - q_\alpha \cdot p \otimes p_\alpha)) \bullet \hat{\sigma}(\Delta_A(a))  \\
        &&= \Delta_P(p) \bullet \Delta_A(a) + (\sum_\alpha \sigma(-p_\alpha \otimes q_\alpha \cdot p + p \cdot p_\alpha \otimes q_\alpha - p_\alpha \cdot p \otimes q_\alpha) ) \bullet \hat{\sigma}(\Delta_A(a)) \\
        &&= \Delta_P(p) \bullet \Delta_A(a) - \hat{\sigma}(\Delta_P(p) \bullet \Delta_A(a)) = \delta(p \otimes a),
    \end{eqnarray*}
    which completes the proof.
\end{proof}

Combining Propositions~\ref{pro:pp2pb}, \ref{pro:lbc} and
\ref{pro:pybe2cybe}, we have the following construction of a
completed Lie bialgebra from a pre-perm algebra.
\begin{pro}\label{pro:pp2clb}
    Let $(P, \lhd, \rhd)$ be a pre-perm algebra and $(P, \cdot)$ be the sub-adjacent perm algebra of $(P, \lhd, \rhd)$.
    Let $\{e_1, e_2, \cdots, e_n\}$ be a basis of $P$ and $\{e_1^*, e_2^*, \cdots, e_n^*\}$ be the dual basis.
    Let $r = \sum_{i=1}^n (e_i \otimes e_i^* + e_i^* \otimes e_i)$ be a symmetric solution of the perm-YBE in the perm algebra
    $(\widetilde{P} = P \ltimes_{L_\rhd^*, L_\rhd^*-R_\lhd^*}P^*, \cdot)$.
    Let $(A, \diamond, \omega)$ be a quadratic $\mathbb{Z}$-graded pre-Lie algebra, and $(L := \widetilde{P} \otimes A, [-, -])$ be the induced Lie algebra from $(\widetilde{P}, \cdot)$ and $(A, \diamond)$.
    Then $\tilde{r}$ defined by Eq.~\eqref{eq:pybe2cybe} is a skew-symmetric completed solution of the CYBE in $(L, [-, -])$.
    Hence there is a completed Lie bialgebra $(L, [-, -], \delta)$ where $\delta$ is defined by Eq.~\eqref{eq:lbc} through $\tilde{r}$.
\end{pro}

We conclude this section by presenting an example illustrating the
above construction.

\begin{ex}
    Let $(P = \mathbf{k}e, \lhd, \rhd)$ be the one-dimensional pre-perm algebra given by
    \begin{equation*}
        e \lhd e = 0, \;\; e \rhd e = e.
    \end{equation*}
    Let $e^*$ be the dual basis.
    Then the perm algebra $(P \ltimes_{L_\rhd^*, L_\rhd^*-R_\lhd^*}P^*, \cdot)$ is the vector space $\mathbf{k}e \oplus \mathbf{k} e^*$ endowed with
    the following products:
    \begin{equation*}
        e \cdot e = e, \; e \cdot e^* = e^* \cdot e = e^*, \;\; e^* \cdot e^* = 0.
    \end{equation*}
    By Proposition~\ref{pro:pp2pb}, $r = e \otimes e^* + e^* \otimes e$ is a symmetric solution of the perm-YBE in $(P \ltimes_{L_\rhd^*, L_\rhd^*-R_\lhd^*}P^*, \cdot)$.
    Let $(A, \diamond, \omega)$ be the quadratic $\mathbb{Z}$-graded pre-Lie algebra given in Example~\ref{ex:qgpl}.
    Then the induced infinite-dimensional Lie algebra $(L, [-, -])$ from $(P \ltimes_{L_\rhd^*, L_\rhd^*-R_\lhd^*}P^*, \cdot)$ and $(A, \diamond)$ is the vector space spanned by $\{e t^i, e^* t^i, e s^{-i}, e^* s^{-i}: i \in \mathbb{Z}\}$ with the Lie brackets given by
    \begin{eqnarray*}
        \hphantom{}[e t^i, e t^j] = (j-i)et^{i+j-1}, \;\; [e t^i, e^* t^j] = (j-i) e^* t^{i+j-1}, \;\; [e t^i, e s^{-j}] = (i-j-1)e s^{i-j-1}, \\
        \hphantom{}[e t^i, e^* s^{-j}] = (i-j-1)e^* s^{i-j-1}, \;\; [e^* t^i, e s^{-j}] = (i-j-1)e^* s^{i-j-1},\\
        \hphantom{}[e^* t^i, e^* s^{-j}] = [e^* t^i, e^* t^j] = [e s^{-i}, e s^{-j}] = [e s^{-i}, e^* s^{-j}] = [e^* s^{-i}, e^* s^{-j}] = 0.
    \end{eqnarray*}
    By Proposition~\ref{pro:pybe2cybe}, $\tilde{r} := \sum_{i \in \mathbb{Z}}(e t^i \otimes e^* s^{-i} - e s^{-i} \otimes e^* t^{i} + e^* t^i \otimes e s^{-i} - e^* s^{-i} \otimes e t^{i})$ is a skew-symmetric completed solution of the CYBE in $(L, [-, -])$.
    Then by Proposition~\ref{pro:lbc}, there is a completed Lie bialgebra $(L, [-, -], \delta)$ where $\delta$ is given by
    \begin{eqnarray*}
        \delta(e t^i) &=& \sum_{j \in \mathbb{Z}} ( -i e t^{-j} \otimes e^* s^{i+j-1} - i e^* t^{-j} \otimes e s^{i+j-1} + i e^* s^{i+j-1} \otimes e t^{-j} + i e s^{i+j-1} \otimes e^* t^{-j}), \\
        \delta(e s^{-i}) &=& \sum_{j \in \mathbb{Z}}(i+2j-1)( e s^{-j} \otimes e^* s^{i+j-1} + e^* s^{-j} \otimes e s^{i+j-1} ) ,\\
        \delta(e^* t^i) &=&  \sum_{j \in \mathbb{Z}} i( e^* s^{i+j-1} \otimes e^* t^{-j} - e^* t^{-j} \otimes e^* s^{i+j-1}), \\
        \delta(e^* s^{-i}) &=&  \sum_{j \in \mathbb{Z}}(i+2j-1) e^* s^{-j} \otimes e^* s^{i+j-1}, \;\; \forall i \in \mathbb{Z}.
    \end{eqnarray*}
\end{ex}

%%%%%%%%%%%%%%%%%%%%%%%%%%%%%%%%%%%%%%%%%%%%%%%%%%%%%%%%%%%%%%%%%%%%%%%%%%%%%%%%%%%%%%%%%%
\section{Infinite-dimensional Lie bialgebras from pre-Lie bialgebras}\label{sec:plb}
We construct an infinite-dimensional Lie bialgebra on the tensor
product of a pre-Lie bialgebra and a quadratic $\mathbb{Z}$-graded perm algebra, and
show that a pre-Lie bialgebra could be characterized by such a
construction provided a quadratic $\mathbb{Z}$-graded perm algebra
on $\{f_1 \partial_1 + f_2 \partial_2: f_1, f_2 \in
\mathbf{k}[x_1^{\pm}, x_2^{\pm}]\}$. A construction of a Manin
triple of Lie algebras from a para-K\"ahler Lie algebra and a
quadratic perm algebra is obtained. We also present a construction
of a skew-symmetric solution of the CYBE in the corresponding Lie
algebra from a symmetric solution of the $S$-equation in a pre-Lie
algebra.
The results in this section are similar to those for perm algebras
in Section~\ref{sec:ilb}, and hence most of the proofs are omitted.

%%%%%%%%%%%%%%%%%%%%%%%%%%%%%%%%%%%%%%%%%%%%%%%%%%%%%%%%%%%%%%%%%%%%%%%%%%%%%%%%
\subsection{Lie bialgebras from pre-Lie bialgebras and quadratic perm algebras}

\begin{defi}
    A \textbf{quadratic $\mathbb{Z}$-graded perm algebra $(P, \cdot, \kappa)$} is a $\mathbb{Z}$-graded perm algebra $(P = \oplus_{i \in \mathbb{Z}} P_i, \cdot)$ with a skew-symmetric nondegenerate invariant graded bilinear form $\kappa$ on $(P, \cdot)$.
\end{defi}
In particular, when $P = P_0$, it coincides with the notion of a
quadratic perm algebra given in Definition~\ref{defi:qp}.

\begin{ex}\label{ex:qgp}
    Let $(P = \oplus_{i \in \mathbb{Z}} P_i, \cdot)$ be the $\mathbb{Z}$-graded perm algebra given in Example~\ref{ex:zgp}.
    Define a skew-symmetric bilinear form $\kappa$ on $(P, \cdot)$ by
    \begin{eqnarray*}
        &&\kappa(x_1^{i_1}x_2^{i_2} \partial_2, x_1^{j_1}x_2^{j_2} \partial_1) = -\kappa(x_1^{j_1}x_2^{j_2} \partial_1, x_1^{i_1}x_2^{i_2} \partial_2) = \delta_{i_1+j_1, 0}\delta_{i_2+j_2, 0}, \\
        &&\kappa(x_1^{i_1}x_2^{i_2} \partial_1, x_1^{j_1}x_2^{j_2} \partial_1) = \kappa(x_1^{i_1}x_2^{i_2} \partial_2, x_1^{j_1}x_2^{j_2} \partial_2) = 0, \;\; \forall i_1, i_2, j_1, j_2 \in \mathbb{Z}.
    \end{eqnarray*}
    Then $(P, \cdot, \kappa)$ is a quadratic $\mathbb{Z}$-graded perm algebra.
    Moreover, $\{x_1^{-i_1}x_2^{-i_2} \partial_2, -x_1^{-i_1}x_2^{-i_2} \partial_1: i_1,i_2 \in \mathbb{Z}\}$ is the dual basis of $\{x_1^{i_1}x_2^{i_2} \partial_1, x_1^{i_1}x_2^{i_2} \partial_2: i_1,i_2 \in \mathbb{Z}\}$ with respect to $\kappa$, consisting of homogeneous elements.
\end{ex}

\begin{lem}\label{lem:qp2pc}
    Let $(P = \oplus_{i \in \mathbb{Z}} P_i, \cdot, \kappa)$ be a quadratic $\mathbb{Z}$-graded perm algebra.
    Define a linear map $\Delta_P: P \to P \hat{\otimes} P$ by
    \begin{equation}\label{eq:qp2pc}
        \widetilde{\kappa}(\Delta_P(p), q \otimes r) = - \kappa(p, q \cdot r), \;\; \forall p, q, r \in P.
    \end{equation}
    Then $(P, \Delta_P)$ is a completed perm coalgebra.
\end{lem}
\begin{proof}
    The proof is similar to that of Lemma~\ref{lem:qgpl2cplc}.
\end{proof}

\begin{rmk}\label{rmk:qcpa}
    The completed perm coalgebra obtained by Lemma~\ref{lem:qp2pc} from the quadratic $\mathbb{Z}$-graded perm algebra $(P = \oplus_{i \in \mathbb{Z}} P_i, \cdot, \kappa)$ in Example~\ref{ex:qgp} coincides with the one given in Example~\ref{ex:cpc}.
\end{rmk}

The notion of a pre-Lie bialgebra was introduced in
\cite{bai2008left}, whereas an equivalent expression is given as follows.

\begin{defi}
    A \textbf{pre-Lie bialgebra} is a triple $(A, \diamond, \Delta_A)$ where $(A, \diamond)$ is a pre-Lie algebra and $(A, \Delta_A)$ is a pre-Lie coalgebra satisfying
    \begin{equation}\label{eq:plbi}
        \zeta(b, a) = \zeta(a, b) = \sigma(\zeta(a, b)), \;\; \forall a, b \in A,
    \end{equation}
    where $\zeta(a, b)$ is defined by
    \begin{equation*}
        \zeta(a, b) := \Delta_A(a \diamond b) - (L_\diamond(a) \otimes \id + \id \otimes L_\diamond(a))\Delta_A(b) - (\id \otimes R_\diamond(b))\Delta_A(a), \;\; \forall a, b \in A.
    \end{equation*}
\end{defi}

\begin{thm}\label{thm:plb2lb}
    Let $(A, \diamond, \Delta_A)$ be a pre-Lie bialgebra,  $(P = \oplus_{i \in \mathbb{Z}} P_i, \cdot, \kappa)$ be a quadratic $\mathbb{Z}$-graded perm algebra
    and $(L := A \otimes P, [-, -])$ be the induced Lie algebra from $(A, \diamond)$ and $(P, \cdot)$.
    Let
    $\Delta_P: P \to P \hat{\otimes} P$ be the linear map defined by Eq.~\eqref{eq:qp2pc},
    and $\delta: L \to L \hat{\otimes} L$ be the linear map defined by Eq.~\eqref{eq:plc2lc}.
    Then $(L, [-, -], \delta)$ is a completed Lie bialgebra.
    Further, if $(P, \cdot, \kappa)$ is the quadratic $\mathbb{Z}$-graded perm algebra given in Example~\ref{ex:qgp}, then $(L, [-, -], \delta)$ is a completed Lie bialgebra if and only if $(A, \diamond, \Delta_A)$ is a pre-Lie bialgebra.
\end{thm}
\begin{proof}
    The proof of the first part is similar to that of Theorem~\ref{thm:pb2lb}.
    Conversely, suppose that $(P, \cdot, \kappa)$ is the quadratic $\mathbb{Z}$-graded perm algebra given in Example~\ref{ex:qgp} and $(L, [-, -], \delta)$ is a completed Lie bialgebra.
    Then $(A, \diamond)$ is a pre-Lie algebra by Theorem~\ref{thm:pla2l} and $(A, \Delta_A)$ is a pre-Lie coalgebra by Theorem~\ref{thm:plca2lc} and Remark~\ref{rmk:qcpa}.
    Thus it is sufficient to show that Eq.~\eqref{eq:plbi} holds.
    Note that in this case, $\delta$ is defined by Eqs.~\eqref{eq:mpc1} and \eqref{eq:mpc2}.
    For all $a, b \in A$ and $m, n, s, t, i_1, i_2 \in \mathbb{Z}$, comparing the coefficients of $x_1^{i_1+s+1}x_2^{i_2+t} \partial_2 \otimes x_1^{m-i_1+1}x_2^{n-i_2}\partial_1$ in the expansion of
    \begin{eqnarray*}
        \delta([a x_1^m x_2^n \partial_1, b x_1^s x_2^t \partial_1]) &=& (\ad(a x_1^m x_2^n \partial_1) \otimes \id + \id \otimes \ad(a x_1^m x_2^n \partial_1))\delta(b x_1^s x_2^t \partial_1) \\
        &&- (\ad(b x_1^s x_2^t \partial_1) \otimes \id + \id \otimes \ad(b x_1^s x_2^t \partial_1))\delta(a x_1^m x_2^n \partial_1),
    \end{eqnarray*}
    we get $\zeta(a, b) = \zeta(b, a)$.
    Comparing the coefficients of $x_1^{m+s-i_1+1}x_2^{n+t-i_2+1} \partial_2 \otimes x_1^{i_1}x_2^{i_2}\partial_1$ in the expansion of
    \begin{eqnarray*}
        \delta([a x_1^m x_2^n \partial_2, b x_1^s x_2^t \partial_1]) &=& (\ad(a x_1^m x_2^n\partial_2) \otimes \id + \id \otimes \ad(a x_1^m x_2^n\partial_2))\delta(b x_1^s x_2^t \partial_1) \\
        &&- (\ad(b x_1^s x_2^t \partial_1) \otimes \id + \id \otimes \ad(b x_1^s x_2^t \partial_1))\delta(a x_1^m x_2^n \partial_2),
    \end{eqnarray*}
    we get $\zeta(a, b) = \sigma(\zeta(a, b))$.
    Hence, $(A, \diamond, \Delta_A)$ is a pre-Lie bialgebra.
\end{proof}

%%%%%%%%%%%%%%%%%%%%%%%%%%%%%%%%%%%%%%%%%%%%%%%%%%%%%%%%%%%%%%%%%%%
\subsection{Para-K\"ahler Lie algebras and Manin triples of Lie algebras}
In this subsection, we assume that all pre-Lie algebras and Lie
algebras are finite-dimensional.

\begin{defi}
    Let $(L, [-, -])$ be a  Lie algebra.
    If there is a skew-symmetric nondegenerate bilinear form $(\ ,\ )$ on $(L, [-, -])$ such that
    \begin{equation*}
        ([a, b], c) + ([b, c], a) + ([c, a], b) = 0, \;\; \forall a, b, c \in L,
    \end{equation*}
    then $(L, [-, -], (\ ,\ ))$ is called a \textbf{symplectic Lie algebra}.
\end{defi}
A symplectic Lie algebra is also called a \textbf{quasi-Frobenius Lie algebra} (\cite{beidar1997frobenius}).

\begin{defi}
    A symplectic Lie algebra $(\mathfrak{g}, [-, -], \omega)$ is called a \textbf{para-K\"ahler Lie algebra} if $\mathfrak{g} = \mathfrak{g}_1 \oplus \mathfrak{g}_2$ is the direct sum of the underlying vector spaces of two Lie subalgebras such that $\omega(\mathfrak{g}_i, \mathfrak{g}_i) = 0$ for $i=1, 2$,
    which is denoted by $(\mathfrak{g}_1 \bowtie \mathfrak{g}_2, \mathfrak{g}_1, \mathfrak{g}_2, \omega)$.
\end{defi}

Let $(A, \diamond)$ be a pre-Lie algebra. The commutator $[a, b]
:= a \diamond b - b \diamond a$ for all $a, b \in A$ defines a Lie
algebra $(A, [-, -])$, which is called the \textbf{sub-adjacent Lie algebra} of $(A, \diamond)$ and denoted by $(\mathfrak{g}(A),
[-, -])$, and $(A, \diamond)$ is also called a \textbf{compatible pre-Lie algebra structure} on the Lie algebra $(\mathfrak{g}(A), [-, -])$.

\begin{thm}\label{thm:plmmb}
    {\rm (\cite{bai2008left})}
    Let $(A, \diamond)$ be a pre-Lie algebra and $\Delta_A: A \to A \otimes A$ be a linear map.
    Suppose that the linear dual of $\Delta_A$ gives a pre-Lie algebra structure $\circ$ on $A^*$.
    Then $(A, \diamond, \Delta_A)$ is a pre-Lie bialgebra if and only if $(\mathfrak{g}(A) \bowtie \mathfrak{g}(A^*), \mathfrak{g}(A), \mathfrak{g}(A^*), \omega_p)$ is a para-K\"ahler Lie algebra where $\omega_p$ (also see
    Eq.~{\rm (\ref{eq:mtpa})}) is defined by
    \begin{equation}\label{eq:ff}
        \omega_p(a + a^*, b + b^*) := \langle a^*, b\rangle - \langle b^*, a\rangle, \;\; \forall a, b \in A, a^*, b^* \in A^*.
    \end{equation}
\end{thm}
Note that in this case, $(A \oplus A^*, \star, \omega_p)$ is a quadratic pre-Lie algebra where $\star$ is defined by
\begin{equation*}
    (a+a^*) \star (b+b^*) = (a \diamond b - (L_\circ^* - R_\circ^*)(a^*)b + R_\circ^*(b^*)a) + (a^* \circ b^* - (L_\diamond^* - R_\diamond^*)(a)b^* + R_\diamond^*(b)a^*).
\end{equation*}

The following results are similar to
Proposition~\ref{pro:pmt2lmp}, Lemma~\ref{lem:mtpeq} and
Proposition~\ref{pro:pblbmm} respectively.
\begin{pro}\label{pro:plmt2lmp}
    Let $(P, \cdot, \kappa)$ be a quadratic perm algebra.
    Let $(A, \diamond)$ and $(A^*, \circ)$ be pre-Lie algebras,
    and $(A \otimes P, [-, -]_{A \otimes P})$ and $(A^* \otimes P, [-, -]_{A^* \otimes P})$ be the induced Lie algebras.
    If $(\mathfrak{g}(A) \bowtie \mathfrak{g}(A^*), \mathfrak{g}(A), \mathfrak{g}(A^*), \omega_p)$ is a para-K\"ahler Lie algebra, then $((A \oplus A^*)\otimes P, A \otimes P, A^* \otimes P)$ is a Manin triple of Lie algebras associated with the bilinear form
    \begin{equation}\label{eq:plmt2lmp}
        \mathfrak{B}(a_1 \otimes p_1 + a_1^* \otimes p_2 , a_2 \otimes p_3 + a_2^* \otimes p_4) := \langle a_1^*, a_2 \rangle \kappa( p_2, p_3 ) - \langle a_2^*, a_1 \rangle  \kappa(p_1, p_4),
    \end{equation}
    for all $a_1, a_2 \in A, a_1^*, a_2^* \in A^*$ and $p_1, p_2, p_3, p_4 \in A$.
\end{pro}

Let $(P, \cdot, \kappa)$ be a quadratic perm algebra. By the
linear isomorphism $\varphi: P \to P^*$ defined by $\langle
\varphi(p), q\rangle = \kappa(p, q)$ for all $p, q \in P$, a perm
algebra $(P^*, \cdot^\prime)$ is obtained by
\begin{equation*}
    p^* \cdot^\prime q^* := \varphi(\varphi^{-1}(p^*) \cdot \varphi^{-1}(q^*)), \;\; \forall p^*, q^* \in P^*.
\end{equation*}

\begin{lem}\label{lem:mtpleq}
    With the same assumptions in Proposition~\ref{pro:plmt2lmp} and let $(L^* = A^* \otimes P^*, [-, -]_{A^* \otimes P^*})$ be the induced Lie algebra from $(A^*, \circ)$ and $(P^*, \cdot^\prime)$.
    Then $(L \oplus L^*, L = A \otimes P, L^* = A^* \otimes P^*)$ is a Manin triple of Lie algebras associated with the bilinear form defined by Eq.~\eqref{eq:smp}.
    Moreover, this Manin triple is isomorphic to the Manin triple $((A \oplus A^*) \otimes P, A \otimes P, A^* \otimes P)$ associated with the bilinear form defined by Eq.~\eqref{eq:plmt2lmp}.
\end{lem}

\begin{pro}\label{pro:plblbmp}
    Let $(A, \diamond, \Delta_A)$ be a pre-Lie bialgebra and $(P, \cdot, \kappa)$ be a quadratic perm algebra.
    Then the Lie bialgebra $(L, [-, -], \delta)$ obtained in Theorem~\ref{thm:plb2lb} coincides with the one obtained from the Manin triple $(L \oplus L^*, L = A \otimes P, L^* = A^* \otimes P^*)$ of Lie algebras given in Lemma~\ref{lem:mtpleq}.
    That is, we have the following commutative diagram.
    \begin{equation*}
        \xymatrix@C=3cm@R=0.75cm{
            \txt{$(A, \diamond, \Delta_A)$ \\ a pre-Lie bialgerbra} \ar@{<->}[r]^-{Thm.~\ref{thm:plmmb}} \ar[d]^-{Thm.~\ref{thm:plb2lb}} & \txt{$(\mathfrak{g}(A) \bowtie \mathfrak{g}(A^*), \mathfrak{g}(A), \mathfrak{g}(A^*), \omega_p)$ \\ a para-K\"ahler Lie algebra} \ar[d]^-{Prop.~\ref{pro:plmt2lmp},~Lem.~\ref{lem:mtpleq}}\\
            \txt{$(L = A \otimes P, [-, -], \delta)$ \\ a Lie bialgebra} \ar@{<->}[r]^-{Pro.~\ref{pro:lb2lmt}} & \txt{$(L \oplus L^*, L = A \otimes P, L^* = A^* \otimes P^*)$ \\ a Manin triple of Lie algebras}
        }
    \end{equation*}
\end{pro}

%%%%%%%%%%%%%%%%%%%%%%%%%%%%%%%%%%%%%%%%%%%%%%%%%%%%%%%%%%%%%%%%%%%%%%%%%%%%%%%%
\subsection{Infinite-dimensional Lie bialgebras from the \texorpdfstring{$S$}{S}-equation}
Let $(A, \diamond)$ be a pre-Lie algebra. Let $r = \sum_{\alpha
\in I} a_\alpha \otimes b_\alpha \in A \otimes A$, where $I$ is a
finite index set. Set
\begin{eqnarray*}
    &&r_{12} \diamond r_{23} := \sum_{\alpha, \beta} a_\alpha \otimes b_\alpha \diamond a_\beta \otimes b_\beta, \;\; r_{13} \diamond r_{23} := \sum_{\alpha, \beta} a_\alpha \otimes a_\beta \otimes b_\alpha \diamond b_\beta, \\
    &&r_{12} \diamond r_{13} := \sum_{\alpha, \beta} a_\alpha \diamond a_\beta \otimes b_\alpha \otimes b_\beta, \;\; r_{23} \diamond r_{13} := \sum_{\alpha, \beta} a_\beta \otimes a_\alpha \otimes b_\alpha \diamond b_\beta.
\end{eqnarray*}
\begin{defi}
    Let $(A, \diamond)$ be a pre-Lie algebra and $r \in A \otimes A$.
    We call the following equation the \textbf{$S$-equation} in $(A, \diamond)$:
    \begin{equation*}
        {\mathcal S}(r) := -r_{12} \diamond r_{13} + r_{12} \diamond r_{23} + r_{13} \diamond r_{23} - r_{23} \diamond r_{13} = 0.
    \end{equation*}
\end{defi}

\begin{pro}\label{pro:s2plb}
    {\rm (\cite{bai2008left})}
    Let $(A, \diamond)$ be a pre-Lie algebra and $r \in A \otimes A$.
    Define a linear map $\Delta_A: A \to A \otimes A$ by
    \begin{equation}\label{eq:plcob}
        \Delta_A(a) = (L_\diamond(a) \otimes \id + \id \otimes (L_\diamond - R_\diamond)(a))(r), \;\; \forall a \in A.
    \end{equation}
    If $r$ is a symmetric solution of the $S$-equation in $(A, \diamond)$, then $(A, \diamond, \Delta_A)$ is a pre-Lie bialgebra.
\end{pro}

\begin{pro}\label{pro:s2cybe}
    Let $(A, \diamond)$ be a pre-Lie algebra, $(P = \oplus_{i \in \mathbb{Z}} P_i, \cdot, \kappa)$ be a quadratic $\mathbb{Z}$-graded perm algebra and $(L := A \otimes P, [-, -])$ be the induced Lie algebra from $(A, \diamond)$ and $(P, \cdot)$.
    Let $\{e_\lambda\}_{\lambda \in \Lambda}$ be a basis of $P$ consisting of homogeneous elements and $\{f_\lambda\}_{\lambda \in \Lambda}$ be the dual basis with respect to $\kappa$ consisting of homogeneous elements.
    If $r = \sum_{\alpha} a_{\alpha} \otimes b_{\alpha}$ is a symmetric solution of the $S$-equation in $(A, \diamond)$,
    then the tensor element
    \begin{equation}\label{eq:s2cybe}
        \tilde{r} := \sum_{\lambda \in \Lambda} \sum_{\alpha} (a_\alpha \otimes e_\lambda ) \otimes (b_\alpha \otimes f_\lambda) \in L \hat{\otimes} L
    \end{equation}
    is a skew-symmetric completed solution of the CYBE in $(L, [-, -])$.
    Further, if $(P, \cdot, \kappa)$ is the quadratic $\mathbb{Z}$-graded perm algebra given in Example~\ref{ex:qgp},
    then
    \begin{equation*}
        \tilde{r} := \sum_{m, n \in \mathbb{Z}} \sum_{\alpha} (a_\alpha x_1^{m}x_2^{n} \partial_1 \otimes b_\alpha x_1^{-m} x_2^{-n} \partial_2 - a_\alpha x_1^{m}x_2^{n} \partial_2 \otimes b_\alpha x_1^{-m} x_2^{-n} \partial_1)  \in L \hat{\otimes} L
    \end{equation*}
    is a skew-symmetric completed solution of the CYBE in $(L, [-, -])$ if and only if $r$ is a symmetric solution of the $S$-equation in $(A, \diamond)$.
\end{pro}
\begin{proof}
    An argument similar to the one used in the proof of Proposition~\ref{pro:pybe2cybe} shows that
    \begin{eqnarray*}
        &&[\tilde{r}_{12}, \tilde{r}_{13}] + [\tilde{r}_{12}, \tilde{r}_{23}] + [\tilde{r}_{13}, \tilde{r}_{23}] \\
        &&= -({\mathcal S}(r))\bullet (\sum_{\lambda, \eta \in \Lambda} e_\lambda \cdot e_\eta \otimes f_\lambda \otimes f_\eta) + ((\id \otimes \sigma)({\mathcal S}(r)))\bullet \sum_{\lambda, \eta \in \Lambda}(e_\eta \cdot e_\lambda \otimes f_\lambda \otimes f_\eta) = 0.
    \end{eqnarray*}
    Consider the case when $(P, \cdot, \omega)$ is the quadratic $\mathbb{Z}$-graded perm algebra given in Example~\ref{ex:qgp}.
    Note that $\{x_1^{-m}x_2^{-n} \partial_2$, $-x_1^{-m}x_2^{-n} \partial_1: m, n \in \mathbb{Z}\}$ is the dual basis of $\{x_1^{m}x_2^{n} \partial_1, x_1^{m}x_2^{n} \partial_2: m, n \in \mathbb{Z}\}$,
    the ``if'' part is obvious.
    Conversely, suppose that $\tilde{r}$ is a skew-symmetric  completed solution of the CYBE in $(L, [-, -])$.
    Note that
    \begin{eqnarray*}
        0 &=& \hat{\sigma}(\tilde{r}) + \tilde{r} \;\;\;=\;\;\; \sum_{m, n \in \mathbb{Z}}\sum_{\alpha} \biggl(
        b_\alpha x_1^{-m} x_2^{-n} \partial_2 \otimes a_\alpha x_1^{m}x_2^{n} \partial_1 - b_\alpha x_1^{-m} x_2^{-n} \partial_1 \otimes a_\alpha x_1^{m}x_2^{n} \partial_2 \\
        && + a_\alpha x_1^{m}x_2^{n} \partial_1 \otimes b_\alpha x_1^{-m} x_2^{-n} \partial_2 - a_\alpha x_1^{m}x_2^{n} \partial_2 \otimes b_\alpha x_1^{-m} x_2^{-n} \partial_1 \biggr).
    \end{eqnarray*}
    Comparing the coefficients of $x_1^m x_2^n \partial_2 \otimes x_1^{-m}x_2^{-n} \partial_1$, we show that $r$ is symmetric.
    Furthermore, comparing the coefficients of $x_1^{m+s+1} x_2^{n+t} \partial_2 \otimes x_1^{-m} x_2^{-n} \partial_2 \otimes x_1^{-s} x_2^{-t} \partial_1$ in the expansion of
    \begin{equation*}
        0 = [\tilde{r}_{12}, \tilde{r}_{13}] + [\tilde{r}_{12}, \tilde{r}_{23}] + [\tilde{r}_{13}, \tilde{r}_{23}],
    \end{equation*}
    we show that ${\mathcal S}(r) = 0$.
    The proof is completed.
\end{proof}

\begin{cor}\label{cor:se2lb}
    With the same assumptions in Proposition~\ref{pro:s2cybe},
    let $\Delta_A: A \to A \otimes A$ be a linear map defined by Eq.~\eqref{eq:plcob} through $r \in A \otimes A$,
    $\Delta_P: P \to P \hat{\otimes} P$ be the linear map defined by Eq.~\eqref{eq:qp2pc}
    and $\delta: L \to L \hat{\otimes} L$ be a linear map defined by Eq.~\eqref{eq:plc2lc}.
    Then $(A, \diamond, \Delta)$ is a pre-Lie bialgebra by Proposition~\ref{pro:s2plb} and hence $(L, [\cdot, \cdot], \delta)$ is a completed Lie bialgebra by Theorem~\ref{thm:plb2lb}.
    It coincides with the completed Lie bialgebra with $\delta$ defined by Eq.~\eqref{eq:lbc} through $\tilde{r}$ by Proposition~\ref{pro:lbc}, where $\tilde{r}$ is defined by Eq.~\eqref{eq:s2cybe}.
    Hence, we have the following commutative diagram.
    \begin{equation*}
        \xymatrix@C=3cm{
            \txt{symmetric solutions of the $S$-equation} \ar[r]^-{Prop.~\ref{pro:s2plb}} \ar[d]^-{Prop.~\ref{pro:s2cybe}} & \txt{pre-Lie bialgebras} \ar[d]^-{Thm.~\ref{thm:plb2lb}}\\
            \txt{skew-symmetric solutions of the CYBE} \ar[r]^-{Prop.~\ref{pro:lbc}} & \txt{Lie bialgebras}
        }
    \end{equation*}
\end{cor}
\begin{proof}
    The proof is analogous to that of Corollary~\ref{cor:pybe2lb}.
\end{proof}

We would like to point out that there are also the notions of
$\mathcal O$-operators of pre-Lie algebras (\cite{bai2008left})
and L-dendriform algebras (\cite{bai2010some}) playing the similar
roles as pre-perm algebras, and hence similarly as
Proposition~\ref{pro:pp2clb}, there is a construction of a
(completed) Lie bialgebra from an L-dendriform algebra along the
approach given in this section.

%%%%%%%%%%%%%%%%%%%%%%%%%%%%%%%%%%%%%%%%%%%%%%%%%%%%%%%%%%%%%%%%%%%%%%%%%%%%%%%%

\appendix
\newcommand{\loneappendix}{
    \numberwithin{equation}{section}
    \stepcounter{section}
    \renewcommand{\thesection}{A}
}

%%%%%%%%%%%%%%%%%%%%%%%%%%%%%%%%%%%%%%%%%%%%%%%%%%%%%%%%%%%%%%%%%%%%%%%%%%%
\section*{Appendix. Constructions of quadratic \texorpdfstring{$\mathbb{Z}$}{Z}-graded pre-Lie algebras}\label{app:qpl}
\loneappendix %  do not need the numbered letter for appendix
We illustrate the construction of
Example~\ref{ex:qgpl}, which plays important roles in the
affinization of a perm bialgebra given in this paper.

\begin{thm}\label{thm:nqpl}
    $(W_n, \diamond)$ cannot be a quadratic pre-Lie algebra.
\end{thm}
\begin{proof}
    Suppose that $\omega$ is a skew-symmetric invariant bilinear form on $(W_n, \diamond)$.
    For all $i_1, \cdots, i_n, j_1, \cdots, j_n \in \mathbb{Z}$ and $k \in \{1, \cdots, n\}$, we have
    \begin{eqnarray*}
        i_1 \omega(x_1^{i_1} \cdots x_{n}^{i_n}\partial_1, x_1^{j_1} \cdots x_{n}^{j_n} \partial_k)
        &=& \omega(x_1 \partial_1 \diamond x_1^{i_1} \cdots x_{n}^{i_n}\partial_1, x_1^{j_1} \cdots x_{n}^{j_n} \partial_k) \\
        &=& -\omega(x_1^{i_1} \cdots x_{n}^{i_n}\partial_1, x_1 \partial_1 \diamond x_1^{j_1} \cdots x_{n}^{j_n} \partial_k - x_1^{j_1} \cdots x_{n}^{j_n} \partial_k \diamond x_1 \partial_1) \\
        &=& -j_1\omega(x_1^{i_1} \cdots x_{n}^{i_n}\partial_1, x_1^{j_1} \cdots x_{n}^{j_n} \partial_k) + \delta_{k, 1}\omega(x_1^{i_1} \cdots x_{n}^{i_n}\partial_1, x_1^{j_1} \cdots x_{n}^{j_n}\partial_1), \\
        %%%%%%%%%%%%%%%%%%%%%%%%%%%%%%%%%%%%%%%%%%%%%%%%%%%
        j_1 \omega(x_1^{i_1} \cdots x_{n}^{i_n}\partial_1, x_1^{j_1} \cdots x_{n}^{j_n} \partial_k)
        &=& \omega(x_1^{i_1} \cdots x_{n}^{i_n}\partial_1, x_1 \partial_1 \diamond x_1^{j_1} \cdots x_{n}^{j_n} \partial_k) \\
        &=& \omega(x_1^{i_1} \cdots x_{n}^{i_n}\partial_1 \diamond x_1 \partial_1 - x_1 \partial_1 \diamond x_1^{i_1} \cdots x_{n}^{i_n}\partial_1, x_1^{j_1} \cdots x_{n}^{j_n} \partial_k) \\
        &=& (1-i_1)\omega(x_1^{i_1} \cdots x_{n}^{i_n}\partial_1, x_1 \partial_1 \diamond x_1^{j_1} \cdots x_{n}^{j_n} \partial_k).
    \end{eqnarray*}
    Thus, we have the following conclusions.
    \begin{enumerate}
        \item If $k = 1$, then
        \begin{equation}\label{eq:pf13}
            (i_1+j_1-1)\omega(x_1^{i_1} \cdots x_{n}^{i_n}\partial_1, x_1^{j_1} \cdots x_{n}^{j_n} \partial_1) = 0.
        \end{equation}
        \item If $k \neq 1$, then
        \begin{equation*}
            (i_1+j_1)\omega(x_1^{i_1} \cdots x_{n}^{i_n}\partial_1, x_1^{j_1} \cdots x_{n}^{j_n} \partial_k) = (i_1+j_1-1)\omega(x_1^{i_1} \cdots x_{n}^{i_n}\partial_1, x_1^{j_1} \cdots x_{n}^{j_n} \partial_k) = 0.
        \end{equation*}
        Note that $i_1+j_1$ and $i_1+j_1-1$ could not be equal to zero simultaneously, we have
        \begin{equation}
            \omega(x_1^{i_1} \cdots x_{n}^{i_n}\partial_1, x_1^{j_1} \cdots x_{n}^{j_n} \partial_k) = 0, \;\; \forall i_1, \cdots, i_n, j_1, \cdots, j_n \in \mathbb{Z}, k \neq 1. \label{eq:pf5}
        \end{equation}
    \end{enumerate}
    On the other hand, we have
    \begin{eqnarray*}
        i_k \omega(x_1^{i_1} \cdots x_{n}^{i_n}\partial_1, x_1^{j_1} \cdots x_{n}^{j_n} \partial_1)
        &=& \omega(x_k \partial_k \diamond x_1^{i_1} \cdots x_{n}^{i_n}\partial_1, x_1^{j_1} \cdots x_{n}^{j_n} \partial_1)\\
        &=& -\omega(x_1^{i_1} \cdots x_{n}^{i_n}\partial_1, x_k \partial_k \diamond  x_1^{j_1} \cdots x_{n}^{j_n} \partial_1 - x_1^{j_1} \cdots x_{n}^{j_n} \partial_1 \diamond x_k \partial_k) \\
        &=& -j_k\omega(x_1^{i_1} \cdots x_{n}^{i_n}\partial_1, x_1^{j_1} \cdots x_{n}^{j_n} \partial_1)  + \delta_{k,1} \omega(x_1^{i_1} \cdots x_{n}^{i_n}\partial_1, x_1^{j_1} \cdots x_{n}^{j_n} \partial_1).
    \end{eqnarray*}
    Hence, if $j_k \neq 0$ for some $k \neq 1$, then $$\omega(x_1\partial_1, x_1^{j_1} \cdots x_n^{j_n} \partial_1) = 0.$$
    Eq.~\eqref{eq:pf13} shows that if $j_1 \neq 0$, then $\omega(x_1\partial_1, x_1^{j_1} \cdots x_n^{j_n} \partial_1) = 0$.
    Therefore, we have
    \begin{equation*}
        \omega(x_1\partial_1, x_1^{j_1} \cdots x_n^{j_n} \partial_1) = 0, \text{whenever $j_k \neq 0$ for some $k \in \{1, \cdots, n\}$}.
    \end{equation*}
    Finally, consider the case $\omega(x_1 \partial_1, \partial_1)$:
    \begin{eqnarray*}
        \omega(x_1 \partial_1, \partial_1) = \frac{1}{2}\omega( \partial_1 \diamond x_1^2 \partial_1, \partial_1) = \frac{1}{2}\omega(x_1^2\partial_1, \partial_1 \diamond \partial_1 - \partial_1 \diamond \partial_1) = 0.
    \end{eqnarray*}
    We conclude that
    \begin{equation}
        \omega(x_1\partial_1,  x_1^{j_1} \cdots x_n^{j_n} \partial_1) = 0, \;\; \forall j_1, \cdots, j_n \in \mathbb{Z} \label{eq:pf6}
    \end{equation}
    Now combining Eqs.~\eqref{eq:pf5}-\eqref{eq:pf6}, we have $\omega(x_1 \partial_1, W_n) = 0$. Therefore $\omega$ is degenerate.
\end{proof}
Though $(W_n, \diamond)$ cannot be quadratic, we can construct a
quadratic $\mathbb{Z}$-graded pre-Lie algebra related to
$(W_1,\diamond)$. We first recall some facts.

\begin{pro}\label{pro:sl2pl}
    {\rm (\cite{chu1974symplectic})}
    Let $(A, [-, -], \omega)$ be a finite-dimensional symplectic Lie algebra.
    Then there exists a compatible pre-Lie algebra structure $\diamond$ on $A$ given by
    \begin{equation*}
        \omega(a \diamond b, c) = -\omega(b, [a, c]), \;\; \forall a, b, c \in A.
    \end{equation*}
\end{pro}

\begin{pro}\label{pro:pl2sl}
    {\rm (\cite{bai2006further, bai2021introduction})}
    Let $(A, \diamond)$ be a  finite-dimensional pre-Lie algebra.
    Let $(\mathfrak{g}(A) \ltimes_{-L_\diamond^*} A^*, [-, -])$ be the semi-product Lie algebra of $(\mathfrak{g}(A), [-, -])$ and its representation $(A^*, -L_\diamond^*)$.
    Then $(\mathfrak{g}(A) \ltimes_{-L_\diamond^*} A^*, [-, -], \omega_p)$ is a symplectic Lie
    algebra, where $\omega_p$ is defined by Eq.~{\rm (\ref{eq:ff})}.
\end{pro}
Therefore, given a finite-dimensional pre-Lie algebra $(A,
\diamond)$, Proposition~\ref{pro:sl2pl} and \ref{pro:pl2sl} show
that there is a compatible pre-Lie algebra structure $(A \oplus
A^*, \diamond^\prime)$ on the Lie algebra $(\mathfrak{g}(A)
\ltimes_{-L_\diamond^*} A^*, [-, -])$ such that $(A \oplus A^*,
\diamond^\prime, \omega_p)$ is quadratic pre-Lie algebra, where
$\omega_p$ is defined by Eq.~\eqref{eq:ff}. When $(A, \diamond)$ is
a $\mathbb{Z}$-graded pre-Lie algebra, we might consider the
restricted dual of $A$ (\cite{lepowsky2004introduction})  and have
the following conclusion.

\begin{pro}\label{pro:plfd2qpl}
    Let $(A = \oplus_{i \in \mathbb{Z}} A_i, \diamond)$ be a $\mathbb{Z}$-graded pre-Lie algebra with all $A_i$ finite-dimensional.
    Let $A^\circ := \oplus_{i \in \mathbb{Z}} A_{-i}^*$ be the restricted dual of $A$.
    Define linear maps $L_\diamond^*, R_\diamond^*: A \to \End_\mathbf{k}(A^\circ)$ by
    \begin{equation*}
         \langle L_\diamond^*(a)(a^*), b \rangle := \langle a^*, a \diamond b \rangle, \; \langle R_\diamond^*(a)(a^*), b\rangle := \langle a^*, b \diamond a \rangle, \;\; \forall a, b \in A, a^* \in A^\circ.
    \end{equation*}
    Then $(A \oplus A^\circ, \diamond^\prime)$ is a $\mathbb{Z}$-graded pre-Lie algebra with the linear decomposition $A \oplus A^\circ = \oplus_{i\in \mathbb{Z}} (A_i \oplus A_{-i}^*)$,
    where $\diamond^\prime$ is defined by
    \begin{equation*}
        a \diamond^\prime b = a \diamond b, \;
        a \diamond^\prime a^* = -L_\diamond^*(a)a^* + R_\diamond^*(a)a^*, \;
        a^* \diamond^\prime a = R_\diamond^*(a)a^*, \;
        a^* \diamond^\prime b^* = 0, \;\;
        \forall a, b \in A, a^*, b^* \in A^\circ.
    \end{equation*}
    Moreover, $(A \oplus A^\circ = \oplus_{i\in \mathbb{Z}} (A_i \oplus A_{-i}^*), \diamond^\prime, \omega)$ is a quadratic $\mathbb{Z}$-graded pre-Lie algebra,
    where $\omega$ is defined by
    \begin{align*}
        \omega(a + a^*, b + b^*) := \langle a^* , b \rangle - \langle b^*, a\rangle, \forall a, b \in A, a^*, b^* \in A^\circ.
    \end{align*}
    In particular, let $(W_1 = \mathbf{k}[t,t^{-1}], \diamond)$ be the pre-Lie algebra given in Example~\ref{ex:cd2pl} for $n=1$
    and $\mathbf{k}[s,s^{-1}]$ be the restricted dual of $W_1$ with $s^{i}$ being the dual of $t^{-i}$ such that $\langle s^i, t^j\rangle = \delta_{i+j, 0}$ for all $i, j \in \mathbb{Z}$, we obtain Examples~\ref{ex:zgpl} and \ref{ex:qgpl}.
\end{pro}
\begin{proof}
    Note that $(A = \oplus_{i \in \mathbb{Z}} A_i, \diamond)$ is a $\mathbb{Z}$-graded pre-Lie algebra with all $A_i$ finite-dimensional,
    hence $L_\diamond^*$ and $R_\diamond^*$ are well defined.
    Moreover, if $a \in A_i, a^* \in A_{-j}^*$ for $i, j \in \mathbb{Z}$,
    then $L_\diamond^*(a)a^*, R_\diamond^*(a)a^* \in A_{-i-j}^*$.
    Therefore, by a routine verification, we show that $(A \oplus A^\circ, \diamond^\prime)$ is a $\mathbb{Z}$-graded pre-Lie algebra.
    For all $a, b, c \in A$ and $a^*, b^*, c^* \in A^\circ$, we have
    \begin{eqnarray*}
        &&\omega((a + a^*) \diamond^\prime (b+b^*), c+c^*) = \omega(a \diamond b - L_\diamond^*(a)b^* + R_\diamond^*(a)b^* + R_\diamond^*(b)a^*, c + c^*) \\
        &&= -\langle c^*, a \diamond b \rangle - \langle b^*, a \diamond c - c \diamond a \rangle + \langle a^*, c \diamond b \rangle, \\
        &&\omega(b+b^*, (a + a^*) \diamond^\prime (c+c^*) - (c + c^*) \diamond^\prime (a+a^*)) = \omega(b+b^*, a \diamond c - c \diamond a -L_\diamond^*(a)c^* + L_\diamond^*(c)a^*) \\
        &&=\langle b^*, a \diamond c - c \diamond a \rangle + \langle c^*, a \diamond b \rangle  - \langle a^*, c \diamond b \rangle.
    \end{eqnarray*}
    That is, $\omega((a + a^*) \diamond^\prime (b+b^*), c+c^*) = -\omega(b+b^*, (a + a^*) \diamond^\prime (c+c^*) - (c + c^*) \diamond^\prime (a+a^*))$.
    Hence, $(A \oplus A^\circ, \diamond^\prime, \omega)$ is a quadratic $\mathbb{Z}$-graded pre-Lie algebra.
    The particular case follows from a direct verification.
\end{proof}

\noindent{\bf Acknowledgments.} This work is supported by NSFC
(11931009, 12271265, 12261131498, 12326319), Fundamental Research
Funds for the Central Universities and Nankai Zhide Foundation.
The authors thank the referee for valuable suggestions.

\smallskip

\noindent {\bf Declaration of interests. } The authors have no
conflicts of interest to disclose.

\smallskip

\noindent {\bf Data availability. } Data sharing is not applicable
to this article as no new data were created or analyzed in this
study.

\smallskip

%%%%%%%%%%%%%%%%%%%%%%%%%%%%%%%%%%%%%%%%%%%%%%%%%%%%%%%%%%%%%%%%%%%%%%%%%%%%%%%%


\begin{thebibliography}{99}

    \bibitem{aguiar2000pre}
    M.~Aguiar, Pre-Poisson algebras, \textit{Lett. Math. Phys.} 54 (2000), 263--277.

    \bibitem{bai2006further}
    C.~Bai, A further study on non-abelian phase spaces: Left-symmetric algebraic approach and related geometry, \textit{Rev. Math. Phys.} 18 (2006), 545--564.

    \bibitem{bai2007unified}
    C.~Bai, A unified algebraic approach to the classical Yang-Baxter equation, \textit{J. Phys. A: Math. Theor.} 40 (2007), 11073--11082.

    \bibitem{bai2008left}
    C.~Bai, Left-symmetric bialgebras and an analogue of the classical Yang-Baxter equation, \textit{Commun. Contemp. Math.} 10 (2008), 221--260.

    \bibitem{bai2010double}
    C.~Bai, Double constructions of Frobenius algebras, Connes cocycles and their duality, \textit{J. Noncommut. Geom.} 4 (2010), 475--530.

    \bibitem{bai2021introduction}
    C.~Bai, An Introduction to Pre-Lie Algebras, In ``Algebra and Applications 1: Nonssociative Algebras and Categories", Wiley Online Library, 2021, 245--273.



    \bibitem{bai2010some}
    C.~Bai, L.~Liu, X.~Ni, Some results on L-dendriform algebras, \textit{J. Geom. Phys.} 60 (2010), 940--950.


    \bibitem{balinsky1985poisson}
    A.~Balinsky, S.~Novikov, Poisson brackets of hydrodynamic type, Frobenius algebras and Lie algebras, \textit{Sov. Math. Dokl.} 32 (1985), 228--231.

    \bibitem{beidar1997frobenius}
    K.~Beidar, Y.~Fong, A.~Stolin, On Frobenius algebras and the quantum Yang-Baxter equation, \textit{Trans. Amer. Math. Soc.} 349 (1997), 3823--3836.

    \bibitem{burde2006left}
    D.~Burde, Left-symmetric algebras, or pre-Lie algebras in geometry and physics, \textit{Cent. Eur. J. Math.} 4 (2006), 323--357.


    \bibitem{chapoton2001on}
    F.~Chapoton, Un endofoncteur de la cat\'egorie des op\'erades, In ``Dialgebras and Related Operads'', Lecture Notes in Mathematics 1763, Springer, Berlin, Heidelberg, 2001, 105--110.



    \bibitem{chapoton2001pre}
    F.~Chapoton, M.~Livernet, Pre-Lie algebras and the rooted trees operad, \textit{Int. Math. Res. Not.} 8 (2001), 395--408.

    \bibitem{chari1995guide}
    V.~Chari, A.~Pressley, A Guide to Quantum Groups, Cambridge University Press, Cambridge, 1995.

    \bibitem{chu1974symplectic}
    B.~Chu, Symplectic homogeneous spaces, \textit{Trans. Amer. Math. Soc.} 197 (1974), 145--159.

    \bibitem{drinfeld1983hamiltonian}
    V.~Drinfeld, Hamiltonian structure on the Lie groups, Lie bialgebras and the geometric sense of the classical Yang-Baxter equations, \textit{Sov. Math. Dokl.} 27 (1983), 68--71.



    \bibitem{dzhumadildaev2011codimension}
    A.~Dzhumadil'daev, Codimension growth and non-Koszulity of Novikov operad, \textit{Commun. Algebra} 39 (2011), 2943--2952.


    \bibitem{gan2003koszul}
    W.~Gan, Koszul duality for dioperads, \textit{Math. Res. Lett.} 10 (2003), 109--124.



    \bibitem{ginzburg1994koszul}
    V.~Ginzburg, M.~Kapranov, Koszul duality for operads, \textit{Duke Math. J.} 6 (1994), 203--272.

    \bibitem{gnedbaye2017operads}
    A.~Gnedbaye, Operads and triangulation of Loday's diagram on Leibniz algebras, \textit{Afrika Mat.} 28 (2017), 109--118.



    \bibitem{hong2023infinite}
    Y.~Hong, C.~Bai, L.~Guo,  Infinite-dimensional Lie bialgebras via affinization of Novikov bialgebras and Koszul duality, \textit{Commun. Math. Phys.} 401 (2023), 2011--2049.

    \bibitem{hou2023extending}
    B.~Hou, Extending structures for perm algebras and perm bialgebras, \textit{J. Algebra} 649 (2024), 392--432.


    \bibitem{kupershmidt1994non}
    B.~Kupershmidt, Non-abelian phase spaces, \textit{J. Phys. A: Math. Gen.} 27 (1994), 2801--2810.

    \bibitem{kupershmidt1999on}
    B.~Kupershmidt, On the nature of the Virasoro algebra, \textit{J. Nonlinear Math. Phys.} 6 (1999), 222--245.

    \bibitem{kupershmidt1999what}
    B.~Kupershmidt, What a classical $r$-matrix really is, \textit{J. Nonlinear Math. Phys.} 6 (1999), 448--488.

    \bibitem{lepowsky2004introduction}
    J.~Lepowsky, H.~Li, Introduction to Vertex Operator Algebras and Their Representations, Progress in Mathematics 227, Birkh{\"a}user Boston, Inc., Boston, MA, 2004.

    \bibitem{libermann1954on}
    P.~Libermann, Sur le probl\`eme d'\'equivalence de certaines structures infinit\'esimales, \textit{Ann. Mat. Pura Appl.} 36 (1954), 27--120.

    \bibitem{loday2012algebraic}
    J.~Loday, B.~Vallette, Algebraic Operads, Grundlehern Der Mathematischen Wissenschaften 346, Springer, Berlin, Heidelberg, 2012.


    \bibitem{takeuchi1985topological}
    M.~Takeuchi, Topological coalgebras, \textit{J. Algebra} 97 (1985), 505--539.

    \bibitem{vallette2007koszul}
    B.~Vallette, A Koszul duality for props, \textit{Trans. Amer. Math. Soc.} 359 (2007), 4865--4943.



    \bibitem{zhou2023perm}
    P.~Zhou, Perm bialgebras and perm classical Yang-Baxter equation, Master Thesis, Chern Institute of Mathematics, Nankai University, 2023.



\end{thebibliography}
\end{document}